\pdfoutput=1
\documentclass[11pt]{amsart}

\usepackage[T1]{fontenc}
\usepackage[utf8]{inputenc}
\usepackage{amsmath,amssymb,amsthm,mathtools}
\usepackage{stmaryrd}
\usepackage{graphicx}
\usepackage{tikz}
\usetikzlibrary{positioning,calc}
\usepackage{tikz-cd}
\usepackage{xcolor}
\usepackage{etoolbox}
\usepackage{hyperref}
\hypersetup{
  colorlinks = true,
  linkcolor  = blue!60!black,
  citecolor  = blue!60!black,
  urlcolor   = blue!60!black
}
\usepackage{booktabs}
\usepackage{array}
\newcolumntype{L}[1]{>{\raggedright\arraybackslash}p{#1}}

\usepackage[section]{placeins}
\usepackage{cleveref}

\usepackage{newunicodechar}
\newunicodechar{¬}{\ensuremath{\neg}}
\newunicodechar{×}{\ensuremath{\times}}
\newunicodechar{Δ}{\ensuremath{\Delta}}
\newunicodechar{Κ}{\ensuremath{K}}
\newunicodechar{Π}{\ensuremath{\Pi}}
\newunicodechar{Σ}{\ensuremath{\Sigma}}
\newunicodechar{α}{\ensuremath{\alpha}}
\newunicodechar{ε}{\ensuremath{\varepsilon}}
\newunicodechar{ι}{\ensuremath{\iota}}
\newunicodechar{ν}{\ensuremath{\nu}}
\newunicodechar{ω}{\ensuremath{\omega}}
\newunicodechar{ᵀ}{\ensuremath{^{\mathrm{T}}}}
\newunicodechar{⁺}{\ensuremath{^{+}}}
\newunicodechar{₀}{\ensuremath{_{0}}}
\newunicodechar{₁}{\ensuremath{_{1}}}
\newunicodechar{₂}{\ensuremath{_{2}}}
\newunicodechar{₃}{\ensuremath{_{3}}}
\newunicodechar{ₒ}{\ensuremath{_{o}}}
\newunicodechar{𝓤}{\ensuremath{\mathcal{U}}}
\newunicodechar{ℓ}{\ensuremath{\ell}}
\newunicodechar{ℕ}{\ensuremath{\mathbb{N}}}
\newunicodechar{ℤ}{\ensuremath{\mathbb{Z}}}
\newunicodechar{↑}{\ensuremath{\uparrow}}
\newunicodechar{→}{\ensuremath{\rightarrow}}
\newunicodechar{↗}{\ensuremath{\nearrow}}
\newunicodechar{∃}{\ensuremath{\exists}}
\newunicodechar{∈}{\ensuremath{\in}}
\newunicodechar{∑}{\ensuremath{\Sigma}}
\newunicodechar{∘}{\ensuremath{\circ}}
\newunicodechar{∙}{\ensuremath{\bullet}}
\newunicodechar{≃}{\ensuremath{\simeq}}
\newunicodechar{⊣}{\ensuremath{\dashv}}
\newunicodechar{↓}{\ensuremath{\downarrow}}
\newunicodechar{⊲}{\ensuremath{\vartriangleleft}}
\newunicodechar{⊴}{\ensuremath{\trianglelefteq}}
\newunicodechar{⟦}{\ensuremath{\llbracket}}
\newunicodechar{⟧}{\ensuremath{\rrbracket}}
\newunicodechar{𝓚}{\ensuremath{\mathcal{K}}}
\newunicodechar{𝓟}{\ensuremath{\mathcal{P}}}
\newunicodechar{𝓼}{\ensuremath{\mathfrak{s}}}
\newunicodechar{¹}{\ensuremath{^{1}}}
\newunicodechar{Ω}{\ensuremath{\Omega}}
\newunicodechar{η}{\ensuremath{\eta}}
\newunicodechar{𝕋}{\ensuremath{\mathbb{T}}}
\newunicodechar{𝟘}{\ensuremath{\mathbf{0}}}
\newunicodechar{𝟙}{\ensuremath{\mathbf{1}}}
\newunicodechar{𝟚}{\ensuremath{\mathbf{2}}}
\newunicodechar{∞}{\ensuremath{\infty}}
\newunicodechar{∥}{\ensuremath{\Vert}}
\newunicodechar{♯}{\ensuremath{\sharp}}
\newunicodechar{⊥}{\ensuremath{\bot}}
\newunicodechar{⊤}{\ensuremath{\top}}
\newunicodechar{𝕊}{\ensuremath{\mathbb{S}}}
\newunicodechar{𝓢}{\ensuremath{\mathcal{S}}}
\newunicodechar{𝔽}{\ensuremath{\mathbb{F}}}
\newunicodechar{𝔻}{\ensuremath{\mathbb{D}}}
\newunicodechar{𝔹}{\ensuremath{\mathbb{B}}}
\newunicodechar{𝓑}{\ensuremath{\mathcal{B}}}
\newunicodechar{𝓕}{\ensuremath{\mathcal{F}}}
\newunicodechar{≺}{\ensuremath{\prec}}
\newunicodechar{ₛ}{\ensuremath{_{s}}}
\newunicodechar{θ}{\ensuremath{\theta}}
\newunicodechar{τ}{\ensuremath{\tau}}
\newunicodechar{＝}{\ensuremath{=}}
\newunicodechar{⟨}{\ensuremath{\langle}}
\newunicodechar{⟩}{\ensuremath{\rangle}}
\newunicodechar{ᵒ}{\ensuremath{^{o}}}
\newunicodechar{≼}{\ensuremath{\preccurlyeq}}
\newunicodechar{⊏}{\ensuremath{\sqsubset}}
\newunicodechar{⊑}{\ensuremath{\sqsubseteq}}
\newunicodechar{↔}{\ensuremath{\leftrightarrow}}
\newunicodechar{ℚ}{\ensuremath{\mathbb{Q}}}
\newunicodechar{ℝ}{\ensuremath{\mathbb{R}}}
\newunicodechar{⁻}{\ensuremath{^{-}}}
\newunicodechar{³}{\ensuremath{^{3}}}
\newunicodechar{⌜}{\ensuremath{\ulcorner}}
\newunicodechar{⌝}{\ensuremath{\urcorner}}

\newcommand{\UU}{\mathcal{U}}
\newcommand{\VV}{\mathcal{V}}
\newcommand{\WW}{\mathcal{W}}

\newcommand{\Zero}{\mathbf{0}}
\newcommand{\One}{\mathbf{1}}
\newcommand{\Two}{\mathbf{2}}
\newcommand{\pt}{\star}
\newcommand{\N}{\mathbb{N}}
\newcommand{\NInf}{\mathbb{N}_{\infty}}
\newcommand{\sqlt}{\mathrel{\sqsubset}}
\newcommand{\sqleq}{\mathrel{\sqsubseteq}}
\newcommand{\Cantor}{\Two^{\N}}
\newcommand{\Head}{\mathrm{Head}}
\newcommand{\Tail}{\mathrm{Tail}}
\newcommand{\Consinv}{\Cons^{-1}}
\newcommand{\Cons}{\mathrm{Cons}}
\newcommand{\us}[1]{#1}

\newcommand{\LPO}{\mathrm{LPO}}
\newcommand{\WLPO}{\mathrm{WLPO}}

\newcommand{\selection}{\varepsilon}

\newcommand{\apart}{\mathrel{\sharp}}

\newcommand{\tsreflection}[1]{\mathbb{T}\,#1}
\newcommand{\etaT}{\eta^{\mathbb{T}}}

\newcommand{\fib}[2]{#1^{-1}(#2)}

\newcommand{\Pext}{\mathbin{/}}

\newcommand{\ordsub}[1]{\ifblank{#1}{}{_{#1}}}
\newcommand{\Ord}[1][\UU]{\mathrm{Ord}\ordsub{#1}}
\newcommand{\OrdT}[1][\UU]{\mathrm{Ord}^{\top}\ordsub{#1}}
\newcommand{\OrdThree}[1][\UU]{\mathrm{Ord}_{3\ifblank{#1}{}{,#1}}}
\newcommand{\lt}[1]{\mathrel{<_{#1}}}
\newcommand{\gt}[1]{\mathrel{>_{#1}}}
\newcommand{\down}[2]{#1 \mathbin{\downarrow} #2}
\newcommand{\iseq}{\simeq_{o}}

\newcommand{\oemb}{\mathrel{\leq}}
\newcommand{\olt}{\mathrel{<}}
\newcommand{\oembdown}{\mathrel{\rotatebox[origin=c]{-90}{$\leqslant$}}}

\newcommand{\oplusO}{\mathbin{+}}
\newcommand{\otimesO}{\mathbin{\times}}
\newcommand{\oplusT}{\mathbin{+}}

\newcommand{\OneO}{\One}
\newcommand{\OneT}{\One}
\newcommand{\OmegaO}{\Omega}
\newcommand{\TwoO}{\Two}
\newcommand{\TwoT}{\Two}
\newcommand{\ZeroO}{\Zero}
\newcommand{\omegaO}{\omega}
\newcommand{\NInfO}{\NInf}
\newcommand{\succT}[1]{#1 \oplusO \One}

\newcommand{\osum}[2]{\Sigma_{#1} #2}
\newcommand{\summap}[2]{\Sigma(#1,#2)}
\newcommand{\osumthree}[2]{\Sigma_{#1} #2}
\newcommand{\ssumone}[1]{\Sigma^{1} #1}
\newcommand{\ssumsub}[1]{\Sigma_{1} #1}
\newcommand{\ssupone}[1]{\sup\nolimits^{1} #1}
\newcommand{\extend}{\Pext}

\newcommand{\Brouwer}{\mathrm{B}}
\newcommand{\bZ}{\mathsf{Z}}
\newcommand{\bS}{\mathsf{S}}
\newcommand{\bL}{\mathsf{L}}
\newcommand{\stdint}[1]{\llbracket #1 \rrbracket_{\mathrm{sup}}}
\newcommand{\compactsepint}[1]{\llbracket #1 \rrbracket_{\Sigma^{1}}}
\newcommand{\compactint}[1]{\llbracket #1 \rrbracket_{\mathrm{sup}^{1}}}
\newcommand{\trichint}[1]{\llbracket #1 \rrbracket_{\Sigma}}
\newcommand{\alttrichint}[1]{\llbracket #1 \rrbracket_{\Sigma_{1}}}

\newcommand{\Disc}{\Delta}
\newcommand{\Comp}{\mathrm{K}}
\newcommand{\iemb}{\iota}
\newcommand{\cmap}[1][\upsilon]{\mathfrak{c}_{#1}}

\newcommand{\Sier}{\mathbb{S}}
\newcommand{\sier}{\mathfrak{s}}
\newcommand{\botS}{\bot_{\Sier}}
\newcommand{\topS}{\top_{\Sier}}
\newcommand{\dom}{\delta}
\newcommand{\ltS}{\mathrel{<_{\Sier}}}
\newcommand{\extfam}{\bar\alpha}

\newcommand{\Ecodes}{\mathrm{E}}
\newcommand{\eOne}{\ulcorner\One\urcorner}
\newcommand{\eOmegaPlusOne}{\ulcorner\omega{+}\One\urcorner}
\newcommand{\eAdd}{\mathbin{\ulcorner{+}\urcorner}}
\newcommand{\eMul}{\mathbin{\ulcorner{\times}\urcorner}}
\newcommand{\bemb}{\mathfrak{e}}
\newcommand{\eSig}{\ulcorner\Sigma\urcorner}
\newcommand{\limitfn}{\ell}

\newcommand{\id}{\mathrm{id}}
\newcommand{\eps}{\varepsilon_0}

\newcommand{\holds}[1]{\textcolor{green!60!black}{#1}}
\newcommand{\failsc}[1]{\textcolor{red!70!black}{#1}}

\definecolor{agdablue}{HTML}{5E8CB8}

\newif\ifagdarefs
\agdarefsfalse
\newcommand{\AgdaRefs}[1]{\ifagdarefs{\footnotesize\color{agdablue}(#1)}\else\unskip\fi}

\newcommand{\claim}[1]{%
  \ifagdarefs{\footnotesize\color{agdablue}[#1]}\else\unskip\fi
  }

\newcommand{\Agda}[2]{{\footnotesize\color{agdablue}\textsf{#1}%
  \ifblank{#2}{}{\texttt{.}\allowbreak\texttt{#2}}}}

\makeatletter
\let\ttt@plain\texttt
\renewcommand{\texttt}[1]{{\footnotesize\color{agdablue}\ttt@plain{#1}}}
\makeatother

\usepackage{aliascnt}
\newcommand{\sharedtheorem}[3]{%
  \newaliascnt{#1}{theorem}%
  \newtheorem{#1}[#1]{#2}%
  \aliascntresetthe{#1}%
  \crefname{#1}{#2}{#3}%
  \Crefname{#1}{#2}{#3}%
}

\theoremstyle{plain}
\newtheorem{theorem}{Theorem}[section]
\crefname{theorem}{Theorem}{Theorems}
\Crefname{theorem}{Theorem}{Theorems}

\crefname{table}{Table}{Tables}
\Crefname{table}{Table}{Tables}
\crefname{section}{Section}{Sections}
\Crefname{section}{Section}{Sections}
\crefname{subsection}{Section}{Sections}
\Crefname{subsection}{Section}{Sections}
\crefname{subsubsection}{Section}{Sections}
\Crefname{subsubsection}{Section}{Sections}
\sharedtheorem{proposition}{Proposition}{Propositions}
\sharedtheorem{lemma}{Lemma}{Lemmas}
\sharedtheorem{corollary}{Corollary}{Corollaries}

\theoremstyle{definition}
\sharedtheorem{definition}{Definition}{Definitions}
\sharedtheorem{example}{Example}{Examples}
\sharedtheorem{examples}{Examples}{Examples}
\sharedtheorem{remark}{Remark}{Remarks}
\sharedtheorem{notation}{Notation}{Notations}
\sharedtheorem{question}{Question}{Questions}

\title{Compact totally separated types}
\author{Mart\'in H\"otzel Escard\'o}
\date{\today}
\dedicatory{Dedicated to Gordon Plotkin on the occasion of his 80th birthday}

\begin{document}

\begin{abstract}
  Perhaps surprisingly, there are infinite types that can be
  exhaustively searched mechanically in finite time. We use ideas from
  topology to build plenty of them, referring to searchable types as
  \emph{compact types}, and we use ordinals to measure their logical
  complexity. We consider two systems of ordinal notations under which
  a single notation denotes both a discrete ordinal and a compact one,
  with an embedding of the former into the latter whose image has
  empty complement. A boolean valued function decides which points in
  the image of the embedding are isolated and which are topological
  limit points. The first system consists of the
  traditional Brouwer codes and the second is an inductive-recursive
  universe generalizing them. The discrete ordinals so obtained are
  trichotomous, and the compact ones have the least element property
  for complemented subsets, but these two desirable properties cannot
  be fulfilled simultaneously in a constructive setting. The compact ordinals
  obtained from Brouwer codes further enjoy a boolean Leibniz
  principle, which has the notion of total separatedness as its
  topological counterpart. This extends previous work from G\"odel's
  system~$T$ to intensional Martin-L\"of type theory with univalent universes,
  and is formalized in Agda in the
  \href{https://github.com/martinescardo/TypeTopology}{\texttt{TypeTopology}}
  repository.
\end{abstract}

\maketitle

\begingroup
\setcounter{tocdepth}{1}
\makeatletter
\@startsection{}\@M\z@{\linespacing\@plus\linespacing}%
  {.5\linespacing}{\centering\contentsnamefont}{\contentsname}%
\@input{\jobname.toc}
\contentsline{section}{\tocsection{}{}{%
  \hyperref[toc:full]{Full table of contents}}}{\pageref{toc:full}}{}
\makeatother
\endgroup

\newpage

\section{Introduction}
\label{sec:intro}

We consider the following search problem for a type $X$.  Given a
function $p : X \to \Two$ into the type with elements $0$ and $1$, we
want to either exhibit a point $x : X$ with $p\,x = 0$, that is, a
\emph{root} of~$p$, or else determine that~$p$ has no root. This
problem amounts to the type
\[
  \textstyle\prod_{p\,:\,X\to\Two}\,
  \Big(\big(\textstyle\Sigma_{x:X} p\,x = 0\big)
  \;+\;
  \big(\textstyle\prod_{x:X} p\,x = 1\big)\Big).
\]
For a finite type $X$ this can always be done, by trying the finitely
many points in turn. The question investigated in this paper is which
\emph{infinite} types $X$ admit such a procedure.  Three related
subjects view this requirement in different ways.
\begin{enumerate}
\item \emph{Computer science} reads it as a \emph{search problem}. Can
  an infinite type be exhaustively searched mechanically in finite
  time?  Perhaps surprisingly, there are plenty of infinite types that
  satisfy this searchability
  condition~\cite{EscardoSynthetic2004,EscardoFastSearch2007,EscardoExhaustible2008,EscardoOmniscientJSL2013,LongleyNormann2015}.
\item \label{view:topology} In \emph{topology}, this turns out to be a
  \emph{compactness problem}~\cite{EscardoExhaustible2008,LongleyNormann2015}. In
  Johnstone's topological topos~\cite{Johnstone1979}, which we take as
  a guiding model, the types of our theory are interpreted as spaces,
  functions into $\Two$ classify the clopen subspaces, and for a closed
  subtype of a \emph{simple type}, the requirement above holds exactly when
  the corresponding space is compact in the sense of
  topology~\cite{EscardoExhaustible2008,LongleyNormann2015}.
\item In \emph{logic} it amounts to a \emph{choice problem}. Which
  infinite instances of choice can be proved without being taken as
  axioms?
\end{enumerate}
In classical mathematics, excluded middle decides
whether $p$ has a root, and a choice principle produces one when it
does. Asking every type $X$ of a universe $\UU$
to satisfy the above is equivalent to \emph{global choice}, the
principle
\[
  \textstyle\prod_{X:\UU}\,\big(\neg\neg X \to X\big),
\]
which is stronger than the axiom of choice and says that we can pick
an element of every non-empty type~\cite[Theorem~3.2.2 and
Remark~3.2.5]{HoTTBook}.

In this paper we work constructively in an intensional Martin-L\"of
type theory, rather than classically in set theory as in our previous
paper~\cite{EscardoExhaustible2008}. As in that paper, we take
inspiration from topology to develop examples and counterexamples of
such types, as well as their general theory, although we do not assume
previous familiarity with topology, except for the sake of
motivation~\cite{Smyth1983,Smyth1992,EscardoSynthetic2004}. Because of
this and by~(\labelcref{view:topology}) above, we refer to the types
that satisfy the searchability condition as \emph{compact}.

Moreover, we adopt the univalent point of view~\cite{HoTTBook}, due to
its direct way of capturing various notions of extensionality, as well
as the distinction between data and property. Alternatively, in
principle everything discussed here could be redeveloped using
setoids~\cite{Hofmann1995Extensional,Hofmann1997Extensional,pitts2026setoidsintensionaltypetheory}
instead, but we leave such a reworking to interested readers.

The natural numbers fail to be compact constructively. There are
nonetheless many infinite compact types, as mentioned above.  Perhaps
the simplest example is the type of decreasing binary sequences,
\[
  \NInf \;:\equiv\; \textstyle\Sigma_{\alpha\,:\,\Cantor}\,
  \textstyle\prod_{i:\N}\, \alpha\,(i{+}1) \le \alpha\,i ,
\]
which satisfies the universal property of the conatural numbers, and
whose topological reading is the one-point compactification of the
discrete space of natural numbers with a new point at infinity. It is
compact in our type theoretic sense, and we had already shown it to be
so in a system as weak as G\"odel's~$T$~\cite{EscardoOmniscientJSL2013}.

The compactness of $\NInf$ has found applications outside the concerns
of this paper.  Normann and Tait~\cite{NormannTait2017} use its form
in G\"odel's $T$ to fill a gap in an unpublished but widely circulated
manuscript of Tait from 1958 on the computability of the fan
functional. Pradic and Brown~\cite{PradicBrown2019} use it to prove
that the Cantor--Bernstein theorem implies excluded middle.

Compactness as defined above in type theory is uninformative for a
type with too few functions into the type $\Two$ of booleans. The real
numbers are a standard example, where it is consistent that all
functions into the booleans are constant, rendering the search problem
trivial. We say that a type $X$ is \emph{totally separated} when the
functions into $\Two$ separate the points, in the positive sense that
\[
  \textstyle\prod_{x,y:X}\,
  \big(\textstyle\prod_{p\,:\,X\to\Two} p\,x = p\,y\big) \to x = y .
\]
This is a boolean Leibniz principle, whose classical topological
reading is that the clopens separate the points. Our type theoretic
notion of compactness is faithful to the topological notion with the
same name only for types that are totally separated. We may envision
obtaining a general type theoretic notion of compactness by changing
the booleans to a Sierpi\'nski-like type using
dominances~\cite{EscardoKnapp2017} and changing the searchability
condition slightly (cf.~\cite{EscardoExhaustible2008}), but this is
not investigated in the present work.

Totally separated types are sets in the sense of
HoTT/UF~\cite{HoTTBook}, and they are
closed under $\Pi$, $+$, $\times$ and retracts, and include the
simple types and the discrete types (those with decidable equality), but they fail to be closed
under $\Sigma$ constructively in general, and we investigate
particular cases of totally separated $\Sigma$ types of interest.

All compact types we are able to construct turn out to come equipped
with well-orders. Moreover, the constructions that preserve
compactness discussed in the body of the paper also preserve
well-orderability. We adopt the HoTT Book~\cite[Section~10.3]{HoTTBook} definition
of ordinal, namely a type equipped with a proposition valued order
that is transitive, extensional and well-founded. (The HoTT Book also
requires the underlying type to be a set, but we observe that this
follows automatically from the extensionality~\mbox{condition}.)

Trichotomy, namely the principle that $x < y$ or $x = y$ or $x > y$
for all elements $x$ and $y$ of an ordinal, is equivalent to the
principle of excluded middle when required for all ordinals, but there
are many trichotomous ordinals without assuming excluded
middle. Any trichotomous ordinal is discrete,
but e.g.\ in the topological topos, discrete compact objects are
necessarily finite, so that asking for trichotomy together with
compactness gives only finite examples.

Another important classical property of ordinals is that every
non-empty subset has a least element, which is known to fail in
general in a constructive setting. But we do obtain the least element
property constructively for a large class of compact totally separated
ordinals, replacing \emph{subset} by \emph{complemented subset}
(classically, every subset is complemented, and, topologically,
complemented subsets correspond to clopen subsets).

Although the above phenomena exhibit a dichotomy between trichotomy
and compactness, the two notions are related.  In fact, a further
phenomenon becomes visible in a constructive setting. Two ordinals
that are seen to be the same classically can behave in vastly
different ways constructively. An example is the classical ordinal
$\omega + 1$, which has two manifestations (among others). In one, we
equip the type $\N + \One$ with its natural order, which gives rise to
a trichotomous ordinal. In the other, we equip $\NInf$ again with its
natural order, which gives rise to a compact totally separated ordinal
with the least element property for complemented subsets, which fails
to be trichotomous constructively.  In the presence of the principle
of excluded middle, these two ordinals are the same (and $\LPO$
suffices). But, by the above discussion, in the topological topos,
these two ordinals are patently different, one being discrete and the
other compact totally separated (i.e.\ a Stone
space~\cite{johnstone:stonespaces}), and so are they in Hyland's
effective topos. We have the natural inclusion
\[
\begin{array}{rcl}
  \N + \One & \hookrightarrow & \NInf \\
  \mathrm{inl}\,n & \mapsto & 1^n0^\omega \\
  \mathrm{inr}\,\pt & \mapsto & 1^\omega,
\end{array}
\]
which is an injection whose image has empty complement, but it is a
bijection if and only if $\LPO$ holds.  That every decreasing binary
sequence is of one of the forms $1^n0^\omega$ and $1^\omega$ fails
constructively, being equivalent to $\LPO$, while that there is no
decreasing binary sequence other than these always holds
constructively.

We generalize the above situation by working with ordinal codes and
interpreting them in two ways, so that one interpretation gives
trichotomous ordinals and a second one gives compact totally separated
ordinals with the least element property for complemented subsets. We
always have an embedding of the first interpretation into the second,
whose image has empty complement. We also discuss the standard
interpretation, which fails to produce either trichotomous or compact
totally separated ordinals, along with two further ones that arise
naturally, one of which also fails to be totally separated.

We consider two kinds of ordinal codes to obtain the above
results. The first ones are Brouwer codes, inductively defined by
constructors for zero, successor and countable limits. The second ones
are defined together with their interpretations, by
induction-recursion in the sense
of~\cite{Dybjer2000,DybjerSetzer1999,dybjersetzer:IndexedInductionRecursion:2006},
so that a limit constructor is indexed by any previously constructed
ordinal, rather than just~$\omega$, at the cost of total separatedness
of the compact interpretation, which is no longer guaranteed. A crucial ingredient of the
compact totally separated interpretation is a certain \emph{squashed
  sum}, first defined in~\cite[Section~10]{EscardoOmniscientJSL2013}
and generalized here as an \emph{extended sum} for our purposes, with
the aid of the theory of injective types~\cite{EscardoInjective2021,DeJong:Escardo:2026}.

We would like to emphasize an important methodological difference with
our previous
work~\cite{EscardoFastSearch2007,EscardoExhaustible2008}. There we
worked classically using set theory, with externally defined models to
establish our results. Here, instead, we work constructively,
reasoning internally in our type theory, where the beginning of this
transition took place in~\cite{EscardoOmniscientJSL2013}, informally
within constructive mathematics without specifying any particular
formal foundation.

Finally, although the fact that computable or constructively definable
functions are continuous~\cite{LongleyNormann2015,Troelstra1973} plays
a crucial role in the link of topology with computer science and
constructive mathematics, here, as in~\cite{EscardoOmniscientJSL2013},
we \emph{do not} assume Brouwerian continuity
principles~\cite{Troelstra1973,Beeson1985}. In particular, we make no
continuity assumption for the map $p : X \to \Two$ in
the definition of compactness of the type $X$, as opposed to our
previous work~\cite{EscardoFastSearch2007,EscardoExhaustible2008}, so
that our results hold in \emph{all} topos models of (univalent) type
theory~\cite{shulman2019infty1toposesstrictunivalentuniverses} and are
consistent with classical principles, including excluded middle and
choice (but \emph{not} with \emph{global} choice, due to our use of
univalence). So we may consider this work as a form of
\emph{synthetic topology} without topological axioms.

This is a full version of the four-page
abstract~\cite{EscardoTypes2019Abstract}, with further results, and
builds on~\cite{EscardoOmniscientJSL2013}.
All mathematics reported here has been formalized in Agda~\cite{Agda}, in
\textsf{TypeTopology}~\cite{TypeTopology}, and a companion Agda
file~\cite{EscardoCompactCompanion} lists each numbered item of this
paper together with what establishes it.

\subsection*{Organization}

\Cref{sec:foundations}~describes our type theory and
introduces notation and terminology.
\Cref{sec:discrete-tot-sep}~defines and discusses discrete and
totally separated types, covering their closure properties, isolated
points and the totally separated reflection.
\Cref{sec:compact-types-theory}~does the same for compact types, and
discusses the interplay of the three notions, a certain
micro-Tychonoff theorem, and the extension of a family of compact types along
an embedding.
\Cref{sec:ordinals-prelim}~fixes our conventions for ordinals and their
arithmetic, and constructs extended sums and suprema.
\Cref{sec:brouwer-standard}~discusses five interpretations of Brouwer codes
as ordinals satisfying various combinations of compactness,
discreteness, total separation and trichotomy.
\Cref{sec:inductive-recursive}~generalizes Brouwer codes to an
inductive-recursive universe in which the branching of a sum
may be indexed by any previously constructed discrete ordinal, rather
than just $\omegaO$, with two interpretations, one trichotomous and the
other compact, and gives a decision procedure for telling whether a
point in the image of the embedding of the former into the latter is
isolated or a topological limit point.
\Cref{sec:discussion}~summarizes the results, discusses the reach of
these systems and describes the formalization.

\section{Preliminaries}
\label{sec:foundations}

We work constructively and we adopt the foundation described in the
HoTT Book~\cite{HoTTBook}.  In this section, which can be consulted on
demand, we fix terminology and notation, prove some auxiliary facts,
and state basic notions from constructive mathematics, such as
omniscience taboos.

\subsection{Our type theory}

We work in intensional Martin-L\"of type theory with an empty type
$\Zero$, a one-element type $\One$ with point $\pt$, a two-element
type~$\Two$ with points $0,1$, a type $\N$ of natural numbers, the
type formers $+$ (binary sum or coproduct), $\Pi$ (product), $\Sigma$
(sum or disjoint union), $\mathsf{W}$ types, identity types
denoted by the equality sign~$=$, and a hierarchy of universes
$\UU,\VV,\WW,\dots$ closed under them, with $\UU_0$ the first and
$\UU_1$ the next. Binary products, written $X \times Y$, arise both as
general products with index type~$\Two$ and as sums of constant
families, and we choose the latter incarnation by default. We often
write a function type $X \to Y$ in the exponential form $Y^X$, as in
the Cantor type $\Cantor$.

For a family $A : X \to \UU$ and an identification $p : x = y$, we
define $\mathrm{transport}^{A}\,p : A\,x \to A\,y$ by induction on
$p$, taking it to be the identity of $A\,x$ when $p$ is
$\mathrm{refl}$, with the family $A$ omitted when it is understood from the context. We
say that a point $a : A\,x$ is transported along $p : x = y$ to get a
point of the type~$A\,y$.

\begin{definition}[Emptiness, negation and decidability]
\label{def:decidable}
We define
\[
\text{$X$ is empty} \;:\equiv\; (X \to \Zero).
\]
When we think of the type $X$ as a mathematical statement, the
emptiness of $X$ amounts to its negation and is written~$\neg X$.
We say that a type $X$ is \emph{decidable} if we can decide whether it
is pointed or empty:
\[
  \text{$X$ is decidable} \;:\equiv\;  X + \neg X .
\]
We remark that sometimes in the literature the
terminology ``$X$ is decidable'' is used to mean that $X$ has
decidable equality instead. Here we say that $X$ is \emph{discrete} to mean
that it has decidable equality (\Cref{sec:discrete-tot-sep}).
\AgdaRefs{\Agda{MLTT.Negation}{is-decidable}.}
\end{definition}

\begin{definition}[Logical equivalence]
\label{def:logical-equivalence}
Two types $X$ and $Y$ are \emph{logically equivalent} if each maps
into the other:
\[
  \text{$X$ and $Y$ are logically equivalent} \;:\equiv\;
  (X \to Y) \times (Y \to X).
\]
\AgdaRefs{\Agda{Notation.General}{\_↔\_}.}
\end{definition}
When the types $X$ and $Y$ are mathematical statements, we often just
say ``equivalent'' rather than ``logically equivalent'', leaving the
context to resolve the ambiguity with the notion of type
equivalence (\Cref{def:equivalence}).

\subsection{Univalent foundations}
\label{sec:univalent-foundations}

To this type theory we add the univalence
axiom~\cite{HoTTBook,Voevodsky2015} and set quotients. Univalence
gives function extensionality and propositional extensionality, and
set quotients give propositional truncations and set replacement.
Function extensionality says that pointwise equal functions are equal,
where we write $f \sim g$ for the pointwise equality
$\prod_{x:X}\,f\,x = g\,x$ of functions $f, g : \prod_{x:X} A\,x$.
\AgdaRefs{\Agda{MLTT.Id}{\_∼\_}; \Agda{UF.FunExt}{funext},
  \texttt{dfunext}.}
The
Agda development behind this paper~\cite{EscardoCompactCompanion} records
where each of the above assumptions is used, taking them as explicit
hypotheses wherever they are needed. Full univalence in this paper is
assumed from~\Cref{sec:ordinals-prelim} onwards, although not all
results need it, the earlier sections using only its extensionality
consequences, except for the circle
of~\Cref{ex:tot-sep-failures}(\labelcref{item:circle-not-tot-sep}).

A \emph{proposition}, or \emph{subsingleton}, is a type with at most
one element.  We write $\|X\|$ for the propositional truncation of a
type $X$, and $\exists$ and $\vee$ for the truncations of $\Sigma$ and
$+$. A type is \emph{pointed} when it is equipped with a point,
\emph{inhabited} when its truncation is pointed, and \emph{non-empty}
when its emptiness is contradictory, that is, when $\neg\neg X$ holds,
whereas in the constructive mathematics
literature~\cite{bishop_bridges_1985} ``non-empty'' is often used for
inhabitation.
We also read truncation in words, saying that something holds \emph{in
  an unspecified way}, or a similar mode of expression, when what is
asserted is the truncation of the type of data establishing it. For
example a type is compact in an unspecified way when the type of its
compactness data is inhabited. When the disjuncts of a sum are
mutually exclusive propositions, the sum is already a proposition, so
the truncation makes no difference and we write ``or'' for either, as
e.g.\ in the trichotomy of~\Cref{def:trichotomy}.

\begin{definition}[Singletons]
\label{def:singleton}
A type $X$ is a \emph{singleton} or \emph{contractible} if it has a
point $c : X$, its \emph{center of contraction}, to which every point
is equal:
\[
  \text{$X$ is a singleton} \;:\equiv\; \textstyle\Sigma_{c:X}\,\prod_{x:X}\, c = x .
\]
For a family $A : X \to \VV$, there is a unique $x : X$ with $A\,x$
when the sum of $A$ is a singleton:
\[
  \exists!\,x:X,\,A\,x \;:\equiv\;
  \text{the type $\Sigma_{x:X}\,A\,x$ is a singleton}.
\]
A singleton is a proposition, a pointed proposition is a singleton,
and contractibility is a property.
\AgdaRefs{\Agda{UF.Subsingletons}{is-singleton}, \texttt{center};
  \texttt{singletons-are-props},
  \texttt{pointed-props-are-singletons}; \texttt{∃!}.}
\end{definition}

\begin{definition}[Homotopy isolated points and sets]
\label{def:h-isolated}
A point $x : X$ is \emph{homotopy isolated} if the identity type
$x = y$ is a proposition for every $y : X$:
\[
  \text{$x$ is homotopy isolated} \;:\equiv\; \textstyle\prod_{y:X}\,(\text{$x = y$ is a proposition}) .
\]
A type is a \emph{set} if all of its points are homotopy isolated.
\AgdaRefs{\Agda{UF.Sets}{is-h-isolated}, \texttt{is-set}.}
\end{definition}
Call a function \emph{weakly constant} if any two of its values are
equal.
\begin{theorem}[Hedberg~\cite{Hedberg1998}]
\label{thm:hedberg}
If the identity type $x = y$ admits a weakly constant
endofunction for every $y : X$, then $x$ is homotopy isolated. Hence if
this holds at every $x : X$, then $X$ is a set.
\AgdaRefs{\Agda{UF.Hedberg}{local-hedberg},
\texttt{Id-collapsibles-are-sets}; \texttt{wconstant},
\texttt{collapsible} and \texttt{Id-collapsible} for the hypothesis.}
\end{theorem}

\begin{proof}
For each $y : X$ write $f_y$ for the given weakly constant
endofunction of $x = y$, and let $c : x = x$ be $f_x\,\mathrm{refl}$.
We first show that every $p : x = y$ is the composite of $c^{-1}$
followed by $f_y\,p$, by induction on $p$, whose endpoint $y$ varies
with it: for $p = \mathrm{refl}$ this is $c^{-1}$ followed by $c$,
which is $\mathrm{refl}$. Given $p, q : x = y$, both are then
composites of $c^{-1}$ with values of $f_y$, and those two values are
equal by weak constancy, so $p = q$. Thus $x = y$ is a proposition for
every $y : X$, that is, $x$ is homotopy isolated.
\end{proof}

This formulation of the theorem is due to Kraus, Escard\'o, Coquand
and
Altenkirch~\cite{KrausEscardoCoquandAltenkirch2013,KrausEscardoCoquandAltenkirch2017}.

\begin{definition}[The type of propositions]
\label{def:Omega}
We write $\Omega_{\UU}$ for the type of propositions in a universe
$\UU$:
\[
  \Omega_{\UU} \;:\equiv\;
  \textstyle\Sigma_{P : \UU}\,(\text{$P$ is a proposition}) ,
\]
whose two canonical elements are $\bot :\equiv \Zero$ and
$\top :\equiv \One$. Propositional extensionality makes two elements of
$\Omega_{\UU}$ equal exactly when they are logically equivalent, and
$\Omega_{\UU}$ is a set. We omit the universe subscript when it is understood from the context.
\AgdaRefs{\Agda{UF.SubtypeClassifier}{Ω}, \texttt{⊥}, \texttt{⊤},
\texttt{Ω-extensionality}, \texttt{holds-gives-equal-⊤};
\Agda{UF.SubtypeClassifier-Properties}{Ω-is-set}.}
\end{definition}

\begin{definition}[Base and fiber]
\label{def:base-fiber}
For a family $Y : X \to \VV$ we call $X$ the \emph{base} or
\emph{index type} of the sum $\Sigma Y$, that is,
$\Sigma_{x:X}\,Y\,x$, and $Y\,x$ the \emph{fiber} over~$x : X$. The
\emph{fiber} of a map $f : X \to Y$ over a point $y : Y$ is the type
of points that $f$ sends to $y$:
\[
  \fib f y \;:\equiv\; \textstyle\Sigma_{x:X}\, f\,x = y .
\]
The two usages agree up to type equivalence, the fiber of the first
projection $\Sigma_{x:X}\,Y\,x \to X$ over $x$ being equivalent to $Y\,x$.
\AgdaRefs{\Agda{Notation.General}{fiber}.}
\end{definition}

\begin{definition}[Type equivalences]
\label{def:equivalence}
A map $f : X \to Y$ is an \emph{equivalence} if its fibers are all
singletons:
\[
  \text{$f$ is an equivalence} \;:\equiv\;
  \textstyle\prod_{y:Y}\,(\text{$\fib f y$ is a singleton}).
\]
With the above terminology, this says that for every $y : Y$ there is
a unique $x : X$ with $f\,x = y$.  We write $X \simeq Y$ for the type
of equivalences from~$X$ to~$Y$:
\[
  X \simeq Y \;:\equiv\;
  \textstyle\Sigma_{f:X\to Y}\,(\text{$f$ is an equivalence}).
\]
\AgdaRefs{\Agda{UF.Equiv}{is-vv-equiv}.}
\end{definition}

\begin{definition}[Surjections and induction along them]
\label{def:surjection}
A map $f : X \to Y$ is a \emph{surjection} if its fibers are all
inhabited:
\[
  \text{$f : X \to Y$ is a surjection} \;:\equiv\; \textstyle\prod_{y:Y}\, \| \fib f y \| .
\]
This amounts to saying that $\prod_{y:Y}\, \exists_{x : X} \, f\,x = y$.
A surjection supports \emph{surjection induction}: in order to prove a
proposition valued $P : Y \to \VV$ at every $y : Y$, it is enough to
prove $P\,(f\,x)$ at every $x : X$. Conversely, a map supporting this
principle is a surjection.
\AgdaRefs{\Agda{UF.ImageAndSurjection}{is-surjection},
\texttt{Surjection-Induction}, \texttt{surjection-induction},
\texttt{surjection-induction-converse}.}
\end{definition}

\begin{definition}[Embeddings]
\label{def:embedding}
A map $j : X \to Y$ is an \emph{embedding}, written
$j : X \hookrightarrow Y$, if its fibers are all propositions:
\[
  \text{$j : X \to Y$ is an embedding} \;:\equiv\; \textstyle\prod_{y:Y}\, (\text{$\fib jy$ is a proposition}) .
\]
\AgdaRefs{\Agda{UF.Embeddings}{is-embedding}.}
\end{definition}

An embedding is left cancellable, in the sense that $j\,x = j\,x'$
implies $x = x'$, and conversely a left-cancellable map into a
\emph{set} is an embedding, so that the two notions agree for maps
between sets.
\claim{embeddings-and-left-cancellability}
\AgdaRefs{\Agda{UF.Embeddings}{embeddings-are-lc},
  \texttt{lc-maps-into-sets-are-embeddings}.}

We remark that embeddings in this sense are \emph{homotopical}
embeddings rather than \emph{topological} embeddings. For example, the
type of rationals is embedded into the type of Dedekind reals, but not
topologically. In the topological topos the rationals get interpreted
as a discrete space, and the Dedekind reals as the real line with its
usual topology, and the discrete rationals are not a subspace of the
real line. Moreover, the Dedekind reals are not the homotopy type of
the reals, which would be a singleton, but rather a set in the above sense.
\claim{rationals-embed-into-reals}
\AgdaRefs{\Agda{DedekindReals.Order}{ℚ-to-ℝ-is-embedding};
  \Agda{DedekindReals.Type}{ℝ-is-set}.}

\begin{definition}[Complemented families]
\label{def:complemented}
A family $A : X \to \VV$ is \emph{complemented} if $A\,x$ is decidable
for every $x : X$. A complemented family of propositions on $X$ amounts
to a function $X \to \Two$, its characteristic function.
\AgdaRefs{\Agda{NotionsOfDecidability.Complemented}{is-complemented},
  \texttt{characteristic-function}.}
\end{definition}

\begin{definition}[Subsets]
\label{def:subset}
A \emph{subset} of a type $X$ is a map from $X$ to the type of
propositions of~\Cref{def:Omega}:
\[
  S : X \to \Omega_{\UU} .
\]
Equivalently, a subset is a proposition valued map $X \to \UU$. We
write $x \in S$, and say that $x$ is a \emph{member} of $S$, when the
proposition $S\,x$ holds.
\AgdaRefs{\Agda{UF.Powerset-MultiUniverse}{𝓟}, \texttt{\_∈\_},
  \texttt{∈-is-prop}.}
\end{definition}
A proposition valued family $A : X \to \VV$ is equivalently an
embedding $Y \hookrightarrow X$, the family given by its fibers, and
when the family is complemented the embedding is classified by a map
$X \to \Two$, as illustrated by the following pullback diagram:
\[
\begin{tikzcd}[row sep=large, column sep=large]
  Y \arrow[r] \arrow[d, hook, "\text{complemented}"'] & \One \arrow[d, "1"] \\
  X \arrow[r] & \Two.
\end{tikzcd}
\]
In models such as Johnstone's topological topos, complemented subsets
get interpreted as clopens.
\begin{definition}[Image]
\label{def:image}
A point $y : Y$ is in the \emph{image} of a map $f : X \to Y$ when some
$x : X$ is sent to $y$ in an unspecified way, and the \emph{image} of
$f$ is the type of the points of $Y$ that are in it:
\[
  y \in \mathrm{image}\,f \;:\equiv\; \|\fib f y\|,
  \qquad
  \mathrm{image}\,f \;:\equiv\; \textstyle\Sigma_{y:Y}\,(y \in \mathrm{image}\,f) .
\]
The \emph{image} of a subset $S$ of $X$ along $f$ is the subset of $Y$
whose members are the points that some member of $S$ is sent to in an
unspecified way:
\[
  y \in \mathrm{image}\,f\,S \;:\equiv\;
  \exists_{x:X}\,\big(x \in S \times (f\,x = y)\big) .
\]
The \emph{corestriction} of $f$ is the map $X \to \mathrm{image}\,f$
that sends $x$ to $f\,x$ together with a witness that it is in the
image, and it is a surjection.
\AgdaRefs{\Agda{UF.ImageAndSurjection}{image}, \texttt{\_∈image\_},
  \texttt{corestriction}, \texttt{corestrictions-are-surjections};
  \Agda{UF.Powerset-MultiUniverse}{𝓟-image} for the image of a subset.}
\end{definition}

\begin{definition}[Small and locally small types]
\label{def:small}
A type $X$ is \emph{$\UU$-small} if it is equivalent to a type in the
universe $\UU$, and it is \emph{locally $\UU$-small} if the identity
type $x = y$ is $\UU$-small for every $x, y : X$.
\AgdaRefs{\Agda{UF.Size}{is-small}, \texttt{is-locally-small}.}
\end{definition}
Being small is a proposition if and only if univalence holds.
\claim{smallness-and-univalence}
\AgdaRefs{\Agda{UF.Size}{being-small-is-prop} for the direction that
  univalence gives, and \Agda{UF.Yoneda}{univalence-via-singletons←}
  for the converse, with smallness in the same universe.}
\begin{definition}[Set replacement]
\label{def:set-replacement}
Set replacement is the principle that, for a map $f : X \to Y$, if $X$
is $\UU$-small and $Y$ is a locally $\VV$-small set, then the image
of~$f$ (\Cref{def:image}) is small in the least universe containing
both $\UU$ and $\VV$.
\AgdaRefs{\Agda{UF.Size}{Set-Replacement}.}
\end{definition}
This is the particular case for sets of a more general replacement
principle first identified and studied by
Rijke~\cite{rijke2017joinconstruction,rijke_2025}.
\begin{definition}[Set quotients]
\label{def:set-quotients}
Set quotients exist when, for every type $X : \UU$ and every
$\VV$-valued equivalence relation $\approx$ on $X$, there is a set
$X/\approx$, the \emph{set quotient} of $X$ by $\approx$, with a map
$\eta : X \to X/\approx$ that identifies related
points and is universal: every map $f : X \to Y$ into a set $Y$ that
identifies related points factors as $f = \bar f \circ \eta$ for a
unique $\bar f : {X/ \approx} \to Y$.
\AgdaRefs{\Agda{Quotient.Type}{set-quotients-exist}.}
\end{definition}

Set replacement is equivalent to the existence of set quotients.
\claim{set-replacement-and-quotients}
\AgdaRefs{\Agda{Quotient.GivesSetReplacement}{set-replacement-from-set-quotients-and-funext};
\Agda{Quotient.FromSetReplacement}{set-quotients-from-set-replacement}.}

\subsection{Constructive mathematics}
\label{sec:constructive-mathematics}
Although the general principle of double negation elimination is
equivalent to excluded middle, it does hold constructively in many
particular examples of interest.
\begin{definition}[$\neg\neg$-separatedness]
\label{def:nn-separated}
A type $X$ is \emph{$\neg\neg$-separated} if identifications in it are
$\neg\neg$-stable:
\[
  \text{$X$ is $\neg\neg$-separated} \;:\equiv\; \textstyle\prod_{x,y:X}\, \neg\neg(x = y) \to x = y .
\]
A type with decidable equality is $\neg\neg$-separated, because decidable
types are $\neg\neg$-stable.
\AgdaRefs{\Agda{UF.DiscreteAndSeparated}{is-¬¬-separated};
\texttt{discrete-is-¬¬-separated}.}
\end{definition}

\begin{lemma}
\label{lem:nn-separated-is-set}
Every $\neg\neg$-separated type is a set.
\AgdaRefs{\Agda{UF.NotNotStablePropositions}{¬¬-separated-types-are-sets},
via \texttt{¬¬-separated-types-are-Id-collapsible}.}
\end{lemma}

\begin{proof}
  Compose the natural map $x = y \to \neg\neg(x = y)$ with the given
  $\neg\neg(x = y) \to x = y$, obtaining an endofunction of $x =
  y$. It is weakly constant, because $\neg\neg(x = y)$ is a negation
  and hence a proposition, so that any two intermediate values are
  equal. The result then follows by~\Cref{thm:hedberg}.
\end{proof}

\begin{definition}[Apartness and tightness]
\label{def:apartness}
An \emph{apartness} on a type $X$ is a proposition valued relation
$\apart$ which is
\begin{enumerate}
\item \emph{irreflexive}: $\neg(x \apart x)$,
\item \emph{symmetric}: $x \apart y \to y \apart x$, and
\item \emph{cotransitive}: $x \apart y \to (x \apart z) \vee
  (y \apart z)$.
\end{enumerate}
It is \emph{tight} if $\neg(x \apart y)$ implies $x = y$, for all
$x,y : X$.
\AgdaRefs{\Agda{Apartness.Definition}{is-apartness}, \texttt{is-tight},
\texttt{is-prop-valued}, \texttt{is-irreflexive}, \texttt{is-symmetric},
\texttt{is-cotransitive}; \texttt{is-strong-apartness} for the variant
whose cotransitivity is untruncated.}
\end{definition}

An apartness is the positive counterpart of the negation of equality,
in that $x \apart y$ asserts a difference where $\neg(x = y)$ only
denies equality~\cite{BridgesVita2006}.  Irreflexivity makes an
apartness imply negation of equality, and hence makes equality imply
the negation of apartness. Tightness asks for the converse of the
latter. It follows that a type carrying a tight apartness is
$\neg\neg$-separated, and hence a set
by~\Cref{lem:nn-separated-is-set}.  \claim{tight-apartness-gives-set}
\AgdaRefs{\Agda{Apartness.Definition}{apartness-is-irreflexive'},
  \texttt{tight-types-are-¬¬-separated'},
  \texttt{tight-types-are-sets'}.}

\begin{definition}[Extreme density]
\label{def:dense}
For a map $f : X \to Y$ the following are equivalent.
\begin{enumerate}
\item\label{item:dense-image} The complement of the image of $f$ is
  empty:
\[
  \neg\,\Sigma_{y:Y}\,\neg\,(y \in \mathrm{image}\,f).
\]
\item\label{item:dense-fiber} No point of $Y$ lies outside the image
  of $f$:
\[
  \neg\,\Sigma_{y:Y}\,\neg\,\Sigma_{x:X}\, f\,x = y.
\]
\end{enumerate}
A map satisfying either, and hence both, of these conditions is called
\emph{extremely dense}, and extreme density is a proposition. An
alternative terminology, not adopted here, is \emph{$\neg\neg$-density}.
\AgdaRefs{\Agda{TypeTopology.Density}{is-dense}, which
  is~(\labelcref{item:dense-fiber}),
  \texttt{density-characterization},
  \texttt{density-characterization-≃}, \texttt{being-dense-is-prop};
  \Agda{UF.ImageAndSurjection}{complement-of-image},
  \texttt{\_∈image\_} for~(\labelcref{item:dense-image}).}
\end{definition}

\begin{proof}
For any type $A$ the negations $\neg\|A\|$ and $\neg A$ are logically
equivalent, so the truncation in $y \in \mathrm{image}\,f$ can be
dropped under the inner negation. Extreme density is a proposition,
being a negation.
\end{proof}

In the presence of the principle of excluded middle, a map is
extremely dense if and only if it is a surjection.
\claim{em-dense-surjection}
\AgdaRefs{\Agda{TypeTopology.DenseMapsProperties}{surjections-are-dense}
  for the direction that needs no assumption, and the entry
  \texttt{Prose-em-dense-surjection'} of the companion
  file~\cite{EscardoCompactCompanion} for the converse.}  Extreme density is
stronger than density in the topological sense. The rationals are
dense in the reals in the usual sense, but the inclusion of the
rationals into the Dedekind reals is not extremely dense, since
irrational numbers lie outside its image.

\begin{lemma}
\label{lem:dense-into-nn-separated-rc}
An extremely dense map is right-cancellable with respect to maps into a
family of $\neg\neg$-separated types: if $j : X \to Y$ is extremely
dense, if
$Z : Y \to \UU$ has $Z\,y$ $\neg\neg$-separated for every $y : Y$,
and if $f, g : \prod_{y:Y} Z\,y$ satisfy $f\,(j\,x) = g\,(j\,x)$
for every $x : X$, then $f = g$.
\AgdaRefs{\Agda{TypeTopology.Density}{dense-maps-into-¬¬-separated-types-are-rc'};
\texttt{dense-maps-into-¬¬-separated-types-are-rc} for the constant
family.}
\end{lemma}
In particular, taking $Z$ constant, two maps $Y \to Z$ into a
$\neg\neg$-separated type that agree after composition with $j$ are
equal.
\begin{proof}
For $y : Y$ it is enough to prove $\neg\neg(f\,y = g\,y)$,
$\neg\neg$-separatedness of $Z\,y$ supplying $f\,y = g\,y$ and
function extensionality supplying $f = g$. If $f\,y \neq g\,y$ then
no $x : X$ satisfies $j\,x = y$, since an identification
$p : j\,x = y$ transports the hypothesis $f\,(j\,x) = g\,(j\,x)$
to $f\,y = g\,y$, so $y$ lies outside the image of $j$, contradicting
extreme density.
\end{proof}

\begin{lemma}
\label{lem:two-to-Omega-dense}
The embedding $\Two \to \Omega_{\UU}$ that sends $0$ to $\bot$ and $1$
to $\top$ is extremely dense.
\AgdaRefs{\Agda{UF.SubtypeClassifier-Properties}{𝟚-to-Ω},
  \texttt{𝟚-to-Ω-is-embedding}, \texttt{𝟚-to-Ω-fiber};
  \Agda{UF.SubtypeClassifier}{no-truth-values-other-than-⊥-or-⊤}.}
\end{lemma}

\begin{proof}
The map is an embedding because its fiber over $P$ is $\neg P + P$,
a sum of two propositions that exclude each other. Let
$P : \Omega_{\UU}$ lie outside the image, so that $P \neq \bot$ and
$P \neq \top$, these being the two values of the map. Logically
equivalent propositions are equal, so $\neg P$ would give
$P = \bot$, and $P$ would give $P = \top$. Hence $\neg\neg P$ and
$\neg P$, which is a contradiction. So no point of $\Omega_{\UU}$
lies outside the image.
\end{proof}

The fiber of this map over $P$ is the decidability of $P$, so its
extreme density says that no proposition can fail to be decidable,
whereas excluded middle would say that every proposition is decidable.

\begin{definition}[Global choice]
\label{def:global-choice}
\emph{Global choice} for a universe $\UU$ is the principle that every
type in $\UU$ is decidable, so that given any type $X : \UU$ we can
either produce a point of $X$ or establish that $X$ is empty, that
is,
\[
  \textstyle\prod_{X:\UU}\,(X + \neg X).
\]
\AgdaRefs{\Agda{UF.ClassicalLogic}{Global-Choice}.}
\end{definition}

Equivalently, global choice says that we can choose a point of every
non-empty type in $\UU$,
\[
  \textstyle\prod_{X:\UU}\,(\neg\neg X \to X),
\]
since a decidable type is $\neg\neg$-stable and, conversely, the type
$X + \neg X$ is non-empty for every $X$.
\claim{global-choice-reformulation}
\AgdaRefs{\Agda{UF.ClassicalLogic}{Global-Choice'};
  \texttt{Global-Choice-gives-Global-Choice'};
  \texttt{Global-Choice'-gives-Global-Choice}.}  It is a classical
principle, stronger than the axiom of choice, and we do not assume it,
but it arises as a so-called \emph{constructive taboo} in the
conclusion of some statements.

\begin{definition}[Excluded middle]
\label{def:excluded-middle}
Excluded middle for a universe $\UU$ says that every proposition in
$\UU$ is decidable:
\[
  \textstyle\prod_{P:\UU}\,
  \big(\text{$P$ is a proposition} \to P + \neg P\big).
\]
\AgdaRefs{\Agda{UF.ClassicalLogic}{EM}.}
\end{definition}

The following taboos also arise in some of the mathematical questions
considered in this paper.
\begin{definition}[Bishop's principles of omniscience~\cite{Bishop1967}]
\label{def:lpo-wlpo}
Bishop's $\LPO$ (limited principle of omniscience) decides whether a
binary sequence takes the value~$0$ somewhere, and $\WLPO$ (weak
limited principle of omniscience) decides whether it has
value~$1$ everywhere:
\begin{align*}
  \LPO &:\equiv \textstyle\prod_{p:\Cantor}\,
  \big(\text{the type $\Sigma_{n:\N}\, p\,n = 0$ is decidable}\big),
  \\
  \WLPO &:\equiv \textstyle\prod_{p:\Cantor}\,
  \big(\text{the type $\prod_{n:\N}\, p\,n = 1$ is decidable}\big).
\end{align*}
\AgdaRefs{\Agda{Taboos.LPO}{LPO}, \texttt{LPO'};
\Agda{Taboos.WLPO}{WLPO}, \texttt{WLPO-traditional},
\texttt{WLPO-gives-WLPO-traditional},
\texttt{WLPO-traditional-gives-WLPO}.}
\end{definition}
\noindent
In the vocabulary of this paper, $\Sigma$-compactness of $\N$ is $\LPO$ and
$\Pi$-compactness of $\N$ is $\WLPO$ (\Cref{ex:compact-basic}).
\claim{compactness-of-N-is-LPO}
\AgdaRefs{\Agda{Taboos.LPO}{LPO-gives-compact-ℕ},
  \texttt{compact-ℕ-gives-LPO}.}

Excluded middle gives $\LPO$, which gives $\WLPO$. The converse of the
latter follows from Markov's principle (namely that if a complemented family
over $\N$ has non-empty sum then its sum is pointed).  All of these
are independent of our type theory, Markov's principle
by~\cite{CoquandMannaa2017}. Excluded middle, $\LPO$ and $\WLPO$ hold
under classical logic and fail in models where all functions are
continuous or computable, such as Johnstone's topological topos~\cite{Johnstone1979} or Hyland's
effective topos~\cite{Hyland1982} respectively. We neither assume nor
deny them. When we say that something fails constructively, we mean
that it implies a taboo.  Saying instead that it is false would be an
implication into the empty type, which is a stronger statement. The
same applies to the doubly negated forms, such as $\neg\neg\WLPO$,
which are weaker but still independent.  Global choice
(\Cref{def:global-choice}) is stronger than all of them, since it
decides every type.  \claim{EM-LPO-WLPO}
\AgdaRefs{\Agda{Taboos.LPO}{EM-gives-LPO}, \texttt{LPO-gives-WLPO};
  \Agda{Taboos.MarkovsPrinciple}{MP}, \texttt{MP-and-WLPO-give-LPO}.}

We do not assume any Brouwerian continuity principle. Assuming continuity
would settle the above taboos negatively, since it refutes $\WLPO$ and
hence $\LPO$, whereas we keep them undecided in order to stay
compatible with classical mathematics, so that our results hold in
every model of our foundation, in particular
$\infty$-toposes~\cite{shulman2019infty1toposesstrictunivalentuniverses},
the boolean ones included.
\begin{remark}[Weak topological toposes]
  \label{rem:weak-topological-topos}
  The statement $\neg \WLPO$ is independent of our type theory and can
  be seen as a weak continuity
  principle~\cite{EscardoContinuity2015}. We have that $\WLPO$ is
  equivalent to the existence of a function $\NInf \to \N$ that is not
  continuous, and $\neg\WLPO$ to every function $\NInf \to \N$ being
  $\neg\neg$-continuous, the double negation being removable under
  Markov's principle (which we are not assuming). We think of toposes
  in which $\neg\WLPO$ holds as \emph{weak topological toposes}.
  \claim{noncontinuity-and-WLPO}
  \AgdaRefs{\Agda{TypeTopology.DecidabilityOfNonContinuity}{continuous},
    \texttt{\_is-modulus-of-continuity-of\_},
    \texttt{noncontinuous-map-gives-WLPO},
    \texttt{WLPO-gives-noncontinous-map},
    \texttt{¬WLPO-iff-all-maps-are-¬¬-continuous},
    \texttt{MP-and-¬WLPO-give-that-all-functions-are-continuous}.}
\end{remark}
This remark is the inspiration behind~\Cref{def:limit-point}.

\subsection{Retracts and maps of sums}

\begin{definition}[Retracts]
\label{def:retract}
A type $Y$ is a \emph{retract} of a type $X$ if there are specified maps
$r : X \to Y$ and $s : Y \to X$ with $r \circ s$ the identity. We call
$r$ a \emph{retraction} and $s$ a \emph{section}.
\AgdaRefs{\Agda{UF.Retracts}{retract\_of\_}, \texttt{retraction},
\texttt{section}, \texttt{has-section}, \texttt{is-section}.}
\end{definition}

\begin{definition}[Maps of sums]
\label{def:sum-map}
Let $A : X \to \UU$ and $B : Y \to \UU$ be families of types, let
$f : X \to Y$ be a map, and let $g\,x : A\,x \to B\,(f\,x)$ be a map
for every $x : X$. The map of the sums induced by $f$ and $g$ applies
$f$ to the index and $g$ to the fiber component:
\[
  \summap fg : \Sigma A \to \Sigma B, \qquad
  \summap fg\,(x,a) \;=\; (f\,x,\, g\,x\,a).
\]
\AgdaRefs{\Agda{UF.PairFun}{pair-fun}.}
\end{definition}

\begin{lemma}
\label{lem:sum-map}
The map $\summap fg$ is
\begin{enumerate}
\item\label{item:sum-map-equiv} an equivalence if $f$ and every $g\,x$ are
  equivalences,
\item\label{item:sum-map-embedding} an embedding if $f$ and every
  $g\,x$ are embeddings,
\item\label{item:sum-map-dense} extremely dense if $f$ is and every
  $g\,x$ is,
\item\label{item:sum-map-retraction} a retraction if $f$ is the
  identity and every $g\,x$ is.
\end{enumerate}
\AgdaRefs{\Agda{UF.PairFun}{pair-fun-is-equiv}, via
  \texttt{pair-fun-is-vv-equiv}; \texttt{pair-fun-is-embedding};
  \texttt{pair-fun-dense}; \Agda{UF.Retracts}{Σ-retract} for the
  last, which is \Agda{UF.PairFun}{pair-fun} with the identity on the
  index.}
\end{lemma}

\begin{proof}
The fiber of $\summap fg$ over a point $(y,b)$ of $\Sigma B$ is the
sum, over the points $(x,p)$ of the fiber of $f$ over $y$, of the fiber
of $g\,x$ over the point of $B\,(f\,x)$ obtained from $b$ by transport
along $p$ backwards. If the fibers of $f$ and of every $g\,x$ are singletons,
then so are these, a sum of singleton types over a singleton type
being a singleton, which gives
(\labelcref{item:sum-map-equiv}). If they are propositions, then so are
these, a sum of propositions over a proposition being a proposition,
which gives (\labelcref{item:sum-map-embedding}).

For (\labelcref{item:sum-map-dense}), let $(y,b)$ be a point of
$\Sigma B$ and assume that no point of $\Sigma A$ is sent to it. Since
we are deriving a contradiction, we may drop double negations, so
extreme density of $f$ supplies $x : X$ with $p : f\,x = y$. Transporting $b$ along $p$ backwards gives a
point of $B\,(f\,x)$, and extreme density of $g\,x$ supplies
$a : A\,x$ mapped to it, so that $\summap fg\,(x,a) = (y,b)$, which is
the required contradiction.

For (\labelcref{item:sum-map-retraction}), let $s\,x$ be a section of
$g\,x$ for every $x$. Then $\summap{\mathrm{id}}{s}$ is a section of
$\summap{\mathrm{id}}{g}$, the two composing to
$\summap{\mathrm{id}}{g \circ s}$, which is the identity.
\end{proof}

\section{Totally separated types}
\label{sec:discrete-tot-sep}

This section studies the notions of discrete type, isolated point,
and totally separated type.
\AgdaRefs{\Agda{UF.DiscreteAndSeparated}{},
  \Agda{TypeTopology.SigmaDiscrete}{},
  \Agda{TypeTopology.TotallySeparated}{}}
Compactness, which interacts with all of them, is discussed
in~\Cref{sec:compact-types-theory}.

The following type, already mentioned in~\Cref{sec:intro}, plays a
major role in this paper and also illustrates many concepts.
\begin{example}
\label{ex:NInf}
The type $\NInf$ of conaturals, equivalently the final coalgebra of
$X \mapsto \One + X$, is the type of decreasing binary sequences,
\[
  \NInf \;:\equiv\; \Sigma_{\alpha\,:\,\Cantor}\,
  \textstyle\prod_{i:\N}\, \alpha\,(i{+}1) \le \alpha\,i .
\]
\AgdaRefs{\Agda{CoNaturals.UniversalProperty}{PRED-is-the-homotopy-final-coalgebra}.}
\AgdaRefs{\Agda{TypeTopology.GenericConvergentSequence}{ℕ∞},
\texttt{is-decreasing}.}%
The decreasing condition is a proposition, so the first projection
embeds $\NInf$ into the Cantor type $\Cantor$, and we write $u_i$ for
the $i$-th digit of $u : \NInf$. This embedding is a section of the
map $\beta \mapsto \big(i \mapsto \min\{\beta\,0,\dots,\beta\,i\}\big)$
that forces a binary sequence to be decreasing and leaves the decreasing
sequences unchanged, so $\NInf$ is a retract of the Cantor type.
\AgdaRefs{\Agda{TypeTopology.GenericConvergentSequence}{ℕ∞-retract-of-Cantor},
via \texttt{force-decreasing}, \texttt{force-decreasing-unchanged}.}
Write
$\iota : \N \hookrightarrow \NInf$ for the inclusion and $\underline k
= \iota\,k$ for its value at $k$, the sequence of $k$ ones
followed by zeros,
\[
  \underline k \;=\;
  \big(i \mapsto \text{$1$ if $i < k$, and $0$ otherwise}\big),
\]
so that $\underline 0$ is constantly~$0$ and $\underline{k{+}1}$
prefixes a~$1$ to~$\underline k$. We let $\infty : \NInf$ be the
constantly~$1$ sequence. For $n : \N$ and $u : \NInf$ we write $n \sqlt
u$ to say that the $n$-th digit of $u$ is $1$, and $u \sqleq n$ to say
that it is $0$, so that the two are negations of each other. We say that
$u$ is \emph{finite} when $u = \underline k$ for some $k$, equivalently
when $u \sqleq i$ for some $i$. The principles $\LPO$ and $\WLPO$
of~\Cref{def:lpo-wlpo} are equivalent to deciding whether a point of
$\NInf$ is finite and whether it is~$\infty$, respectively. We use
whichever formulation is convenient.
\AgdaRefs{\Agda{TypeTopology.GenericConvergentSequence}{\_⊏\_},
\texttt{\_⊑\_}; \texttt{ℕ-to-ℕ∞},
defined by $\underline 0 = \texttt{Zero}$ and
$\underline{k{+}1} = \texttt{Succ}\,\underline k$. The Agda writes this
map as \texttt{ι}, the overloaded canonical-map notation of
\Agda{Notation.CanonicalMap}{}, through the instance
\texttt{Canonical-Map-ℕ-ℕ∞}.}
\end{example}

\subsection{Discrete types and isolated points}
\label{sec:discrete-isolated}

\begin{definition}[Isolatedness and discreteness]
\label{def:isolated}
A point $x : X$ is \emph{isolated} if we can decide, for every
$y : X$, whether $x$ and $y$ are equal:
\[
  \text{$x$ is isolated} \;:\equiv\; \textstyle\prod_{y:X}\,(\text{$x = y$ is decidable}) .
\]
A type $X$ is \emph{discrete} if every point of $X$ is isolated, which
amounts to saying that $X$ has decidable equality.
\AgdaRefs{\Agda{UF.DiscreteAndSeparated}{is-isolated};
  \texttt{is-discrete}.}
\end{definition}
\begin{examples}
\label{ex:isolated}
  \leavevmode
  \begin{enumerate}
  \item The types $\Zero$, $\One$ and $\Two$ are discrete, and so is
    every proposition.
  \item The type of natural numbers is discrete.
  \item\label{item:NInf-discrete-wlpo} The type $\NInf$ is discrete if
    and only if $\WLPO$ holds.
  \item\label{item:finite-isolated} The finite points $\underline n$ of
    $\NInf$ are isolated.
  \item \label{item:infty-isolated-wlpo} The point $\infty$ of $\NInf$
    is isolated if and only if $\WLPO$ holds.
  \item\label{item:iota1-embedding-dense} The map $\N+\One \to \NInf$
    sending $\mathrm{inl}\,n$ to $\underline n$ and the added point to
    $\infty$ is an extremely dense embedding (\Cref{def:dense}).
  \end{enumerate}
\AgdaRefs{\Agda{UF.DiscreteAndSeparated}{ℕ-is-discrete};
  \Agda{Taboos.WLPO}{ℕ∞-discrete-gives-WLPO},
  \texttt{WLPO-gives-ℕ∞-discrete};
  \Agda{TypeTopology.GenericConvergentSequence}{finite-isolated};
  \Agda{UF.DiscreteAndSeparated}{props-are-discrete},
  \texttt{𝟚-is-discrete};
  \Agda{TypeTopology.GenericConvergentSequence}{ι𝟙-is-embedding},
  \texttt{ι𝟙-dense}.}
\end{examples}

\begin{proof}
Every proposition is discrete, since any two of its points are equal.
The types $\Zero$ and $\One$ are propositions, and $\Two$ is discrete,
its two points being distinct.

The natural numbers are discrete, with equality decided by induction.

A finite point $\underline n$ is isolated, because $u = \underline n$
holds exactly when $u_i = 1$ for $i < n$ and $u_n = 0$, which is decided
by inspecting finitely many digits.

If $\NInf$ is discrete then in particular $\infty$ is isolated, and if
$\infty$ is isolated then $\WLPO$ holds, because a binary sequence $p$
gives the conatural number $u$ with digits
$u_i = \min\{p\,0,\dots,p\,i\}$ (\Cref{ex:NInf}), and $u = \infty$
exactly when $p$ is constantly~$1$. Conversely, $\WLPO$ makes $\NInf$
discrete, because the decreasing condition is a proposition, so that
two points are equal exactly when their digits agree, which is to say
that the sequence $i \mapsto (\text{$u_i$ equals $v_i$})$ is
constantly~$1$.

The map $\N+\One \to \NInf$ is an embedding, since its restrictions to
$\N$ and to $\One$ are embeddings with disjoint images, $\infty$ not
being finite. The restriction to $\N$ is the embedding
$\underline{(-)}$, and the restriction to $\One$ is the constant map
at~$\infty$, which is an embedding because $\NInf$ is a set. The map is
extremely dense, because a point of $\NInf$ outside its image would be
neither finite nor equal to~$\infty$, which is impossible.
\end{proof}

\begin{lemma}
\label{lem:discrete-sets}
\leavevmode
\begin{enumerate}
\item\label{item:discrete-types-are-sets} Every discrete type is a set.
  \AgdaRefs{\Agda{UF.DiscreteAndSeparated}{discrete-is-Id-collapsible};
  \texttt{discrete-types-are-sets}.}
\item\label{item:isolated-gives-h-isolated} Every isolated point is
  homotopy isolated (\Cref{def:h-isolated}).
  \AgdaRefs{\Agda{UF.DiscreteAndSeparated}{isolated-points-are-h-isolated}.}
\item\label{item:being-isolated-is-prop} Isolatedness and
  discreteness are propositions.
  \AgdaRefs{\Agda{UF.DiscreteAndSeparated}{being-isolated-is-prop};
  \texttt{being-discrete-is-prop}.}
\end{enumerate}
\end{lemma}

\begin{proof}
From decidability of $x = y$ we get an endofunction of $x = y$ that
returns a fixed identification in the positive case and is vacuous in
the negative one, hence is weakly constant, so Hedberg's~\Cref{thm:hedberg}
applies, at a single point
for~(\labelcref{item:isolated-gives-h-isolated}) and at every point
for~(\labelcref{item:discrete-types-are-sets}).

For the last item, the isolatedness of $x$ is a $\Pi$-type of decidability
statements about identity types, and to show that it is a proposition
we may assume a witness, since a type is a proposition as soon as it is
one under the assumption of a point. That assumption makes each
type $x = y$ a proposition,
by~(\labelcref{item:isolated-gives-h-isolated}), so each decidability
statement is a binary sum of two propositions that exclude each other,
hence a proposition, and so is the $\Pi$-type. Discreteness is then a
$\Pi$-type of propositions, and hence a proposition.
\end{proof}

\begin{remark}
\label{rem:isolated-vs-h-isolated}
Classically, isolatedness and homotopy isolatedness
(\Cref{def:h-isolated}) coincide. One direction is constructive, every
isolated point being homotopy isolated
(\Cref{lem:discrete-sets}(\labelcref{item:isolated-gives-h-isolated})),
and for the other, excluded middle makes the proposition $x = y$
decidable for every homotopy isolated point~$x$ and every point~$y$.
Their coincidence fails constructively, because $\NInf$ is a set, so
that every point of it is homotopy isolated, whereas $\infty$ is
isolated if and only if $\WLPO$ holds
(\Cref{ex:isolated}(\labelcref{item:infty-isolated-wlpo})).
\AgdaRefs{\Agda{TypeTopology.GenericConvergentSequence}{ℕ∞-is-set}.}
\end{remark}

\begin{proposition}
\label{prop:discrete-closure}
\leavevmode
\begin{enumerate}
\item\label{item:discrete-plus} Isolatedness is preserved by
  $\mathrm{inl}$ and $\mathrm{inr}$, and hence discreteness is closed
  under binary sums.
  \AgdaRefs{\Agda{UF.DiscreteAndSeparated}{inl-is-isolated};
    \texttt{inr-is-isolated}; \texttt{+-is-discrete}.}
\item\label{item:discrete-retract} Discreteness is closed under
  retracts.
  \AgdaRefs{\Agda{UF.DiscreteAndSeparated}{retract-is-discrete}.}
\item\label{item:discrete-sigma} A subtype of a discrete type is
  discrete, and, more generally, at a discrete index type the
  $\Sigma$-types of discrete types are discrete, so that in particular
  binary products are.
\item\label{item:discrete-sigma-isolated} At any index type $X$, if
  $x:X$ and $y:Y\,x$ are isolated then so is $(x,y)$, and conversely,
  isolatedness of $(x,y)$ gives isolatedness of $y$ when $X$ is a set,
  and for products also of $x$, with no hypothesis. (Isolatedness of
  the first component of a general sum is obtained when the fibers are
  compact (\Cref{lem:sigma-isolated}).)
  \AgdaRefs{\Agda{TypeTopology.SigmaDiscrete}{Σ-isolated};
    \texttt{Σ-isolated-right} (needs $X$ a set); \texttt{×-isolated};
    \texttt{×-isolated-left}; \texttt{×-isolated-right}.}
\item\label{item:discrete-reflected} Left-cancellable maps, embeddings
  and equivalences reflect isolatedness, and equivalences also preserve
  isolatedness, and preserve and reflect discreteness.
\end{enumerate}
\AgdaRefs{\Agda{UF.DiscreteAndSeparated}{subtype-is-discrete};
\Agda{TypeTopology.SigmaDiscrete}{Σ-is-discrete}; \texttt{×-is-discrete};
\Agda{UF.DiscreteAndSeparated}{lc-maps-reflect-isolatedness};
\texttt{lc-maps-reflect-discreteness};
\texttt{embeddings-reflect-isolatedness};
\texttt{embeddings-reflect-discreteness};
\texttt{equivs-reflect-isolatedness};
\texttt{equivs-preserve-isolatedness};
\texttt{equivs-preserve-discreteness}; \texttt{equiv-to-discrete}.}
\end{proposition}

\begin{proof}
(\labelcref{item:discrete-plus}) If $x$ is isolated then so is
$\mathrm{inl}\,x$, because an element of $X+Y$ is either of the form
$\mathrm{inl}\,x'$, and then equality with $\mathrm{inl}\,x$ reduces to
$x'=x$, or of the form $\mathrm{inr}\,y$, and then the two are distinct,
and symmetrically for $\mathrm{inr}$. Discreteness of $X+Y$ follows by
cases. (\labelcref{item:discrete-retract}) For a retraction $r$ with
section $s$, decide $s\,y = s\,y'$ in the discrete type. Equality of
these two values gives $y = y'$ after applying the retraction $r$ and
using $r\circ s \sim \id$, and their inequality gives $y \neq y'$.

(\labelcref{item:discrete-sigma}) A proposition valued subtype of a
discrete type is discrete because equality of pairs reduces to equality
of first components, the second being a proposition. For $\Sigma$-types
with a discrete index and discrete fibers, both components of every
point are isolated, so every point is isolated
by~(\labelcref{item:discrete-sigma-isolated}).

(\labelcref{item:discrete-sigma-isolated}) Let $x : X$ and $y : Y\,x$
be isolated. To decide whether $(x,y)$ equals a given pair, decide first
components using the isolatedness of $x$. If they differ, the pairs
differ. If they agree, the identification between them is
$\mathrm{refl}$, an isolated point being homotopy isolated
(\Cref{lem:discrete-sets}(\labelcref{item:isolated-gives-h-isolated})),
so that the second components are compared without transport, using
the isolatedness of $y$.
Conversely, from isolatedness of $(x,y)$ we recover isolatedness of $y$
by pairing an arbitrary $y'$ with $x$, for which we need the
identification of first components to be unique, which holds when $X$
is a set. For $\times$ the set hypothesis is unnecessary, the
fiber not depending on the first component, and there isolatedness of
$x$ is recovered as well, by pairing an arbitrary $x'$ with $y$.

(\labelcref{item:discrete-reflected}) A left-cancellable $f$ reflects
isolatedness, since to decide $x' = x$ we decide $f\,x' = f\,x$ and
cancel.
Embeddings and equivalences are left cancellable. An equivalence also
preserves isolatedness, since its inverse is left cancellable and so
reflects isolatedness.
\end{proof}

\begin{proposition}
\label{prop:discrete-exponential}
\leavevmode
\begin{enumerate}
\item If a type $Y$ has two distinct points, then discreteness of
  $Y^X$ implies that the emptiness of $X$ is decidable.
  \AgdaRefs{\Agda{UF.DiscreteAndSeparated}{discrete-exponential-has-decidable-emptiness-of-exponent}.}
\item A map out of a proposition into homotopy isolated points, in
  particular into isolated points, is an embedding.
  \AgdaRefs{\Agda{UF.DiscreteAndSeparated}{maps-of-props-into-h-isolated-points-are-embeddings};
  \texttt{maps-of-props-into-isolated-points-are-embeddings};
  \texttt{global-point-is-embedding}.}
\end{enumerate}
\end{proposition}

\begin{proof}
For the first claim, let $y_0 \neq y_1$ in $Y$ and decide, in the
discrete type $Y^X$, whether the two constant maps
$\lambda x.\,y_0$ and $\lambda x.\,y_1$ are equal. If they are, then
$y_0 = y_1$ at every $x : X$, so $X$ must be empty. If they are not,
then $X$ cannot be empty, since over an empty $X$ any two maps are
equal.

For the second, let $P$ be a proposition and $f : P \to Y$ land in
homotopy isolated points. The fiber of $f$ over $y$ is
$\Sigma_{p:P} (f\,p = y)$, and any two of its elements have equal first
components, since $P$ is a proposition, and equal second components,
since $f\,p = y$ is a proposition by homotopy isolatedness. So the
fibers are propositions, and hence $f$ is an embedding. Isolated
points are homotopy isolated
(\Cref{lem:discrete-sets}(\labelcref{item:isolated-gives-h-isolated})),
which gives the statement for maps into isolated points.
\end{proof}

\subsection{Totally separated types}
\label{sec:totally-separated-types}
The following is a boolean Leibniz principle, whose classical topological
reading is that the clopens separate the points.
\begin{definition}[Total separatedness]
\label{def:tot-sep}
A type $X$ is \emph{totally separated} if the maps $X \to \Two$
separate the points of $X$, in the positive sense that for all
$x,y:X$, if $p\,x = p\,y$ for every $p : X \to \Two$, then $x = y$.
\AgdaRefs{\Agda{TypeTopology.TotallySeparated}{is-totally-separated}.}
Write \[ x =_2 y \,:\equiv\, \prod_{p:\Two^X} p\,x = p\,y \] for the
equivalence relation that $x$ and $y$ are indistinguishable by
$\Two$-valued maps, so that total separatedness amounts to
$x =_2 y \to x = y$.
\end{definition}

This notion is classically uninteresting, because excluded middle
implies that every set is totally separated, in view of
\Cref{lem:tot-sep-basic}(\labelcref{item:disc-gives-ts}).

\begin{lemma}
\label{lem:tot-sep-basic}
\leavevmode
\begin{enumerate}
\item\label{item:disc-gives-ts} Every discrete type is totally
  separated.
  \AgdaRefs{\Agda{TypeTopology.TotallySeparated}{discrete-types-are-totally-separated}.}
\item\label{item:ts-gives-nnsep} Every totally separated type is
  $\neg\neg$-separated.
  \AgdaRefs{\Agda{TypeTopology.TotallySeparated}{totally-separated-types-are-¬¬-separated}.}
\item\label{item:ts-is-set} Every totally separated type is a set.
  \AgdaRefs{\Agda{TypeTopology.TotallySeparated}{totally-separated-types-are-sets}.}
\item \label{item:being-ts-is-prop} Total separatedness is a
  proposition.
  \AgdaRefs{\Agda{TypeTopology.TotallySeparated}{being-totally-separated-is-prop}.}
\end{enumerate}
\end{lemma}

\begin{proof}
If $X$ is discrete and $x =_2 y$, apply the hypothesis to the
characteristic function of $x$, which sends $x$ to $0$. It must then
send $y$ to $0$ as well, which for that particular function means
$x = y$.

For $\neg\neg$-separatedness, suppose $\neg\neg(x=y)$. For each
$p : X \to \Two$ we get $\neg\neg(p\,x = p\,y)$ by applying $p$, and
$\Two$ is $\neg\neg$-separated because equality in it is decidable. So
$x =_2 y$, and total separatedness gives $x = y$. Hence every totally
separated type is a set, since $\neg\neg$-separated types are sets
by~\Cref{lem:nn-separated-is-set}.

For the last item, given a witness of total separatedness the type is
a set by~(\labelcref{item:ts-is-set}), so total separatedness is a
$\Pi$-type of identity types of a set, hence a proposition, and a type
is a proposition as soon as it is one under the assumption of a point
of it.
\end{proof}

\pagebreak[3]
\begin{examples}
\label{ex:tot-sep-failures}
  \leavevmode
  \begin{enumerate}
\item\label{item:circle-not-tot-sep} The homotopical circle fails to be
  totally separated for homotopical reasons, because it is not a
  set~\cite[Section~6.4]{HoTTBook}, and totally separated types are sets
  by~\Cref{lem:tot-sep-basic}(\labelcref{item:ts-is-set}).
  \AgdaRefs{\Agda{SyntheticHomotopyTheory.Circle.Construction}{Tℤ},
    the circle as the type of torsors over the integers, and
    \texttt{loops-at-base-equivalent-to-ℤ}, which gives that it is not a
    set.}
\item\label{item:Omega-tot-sep-gives-EM} The type $\Omega_{\UU}$ of
  propositions fails to be totally separated for constructive reasons:
  if $\Omega_{\UU}$ is $\neg\neg$-separated (in particular, if it is
  totally separated), then excluded middle holds for~$\UU$
  (\Cref{def:excluded-middle}).
  \AgdaRefs{\Agda{TypeTopology.TotallySeparated}{Ω-¬¬-separated-gives-DNE};
    \texttt{Ω-¬¬-separated-gives-EM};
    \texttt{Ω-totally-separated-gives-EM}.}
\item\label{item:two-infinities-not-tot-sep} More relevantly for us,
  the type $\Sigma_{x:\NInf}\,\Two^{x=\infty}$ fails to be totally
  separated for constructive reasons, because its total separatedness
  implies $\neg\neg\WLPO$. This example
  is~\cite[Theorem~2.7]{EscardoStreicherUniverses2016},
  \AgdaRefs{\Agda{TypeTopology.FailureOfTotalSeparatedness}{ℕ∞₂},
    \texttt{ℕ∞₂-is-not-totally-separated-in-general}.}  and is the
  prototype here for the failure of total separatedness proved
  in~\Cref{sec:brouwer-four} below.
  \end{enumerate}
We can visualize the type $\Sigma_{x:\NInf}\,\Two^{x=\infty}$ of item~(\labelcref{item:two-infinities-not-tot-sep}) as a
sequence that converges to two distinct points at infinity:
  \[
    \begin{tikzpicture}[scale=0.8, every node/.style={font=\footnotesize}]
  \fill (0,0.3) circle (1.5pt) node[above=2pt] {$0$};
  \fill (1.3,0.3) circle (1.5pt) node[above=2pt] {$1$};
  \fill (1.9,0.3) circle (1.5pt) node[above=2pt] {$2$};
  \fill (2.5,0.3) circle (1.5pt) node[above=2pt] {$3$};
  \fill (3.1,0.3) circle (1.5pt) node[above=2pt] {$4$};
  \node at (3.7,0.3) {$\cdots$};
  \fill (5.0,0.55) circle (1.5pt) node[right=2pt] {$\infty_0$};
  \fill (5.0,0.05) circle (1.5pt) node[right=2pt] {$\infty_1$};
\end{tikzpicture}
\]
In classical topology, the intersection of two compact subsets of a
Hausdorff space, and hence of a totally separated space (the clopens
separate the points), is compact. This type, whose total separatedness
fails constructively, has two compact subtypes whose intersection is
not compact:
\[
\begin{tikzpicture}[every node/.style={font=\footnotesize}]
  \node (lhs) {$\{0,1,\ldots,\infty_0\}$};
  \node[right=4pt of lhs] (cap) {$\cap$};
  \node[right=4pt of cap] (rhs) {$\{0,1,\ldots,\infty_1\}$};
  \node[right=4pt of rhs] (eq) {$=$};
  \node[right=4pt of eq] (nat) {$\N$};
  \node[below=10pt of cap] (cpt) {compact};
  \node[below=6pt of nat] (ncpt) {not compact};
  \draw[->] (cpt.north west) to (lhs.south);
  \draw[->] (cpt.north east) to (rhs.south);
  \draw[->] (ncpt.north) -- (nat.south);
\end{tikzpicture}
\]
\end{examples}
\begin{proof}[Proof of~\Cref{ex:tot-sep-failures}(\labelcref{item:Omega-tot-sep-gives-EM},\labelcref{item:two-infinities-not-tot-sep})]
(\labelcref{item:Omega-tot-sep-gives-EM}) Let $P$ be a proposition with
$\neg\neg P$, and regard it as a point $p : \Omega_{\UU}$. Then $P$
holds if and only if $p = \top$, so $\neg\neg P$ gives
$\neg\neg(p = \top)$, and $\neg\neg$-separatedness of $\Omega_{\UU}$
gives $p = \top$, hence $P$. That is double negation elimination, which
is equivalent to excluded middle. Total separatedness gives
$\neg\neg$-separatedness
by~\Cref{lem:tot-sep-basic}(\labelcref{item:ts-gives-nnsep}).

(\labelcref{item:two-infinities-not-tot-sep}) The type has two copies
of the point at infinity, namely the pairs
$\infty_0 = (\infty,\lambda\_.\,0)$ and
$\infty_1 = (\infty,\lambda\_.\,1)$, which are distinct because $\NInf$
is a set, so that the fibers over $\infty$ are compared without
transport. A map $p$ separating them gives the two maps
$u \mapsto p\,(u,\lambda\_.\,0)$ and $u \mapsto p\,(u,\lambda\_.\,1)$
from $\NInf$ to $\Two$, which agree at every finite point, the type
$\underline n = \infty$ being empty, and differ at $\infty$. Comparing
their values at $u : \NInf$ then decides whether $u = \infty$. If the
values agree, then $u \neq \infty$, and if they differ, then $u$ is not
finite, and hence is $\infty$, since a conatural number that is not
finite is $\infty$. That is $\WLPO$ in the formulation
of~\Cref{ex:NInf}.
\AgdaRefs{\Agda{Taboos.BasicDiscontinuity}{disagreement-taboo}.}
So $\neg\WLPO$ leaves the two points
indistinguishable by $\Two$-valued maps, and total separatedness would
then identify them, so total separatedness of the sum gives
$\neg\neg\WLPO$.
\end{proof}

Write \[ x \apart_2 y :\equiv \exists_{p:\Two^X}\, p\,x \neq p\,y \] for
the relation that some $\Two$-valued map separates $x$
from $y$. For every type~$X$, whether or not it is totally separated,
this is an apartness. Irreflexivity says that $p\,x \neq p\,x$ is
impossible, symmetry is clear, and for cotransitivity, given $p$ with
$p\,x \neq p\,y$ and any $z$, decide $p\,x = p\,z$ in $\Two$, which
gives $p\,z \neq p\,y$ in one case and $p\,x \neq p\,z$ in the other.
\claim{boolean-apartness}
\AgdaRefs{\Agda{TypeTopology.TotallySeparated}{\_♯₂\_},
  \texttt{♯₂-is-apartness}.}

\begin{proposition}
\label{prop:tot-sep-equiv}
The following are equivalent for any type $X$.
\begin{enumerate}
\item\label{item:ts-definition} $X$ is totally separated, meaning that
  $x =_2 y \to x = y$.
\item\label{item:ts-quasicomponent} The \emph{quasi-component}
  $\Sigma_y\,x=_2 y$ of every point $x:X$ is a singleton.
\item\label{item:ts-eval-embedding} The evaluation function
  $\mathrm{eval}_X : X \to \Two^{\Two^X}$ that maps $x : X$ to the
  function $(p \mapsto p\,x) : \Two^X \to \Two$ is an embedding.
\item\label{item:ts-apartness-tight} The apartness relation $\apart_2$ is
  tight.
\end{enumerate}
\AgdaRefs{\Agda{TypeTopology.TotallySeparated}{is-totally-separated};
\texttt{is-totally-separated₁};
\texttt{totally-separated-gives-totally-separated₁};
\texttt{totally-separated₁-types-are-totally-separated}; \texttt{eval};
\texttt{is-totally-separated₂};
\texttt{totally-separated-gives-totally-separated₂};
\texttt{totally-separated₂-gives-totally-separated}; module
\texttt{total-separatedness-via-apartness}, \texttt{\_♯₂\_};
\texttt{♯₂-is-apartness}; \texttt{is-totally-separated₃};
\texttt{totally-separated₃-gives-totally-separated};
\texttt{totally-separated-gives-totally-separated₃}.}
\end{proposition}

\begin{proof}
(\labelcref{item:ts-definition})$\Leftrightarrow$%
(\labelcref{item:ts-quasicomponent}).  The quasi-component of a point
$x$ has the point $(x,\mathrm{refl}_2)$, where $\mathrm{refl}_2$
witnesses the reflexivity of $=_2$. Assuming~%
(\labelcref{item:ts-definition}), any point $(y,e)$ of the
quasi-component has $y = x$, and the second components then agree
because $x =_2 y$ is a proposition, being a $\Pi$-type of identity
types in the set $\Two$. So the quasi-component of $x$ is a singleton,
with center $(x,\mathrm{refl}_2)$. Conversely, assume~%
(\labelcref{item:ts-quasicomponent}) and let $e$ witness $x =_2 y$.
The points $(x,\mathrm{refl}_2)$ and $(y,e)$ of the quasi-component of
$x$ are equal, and their first components give $x = y$.

(\labelcref{item:ts-definition})$\Leftrightarrow$%
(\labelcref{item:ts-eval-embedding}).  An embedding is a map whose
fibers are propositions, and the fiber of $\mathrm{eval}_X$ over a
point $\Phi$ of $\Two^{\Two^X}$ is the type
\[
  \Sigma_{x : X}\,(\mathrm{eval}_X\,x = \Phi).
\]
By function extensionality, an identification
$\mathrm{eval}_X\,x = \mathrm{eval}_X\,y$ amounts to $x =_2 y$.
Assuming~(\labelcref{item:ts-eval-embedding}), if $x =_2 y$ then the
points $(x,\mathrm{refl})$ and $(y,e)$, where $e$ is the
identification that $x =_2 y$ gives in this way, lie in the fiber over
$\mathrm{eval}_X\,x$, which is a proposition, so they are equal and
their first components give $x = y$. Conversely, assuming~%
(\labelcref{item:ts-definition}), any two points of a fiber have equal
first components, and their second components then agree because
$\Two^{\Two^X}$ is a set, being a product of copies of the set $\Two$.

(\labelcref{item:ts-definition})$\Leftrightarrow$%
(\labelcref{item:ts-apartness-tight}).  Tightness of $\apart_2$ says
that $\neg(x \apart_2 y)$ implies $x = y$, and the hypothesis says
that no $p$ separates $x$ from $y$. The truncation in $\apart_2$ can
be dropped, because the statement is negated and hence a proposition.
Since $p\,x \neq p\,y$ is decidable in $\Two$, the hypothesis is
equivalent to $p\,x = p\,y$ for every~$p$, that is, to $x =_2 y$. So
tightness of $\apart_2$ and total separatedness are the same
assertion.
\end{proof}

\begin{proposition}
\label{prop:tot-sep-closure}
Total separatedness is closed under
\begin{enumerate}
\item\label{item:tot-sep-retract} retracts, hence under
  equivalence,
  \AgdaRefs{\Agda{TypeTopology.TotallySeparated}{retract-of-totally-separated};
  \texttt{equiv-to-totally-separated}.}
\item\label{item:tot-sep-product} binary products,
  \AgdaRefs{\Agda{TypeTopology.TotallySeparated}{×-totally-separated}.}
\item\label{item:tot-sep-sigma} sums when the index type is discrete, and, as the special case
  of this indexed by $\Two$, under binary sums,
  \AgdaRefs{\Agda{TypeTopology.TotallySeparated}{Σ-is-totally-separated-if-index-type-is-discrete};
    \texttt{+-totally-separated}.}
\item\label{item:tot-sep-pi} arbitrary products.
  \AgdaRefs{\Agda{TypeTopology.TotallySeparated}{Π-is-totally-separated}.}
\end{enumerate}
\end{proposition}

\begin{proof}
For retracts, let $r : X \to Y$ be a retraction with section $s$, and
let $y =_2 y'$ in $Y$. Testing $y$ and $y'$ against the maps of
the form $q \circ s$ with $q : X \to \Two$ gives $s\,y =_2 s\,y'$,
hence $s\,y = s\,y'$ by the total separatedness of $X$, and applying
$r$ gives $y = y'$.

For sums over a discrete index, let $(a,b) =_2 (x,y)$ in
$\Sigma\,Y$. Testing against the maps that ignore the second
component gives $a =_2 x$ in the index, and hence $a = x$ by
discreteness
and~\Cref{lem:tot-sep-basic}(\labelcref{item:disc-gives-ts}). Let $b'$
be the transport of $b$ along this identification. Given
$q : Y\,x \to \Two$, define a function on $\Sigma\,Y$ by deciding,
for each index $u$, whether $u = x$, which discreteness allows, and
sending $(u,v)$ to the value of $q$ at the transport of $v$ along the
identification $u = x$ when there is one, and to $0$ when there is
not. A discrete type is a set
(\Cref{lem:discrete-sets}(\labelcref{item:discrete-types-are-sets})),
so the identification decided at $x$ itself is $\mathrm{refl}$, and the
function takes the value $q\,y$ at $(x,y)$. Testing $(a,b)$ and
$(x,y)$ against it gives $q\,b' = q\,y$. So $b' =_2 y$, and the total
separatedness of $Y\,x$ gives $b' = y$, which together with $a = x$ gives
$(a,b) = (x,y)$.

For binary sums, testing against the map that is $0$ on $X$ and $1$ on
$Y$ shows that two points indistinguishable by $\Two$-valued maps lie
in the same summand, and for $\mathrm{inl}\,x$ and $\mathrm{inl}\,x'$
testing against the maps that apply some $q : X \to \Two$ on $X$ and
are constantly~$0$ on $Y$ gives $x =_2 x'$, and symmetrically for
$\mathrm{inr}$.

For products, let $\varphi =_2 \psi$ in $\prod_x Y\,x$, and fix
$x : X$ and $q : Y\,x \to \Two$. Testing $\varphi$ and $\psi$ against
the map $\lambda\gamma.\,q\,(\gamma\,x)$ gives
$q\,(\varphi\,x) = q\,(\psi\,x)$, so that $\varphi\,x =_2 \psi\,x$,
and the total separatedness of $Y\,x$ gives $\varphi\,x = \psi\,x$.
Since $x$ was arbitrary, the functions $\varphi$ and $\psi$ agree at
every point and hence are equal. For binary products, testing $(a,b)$
and $(x,y)$ against the maps that ignore one of the components gives
$a =_2 x$ and $b =_2 y$.
\end{proof}

\begin{examples}
\label{ex:tot-sep}
\leavevmode
\begin{enumerate}
\item\label{item:ts-not-sigma} Total separatedness fails
  constructively to be closed under $\Sigma$. The index type and the
  fibers of the type $\Sigma_{x:\NInf}\,\Two^{x=\infty}$ are totally
  separated, whereas its own total separatedness implies
  $\neg \neg \WLPO$
  (\Cref{ex:tot-sep-failures}(\labelcref{item:two-infinities-not-tot-sep})).
\item\label{item:simple-types} The \emph{simple types} are the
  smallest collection of types containing $\N$ and closed under
  function types, and the Baire type $\N^\N$ is one of them. Every
  simple type is totally separated and pointed, and has~$\N$ as a
  retract. More generally, every retract of~$\N$, such as the type~$\Two$,
  is a retract of every simple type.
  \AgdaRefs{\Agda{TypeTopology.SimpleTypes}{simple-type},
\texttt{simple-types-are-totally-separated},
\texttt{simple-types-pointed}, \texttt{simple-types-r},
\texttt{ℕ-is-retract-of-any-simple-type};
\Agda{TypeTopology.TotallySeparated}{ℕ-is-totally-separated},
\texttt{Baire-is-totally-separated}.}
\item\label{item:cantor-not-discrete} The converse
  of~\Cref{lem:tot-sep-basic}(\labelcref{item:disc-gives-ts}) fails
  constructively, since the Cantor type~$\Cantor$ is totally separated,
  and its discreteness implies~$\WLPO$.
  \AgdaRefs{\Agda{TypeTopology.TotallySeparated}{𝟚-is-totally-separated},
    \texttt{Cantor-is-totally-separated};
    \Agda{UF.DiscreteAndSeparated}{retract-is-discrete};
    \Agda{Taboos.WLPO}{ℕ∞-discrete-gives-WLPO}.}
\end{enumerate}
\end{examples}
\noindent Total separatedness also fails constructively to be closed under
surjective images, as we will see in \Cref{ex:tot-sep-counterexample}.
\begin{proof}
(\labelcref{item:ts-not-sigma}) The index type $\NInf$ is a retract
of~$\Cantor$ (\Cref{ex:NInf}), and the Cantor type and the fibers
$\Two^{x=\infty}$ are products of copies of~$\Two$, which is totally
separated, being discrete
(\Cref{lem:tot-sep-basic}(\labelcref{item:disc-gives-ts})). So the
index type and the fibers are totally separated, because total
separatedness is closed under products and retracts
(\Cref{prop:tot-sep-closure}).

(\labelcref{item:simple-types}) All three claims are proved by
induction on the construction of the simple types. The type $\N$ is
discrete and hence totally separated, and total separatedness is
closed under $\Pi$
(\Cref{prop:tot-sep-closure}(\labelcref{item:tot-sep-pi})).
Pointedness holds at the base type and passes to function types by
taking constant functions. For the retract claim, the induction step
is that if $A$ is a retract of $Y$ and $X$ is pointed, then $A$ is a
retract of $Y^X$, sending a point of $A$ to the constant function at
its image in $Y$, and recovering it from a function $X \to Y$ by
evaluating at the point of $X$ and retracting.
\AgdaRefs{\Agda{UF.Retracts}{retracts-compose};
\Agda{UF.Retracts-FunExt}{codomain-is-retract-of-function-space-with-pointed-domain}.}
A retract of $\N$ is a
\looseness=-1
retract of the base type, and hence of every simple type, and the type
$\Two$ is a retract of~$\N$.

(\labelcref{item:cantor-not-discrete}) The Cantor type~$\Cantor$ is a
product of copies of $\Two$, and so is totally separated by the same
reasoning as in~(\labelcref{item:ts-not-sigma}). It has $\NInf$ as a
retract (\Cref{ex:NInf}), so its discreteness gives that of $\NInf$
(\Cref{prop:discrete-closure}(\labelcref{item:discrete-retract})),
which implies $\WLPO$
(\Cref{ex:isolated}(\labelcref{item:NInf-discrete-wlpo})).
\end{proof}

As illustrated above, totally separated types are not in general
closed under sums. The following is a sufficient condition for a sum
to be totally separated, which will be useful for our purposes
(\Cref{thm:four-interp-props}(\labelcref{item:compactsep-interp})).
\begin{lemma}
\label{lem:sigma-NInf-tot-sep}
For any family $A : \NInf \to \VV$ with $A\,u$ totally separated for
every $u$ and with $A\,\infty$ a proposition, its sum $\Sigma A$ is totally
separated.
\AgdaRefs{\Agda{TypeTopology.SigmaTotallySeparated}{Σ-indexed-by-ℕ∞-is-totally-separated-if-family-at-∞-is-prop}.}
\end{lemma}

\begin{proof}
Let $(u,a) =_2 (v,b)$. Testing against the maps that ignore the
second component gives $u =_2 v$, and $\NInf$ is totally separated,
being a retract (\Cref{ex:NInf}) of the totally separated type
$\Cantor$ (\Cref{ex:tot-sep}(\labelcref{item:cantor-not-discrete})),
so $u = v$.
Transport $a$ along this identification to $a' : A\,v$, so that
$(u,a) = (v,a')$, and it remains to show that $a' = b$, for which,
$A\,v$ being totally separated, it is enough to show that $a' =_2 b$.
So let $r : A\,v \to \Two$. As $\Two$ is $\neg\neg$-separated it is
enough to rule out $r\,a' \neq r\,b$, so suppose that it holds.
If $v$ were a finite point $\underline n$, then, that point being
isolated (\Cref{ex:isolated}(\labelcref{item:finite-isolated})), $r$
would extend to a function on
$\Sigma A$ by deciding, for each index $w$, whether
$w = \underline n$, applying $r$ after transport when it is and
taking the value $0$ when it is not, as in the proof
of~\Cref{prop:tot-sep-closure} for a discrete index, and testing
$(u,a)$ and $(v,b)$ against it would give $r\,a' = r\,b$. So $v$ is
not a finite point. If $v$ were $\infty$, then $a' = b$, since
$A\,\infty$ is a proposition, and again $r\,a' = r\,b$.  So
$v \neq \infty$. But a conatural number that is not finite is
$\infty$, which is a contradiction.
\end{proof}

Every type can be turned into a totally separated type in a universal
way.
\begin{theorem}
\label{thm:ts-reflection}
Every type $X$ has a \emph{totally separated reflection}
$\tsreflection X$, constructed as the image of the evaluation map
$\mathrm{eval}_X : X \to \Two^{\Two^X}$ defined in
\Cref{prop:tot-sep-equiv}(\labelcref{item:ts-eval-embedding}),
together with a reflector $\etaT : X \to \tsreflection X$, such that
\begin{enumerate}
\item\label{item:ts-reflection-surjection} $\tsreflection X$ is totally
  separated, and $\etaT$ is a surjection,
\item\label{item:ts-reflection-universal} for every totally separated
  $A$ and $f:X\to A$, there is a unique $\bar f : \tsreflection X \to A$
  with $\bar f \circ \etaT = f$, or, equivalently, precomposition with
  $\etaT$ is an equivalence $A^{\tsreflection X} \simeq A^X$,
\item\label{item:ts-reflection-kernel} the reflector identifies $x$
  and $y$ if and only if $x =_2 y$.
\end{enumerate}
\AgdaRefs{\Agda{TypeTopology.TotallySeparated}{module totally-separated-reflection};
\texttt{𝕋}; \texttt{𝕋-is-totally-separated}; \texttt{ηᵀ};
\texttt{ηᵀ-is-surjection}; \texttt{ηᵀ-induction};
\texttt{totally-separated-reflection};
\texttt{totally-separated-reflection'};
\texttt{totally-separated-reflection''}; \texttt{extᵀ}; \texttt{ext-ηᵀ};
\texttt{𝕋-functor}; \texttt{𝕋-natural};
\texttt{ηᵀ-relates-identified-points};
\texttt{ηᵀ-identifies-related-points}.}
\end{theorem}

\begin{proof}
Put $\tsreflection X :\equiv \mathrm{image}(\mathrm{eval}_X)$, a
subtype of $\Two^{\Two^X}$, and let $\etaT$ be the corestriction of
$\mathrm{eval}_X$, which is a surjection onto its image, so that
surjection induction (\Cref{def:surjection}) is available for $\etaT$.

(\labelcref{item:ts-reflection-surjection}) The type $\tsreflection X$
is totally separated, since its elements are pairs $(\varphi,w)$ with
$w$ a truncated witness of being in the image, hence a proposition, so two
elements are equal as soon as their first components are. Given
$(\varphi,w) =_2 (\gamma,w')$, test against the maps $x'
\mapsto \mathrm{pr}_1 x'\,p$ for each $p : X \to \Two$, which gives
$\varphi\,p = \gamma\,p$ for every $p$, so $\varphi = \gamma$.

(\labelcref{item:ts-reflection-universal}) Let $A$ be totally separated
and $f : X \to A$. Uniqueness is immediate from $\etaT$ being a
surjection and $A$ a set
(\Cref{lem:tot-sep-basic}(\labelcref{item:ts-is-set})). For existence,
note that by~\Cref{prop:tot-sep-equiv}(\labelcref{item:ts-eval-embedding})
the map $\mathrm{eval}_A$ is an embedding, so the fibers
$\Sigma_a\,\mathrm{eval}_A\,a = \gamma$ are propositions. Hence, to find
a point in the fiber over $\lambda q.\,\varphi(q\circ f)$, we may
untruncate $w$ to some $x : X$ with $\mathrm{eval}_X\,x = \varphi$, and
then $f\,x$ is such a point, because $q\,(f\,x) = \varphi\,(q\circ f)$
for every $q : A \to \Two$. Define $\bar f(\varphi,w)$ to be the unique
$a$ in this fiber. By construction,
$q\,(\bar f\,(\varphi,w)) = \varphi\,(q\circ f)$ for every
$q : A \to \Two$, so that $q\,(\bar f\,(\etaT\,x)) = q\,(f\,x)$ at
every $x : X$, and the total separatedness of $A$ gives
$\bar f\circ\etaT = f$.

(\labelcref{item:ts-reflection-kernel}) If $\etaT\,x = \etaT\,y$ then
their first components agree, which is exactly $x =_2 y$. Conversely $x
=_2 y$ gives equality of first components, and the second components are
propositions.
\end{proof}

A special case of interest of
\Cref{thm:ts-reflection}(\labelcref{item:ts-reflection-universal}) is
$A = \Two$, which shows that the boolean valued functions on $X$ are
in canonical bijection with those on~$\tsreflection X$.

\section{Compact types}
\label{sec:compact-types-theory}

This section studies the notion of compact type, including its closure
properties, its interplay with the notions of discreteness and total
separatedness, as well as the extension of families of compact types along
embeddings, which is crucial for the purposes
of~\Cref{sec:brouwer-standard,sec:inductive-recursive}.
\AgdaRefs{\Agda{TypeTopology.CompactTypes}{};
  \Agda{TypeTopology.WeaklyCompactTypes}{};
  \Agda{TypeTopology.MicroTychonoff}{}}

\subsection{Definitions and first examples}
\label{sec:prelim-types}

\begin{definition}[Compactness]
\label{def:compact}
A type $X$ is \emph{$\Sigma$-compact}, or \emph{compact} for short, if
it is decidable, for every $p : X \to \Two$, whether $p$ has a root:
\[
  \textstyle\prod_{p : \Two^X}\,
  \big(\text{the type $\Sigma_{x:X}\, p\, x = 0$ is decidable}\big).
\]
We also consider the weaker notions of \emph{$\exists$-compactness}
and \emph{$\Pi$-compactness}, obtained by replacing
$\Sigma_{x:X}\, p\,x = 0$ by $\exists_{x:X}\, p\,x = 0$ and by
$\Pi_{x:X}\, p\,x = 1$ respectively.
\AgdaRefs{\Agda{TypeTopology.CompactTypes}{is-compact}, which unfolds
  to \texttt{is-Σ-compact};
  \Agda{TypeTopology.WeaklyCompactTypes}{is-∃-compact},
  \texttt{is-Π-compact}.}
\end{definition}

The same notions can be equivalently formulated with an arbitrary
complemented family in place of a $\Two$-valued map.

\begin{lemma}
\label{lem:compact-formulations}
For any type $X$ and any universe $\VV$, the compactness of $X$ is
logically equivalent to the statement that the sum
$\Sigma_{x:X}\, A\,x$ is decidable for every complemented family
$A : X \to \VV$, and also to the restriction of this to
proposition valued families~$A$.
\AgdaRefs{\Agda{TypeTopology.CompactTypes}{is-Σ-Compact},
  i.e.~\texttt{is-Compact}, and \texttt{Σ-Compact'} for the
  proposition valued restriction;
  \texttt{compact-types-are-Compact};
  \texttt{Compact-types-are-compact}; \texttt{Compact-resize};
  \texttt{Compact'-types-are-compact}.}
\end{lemma}

\begin{proof}
A complemented family $A$ has a characteristic function $p : X \to
\Two$, with $p\,x = 0$ giving $A\,x$ and $p\,x = 1$ refuting it, and
then $A\,x$ and $p\,x = 0$ are logically equivalent for every $x$,
because $A\,x$ rules out $p\,x = 1$. In the other direction, a
map $p : X \to \Two$ gives the family sending $x$ to $\One$ when
$p\,x = 0$ and to $\Zero$ when $p\,x = 1$, which is complemented and
proposition valued, and which lives in any universe, because $\Zero$
and $\One$ do. Pointwise logically equivalent families
have logically equivalent $\Sigma$-types, and decidability is invariant
under logical equivalence.
\end{proof}

The notions of $\exists$-compactness and $\Pi$-compactness, and the
pointed compactness of~\Cref{lem:compact-pt} below, can be treated by
the same argument.
In view of~\Cref{lem:compact-formulations}, we say simply that a type
is compact, using whichever of these formulations is convenient for
proofs, and saying which one when this matters. In particular,
compactness does not depend on the universe $\VV$.

\begin{lemma}
\label{lem:compact-decidable}
Every compact type is decidable, and every decidable proposition is
compact.
\AgdaRefs{\Agda{TypeTopology.CompactTypes}{compact-types-are-decidable};
\texttt{decidable-propositions-are-compact}.}
\end{lemma}

\begin{proof}
For the first claim, apply compactness of $X$ to the constantly~$0$
map. Either it has a root, which in particular supplies a point
of $X$, or else $0 = 1$ for every point of $X$, which
gives $\neg X$.

For the second, let $X$ be a decidable proposition and let
$p : X \to \Two$. If $\neg X$ then $p\,x = 1$ for every $x : X$
vacuously. If $x : X$, decide the value of $p\,x$. If $p\,x = 0$ then
$p$ has a root. If $p\,x = 1$ then, since $X$ is a proposition, every
$y : X$ is equal to $x$, so $p\,y = 1$.
\end{proof}

\begin{proposition}
\label{prop:global-choice}
Every type in a universe $\UU$ is compact if and only if global choice
holds for $\UU$ (\Cref{def:global-choice}).
\AgdaRefs{\Agda{TypeTopology.CompactTypes}{all-types-compact-gives-global-choice};
  \texttt{global-choice-gives-all-types-compact}.}
\end{proposition}

\begin{proof}
Assume global choice, let $X : \UU$ and let $p : X \to \Two$. Then
$\Sigma_{x:X}\,p\,x = 0$ is a type in $\UU$, so it is decidable by
global choice, which amounts to the compactness of $X$.
Conversely, if every type in $\UU$ is compact then every type in $\UU$
is decidable by~\Cref{lem:compact-decidable}.
\end{proof}

Global choice is not only constructively unacceptable, but also
inconsistent with univalence~\cite[Corollary~3.2.7]{HoTTBook}.
\claim{global-choice-univalence}
\AgdaRefs{\Agda{UF.ClassicalLogic}{Global-Choice-is-inconsistent-with-univalence},
  via \texttt{Global-Choice-gives-sethood} and
  \Agda{UF.Universes}{universes-are-not-sets}.}  Nevertheless, in this
paper we build many compact types in a constructive type theory
compatible with univalence.  We start from a trivial example and some
basic but important counterexamples.

\begin{examples}
\label{ex:compact-basic} \leavevmode
\begin{enumerate}
\item\label{item:finite-compact} Every finite type is compact, a
  \emph{finite type} being one of
  the form $\One + \cdots + \One$, which includes the base case $\Zero$
  in the obvious inductive definition.
  \AgdaRefs{\Agda{Fin.Topology}{Fin-Compact}, for \Agda{Fin.Type}{Fin}.}
\item\label{item:N-compact-lpo} $\N$ is compact if and only if
  $\LPO$ holds.
  \AgdaRefs{\Agda{Taboos.LPO}{compact-ℕ-gives-LPO},
  \texttt{LPO-gives-compact-ℕ}.}
\item\label{item:N-exists-compact-lpo} $\N$ is $\exists$-compact if and
  only if $\LPO$ holds.
  \AgdaRefs{\Agda{TypeTopology.WeaklyCompactTypes}{∃-compact-ℕ-gives-LPO},
  \texttt{LPO-gives-∃-compact-ℕ}.}
\item\label{item:N-pi-compact} $\N$ is $\Pi$-compact if and
  only if $\WLPO$ holds.
  \AgdaRefs{\Agda{TypeTopology.WeaklyCompactTypes}{Π-compact-ℕ-gives-WLPO},
  \texttt{WLPO-gives-Π-compact-ℕ}.}
\end{enumerate}
\end{examples}

\begin{proof}
(\labelcref{item:finite-compact}) follows
from~\Cref{thm:compact-closure}(\labelcref{item:compact-zero-one},\labelcref{item:compact-plus})
below, by induction on the number of summands.
Items~(\labelcref{item:N-compact-lpo})
and~(\labelcref{item:N-pi-compact}) hold by definition
(\Cref{def:lpo-wlpo}). For~(\labelcref{item:N-exists-compact-lpo}),
$\LPO$ gives the compactness of $\N$
by~(\labelcref{item:N-compact-lpo}), and hence its
$\exists$-compactness
(\Cref{prop:weakly-compact-basic}(\labelcref{item:compact-gives-exists-compact})
below). Conversely, if $\exists_n\,p\,n = 0$, then the least $n$ with
$p\,n = 0$ exists in an unspecified way, and, being unique, can be
found. So the $\exists$-compactness of $\N$ gives its compactness,
which is $\LPO$ by~(\labelcref{item:N-compact-lpo}).
\end{proof}

Compact types are closed under retracts, finite disjoint
unions, images, and, when the index type is compact, under
arbitrary sums (\Cref{sec:compact-closure}).
Compactness, discreteness and total separatedness interact with
function types much as their topological namesakes do, with arbitrary
functions of type theory playing the role that continuous maps play in
topology, but without any Brouwerian continuity principle
(\Cref{sec:compact-interplay}). Under classical logic in the
metatheory, the notions coincide with their topological counterparts
in the
model of simple types given by Kleene--Kreisel
spaces~\cite{EscardoExhaustible2008,Kleene1959}.

\subsection{Compact pointed types}
\label{sec:pointed-compactness}
For reasons pertaining to constructivity, a number of definitions and
theorems need to assume that the compact types under consideration are
pointed. Moreover, it is convenient to work with the following
characterization of compact pointed types, which we apply tacitly
in what follows.
\begin{lemma}
\label{lem:compact-pt}
The following are equivalent for any type $X$.
\begin{enumerate}
\item\label{item:compact-and-pointed} $X$ is compact and pointed.
\item\label{item:universal-witness} For every $p : X \to \Two$ we can
  find $x_0 : X$ such that $p\,x_0 = 1$ implies $p\,x=1$ for all
  $x:X$:
\[
\prod_{p:\Two^X} \Sigma_{x_0:X}\,
  \big(p\,x_0 = 1 \to \forall x. \, p\,x = 1\big).
\]
\end{enumerate}
We call such an $x_0$ a \emph{universal witness} for $p$. It is a root
of $p$ if and only if~$p$ has some root.
\AgdaRefs{\Agda{TypeTopology.CompactTypes}{compact-pointed-types-are-compact∙};
\texttt{compact∙-types-are-compact}; \texttt{compact∙-types-are-pointed};
\texttt{putative-root}.}
\end{lemma}
A function $\selection : \Two^X \to X$ assigning a universal witness
to each~$p$ is called a \emph{selection
function}~\cite{EscardoOliva2010} for~$X$.
In~(\labelcref{item:universal-witness}), a candidate root is
produced before it is known whether a root exists.
\begin{proof}
\emph{(\labelcref{item:compact-and-pointed}) implies
(\labelcref{item:universal-witness}).} Let $X$ be compact and
$x_\star : X$, and let $p : X \to \Two$. By
compactness, either $p$ has a root $x$, and we take $x_0 = x$, so
that the implication $p\,x_0 = 1 \to \forall x.\,p\,x=1$ holds
vacuously because $p\,x_0 = 0$, or else $p\,x = 1$ for every $x$, and
we take $x_0 = x_\star$, the conclusion of the implication being
the assumption of this case.

\emph{(\labelcref{item:universal-witness}) implies
  (\labelcref{item:compact-and-pointed}).} Let $p : X \to \Two$, with
universal witness $x_0$. If $p\,x_0 = 0$ then $x_0$ is a root of $p$.
If $p\,x_0 = 1$ then $p\,x = 1$ for every $x$ by the universality of
$x_0$. This is precisely the case distinction required by
compactness. For pointedness, apply the hypothesis to e.g.\ the
constantly~$0$ map, which gives a point of $X$.

\emph{A universal witness is a candidate root.} With $x_0$ a universal
witness for $p$, one direction is trivial, since if $x_0$ is a root
then $p$ has a root. For the other, suppose $p\,x = 0$ for some
$x$. Then it is not the case that $p\,x' = 1$ for every $x'$, since
that would give $p\,x = 1$, contradicting $p\,x = 0$. By the
contrapositive of the universality of $x_0$, we conclude
$p\,x_0 \neq 1$ and so $x_0$ is a root.
\end{proof}

\begin{examples}
\label{ex:compact-pt-basic}
\leavevmode
\begin{enumerate}
\item The type $\Two$ is compact pointed, with a selection function
  given by $\selection\,p = p\,0$ for any $p : \Two^\Two$.
  \AgdaRefs{\Agda{TypeTopology.CompactTypes}{𝟚-is-compact∙}.}
\item The type $\Omega_{\UU}$ of propositions in a universe $\UU$ is
  compact pointed too, despite excluded middle being undecided in our
  type theory.
  \AgdaRefs{\Agda{TypeTopology.CompactTypes}{Ω-is-compact∙}.}
\item \label{item:NInf-compact} The type $\NInf$ of~\Cref{ex:NInf} is
an infinite compact type, with selection function
$\selection\,p = (n \mapsto \min\{p(\underline k) \mid k \le n\})$.
Moreover,
$\selection\,p$ is the infimum of the set of roots of $p$, in the order
of $\NInf$ as an ordinal (\Cref{ex:NInf-ordinal}), which
may or may not be a root, but is a root if and only if $p$ has
some root.
\AgdaRefs{\Agda{TypeTopology.GenericConvergentSequenceCompactness}{ℕ∞-compact∙};
\Agda{TypeTopology.ConvergentSequenceHasInf}{ℕ∞-has-inf} for the
infimum of the set of roots;
\Agda{Ordinals.Closure}{ℕ∞ᵒ-has-infs-of-complemented-subsets} for the
order of $\NInf$ as an ordinal.}
\end{enumerate}
\end{examples}

\begin{proof}
For $\Two$, take $x_0 = p\,0$ and suppose $p\,(p\,0) = 1$. We first get
$p\,0 = 1$ by cases on $p\,0$: it is immediate when $p\,0 = 1$, and when
$p\,0 = 0$, substituting that equation into $p\,(p\,0) = 1$ gives
$p\,0 = 1$, a contradiction. Substituting $p\,0 = 1$ into
$p\,(p\,0) = 1$ then gives $p\,1 = 1$, so $p$ is constantly~$1$.

For $\Omega_{\UU}$, decide the value of $p\,\bot$. If $p\,\bot = 0$,
take $\bot$ as a universal witness, the implication holding
vacuously. If $p\,\bot = 1$, take $\top$, and we must show that if also
$p\,\top = 1$ then $p$ is constantly~$1$. Two maps
$\Omega_{\UU} \to \Two$ that agree at $\bot$ and $\top$ agree
everywhere, because the map $\Two \to \Omega_{\UU}$ is extremely dense
(\Cref{lem:two-to-Omega-dense}) and an extremely dense map is
right-cancellable with respect to maps into a $\neg\neg$-separated type
(\Cref{lem:dense-into-nn-separated-rc}), such as $\Two$, which is
discrete. Applying this to $p$
and the map constantly~$1$ gives $p\,P = 1$ for every $P$.
\AgdaRefs{\Agda{UF.DiscreteAndSeparated}{⊥-⊤-density}, which packages
  the same argument directly.}

For $\NInf$, let $p : \NInf \to \Two$ and $u = \selection\,p$, so that
$u_n = 1$ precisely when $p\,(\underline k) = 1$ for every $k \le n$.
Suppose that $p\,u = 1$. If $u$ were a finite point~$\underline n$,
then $u_n = 0$ while $u_k = 1$ for $k < n$, which by the definition of
$u$ means $p\,(\underline n) = 0$, and so $p\,u = p\,(\underline n) =
0$. So $u$ is not finite, and hence is $\infty$, since a conatural
number that is not finite is $\infty$. This says that
$p\,(\underline k) = 1$ for every $k$, and also
$p\,\infty = p\,u = 1$. So $p$ agrees with the constantly~$1$ map on
the image of the extremely dense map $\N + \One \to \NInf$
(\Cref{ex:isolated}(\labelcref{item:iota1-embedding-dense})), and hence
is constantly~$1$ by~\Cref{lem:dense-into-nn-separated-rc}, $\Two$
being $\neg\neg$-separated. For the claim about the infimum,
see~\cite{EscardoOmniscientJSL2013} or the companion formalization
file~\cite{EscardoCompactCompanion}.
\end{proof}

\subsection{Closure properties of compact types}
\label{sec:compact-closure}

\begin{theorem}
\label{thm:compact-closure}
\leavevmode
\begin{enumerate}
\item\label{item:compact-zero-one} The types $\Zero$ and $\One$ are
  compact, and so is every singleton.
  \AgdaRefs{\Agda{TypeTopology.CompactTypes}{𝟘-is-Compact}; \texttt{𝟙-is-Compact};
  \texttt{singletons-are-Compact}.}
\item\label{item:compact-plus} If $X$ and $Y$ are compact, so is $X+Y$.
  \AgdaRefs{\texttt{+-is-Compact}.}
\item\label{item:compact-sigma} If $X$ is compact and $Y\,x$ is compact
  for every $x : X$, then $\Sigma\,Y$ is compact. In particular, the
  product $X\times Y$ is compact when both $X$ and $Y$ are.
  \AgdaRefs{\texttt{Σ-is-Compact}; \texttt{×-is-Compact}.}
\item\label{item:compact-retract} If $Y$ is a retract of $X$ and $X$ is
  compact, then so is $Y$. In particular, compactness is invariant under
  equivalence.
  \AgdaRefs{\texttt{Compact-closed-under-retracts}; \texttt{Compact-closed-under-≃}.}
\item\label{item:compact-surjection} If $f : X \to Y$ is a surjection
  and $X$ is compact, then so is $Y$, and hence so is the image of any
  map out of a compact type.
  \AgdaRefs{\texttt{CompactTypesPT.surjection-Compact}; \texttt{CompactTypesPT.image-Compact}.}
\end{enumerate}
\end{theorem}

\begin{proof}
By~\Cref{lem:compact-formulations} we may take the given data to be a
complemented family $A$ on the type under consideration rather than a
$\Two$-valued map, so that the task is to decide the sum of $A$.

Items~(\labelcref{item:compact-plus}),
(\labelcref{item:compact-retract})
and~(\labelcref{item:compact-surjection}) follow one pattern, which we
state first. If $f : X \to Y$ is a map and $A$ is a complemented
family on $Y$, then the composite $A \circ f$ is a complemented family
on $X$, and a point $(x,a)$ of $\Sigma\,(A \circ f)$ gives the point
$(f\,x, a)$ of $\Sigma A$. So when $X$ is compact, the sum of
$A \circ f$ is decidable, and what remains is to show that $\Sigma A$
is empty in the case that $\Sigma\,(A \circ f)$ is, for which the
hypothesis on $f$ is used.

(\labelcref{item:compact-zero-one}) The types $\Zero$ and $\One$ are
decidable propositions, and so is any singleton, so this is an instance
of~\Cref{lem:compact-decidable}.

(\labelcref{item:compact-plus}) Apply the pattern to the injections
$\mathrm{inl} : X \to X + Y$ and $\mathrm{inr} : Y \to X + Y$, whose
domains are compact by hypothesis. If either of
$\Sigma\,(A \circ \mathrm{inl})$ and $\Sigma\,(A \circ \mathrm{inr})$
has a point, then so does $\Sigma A$. If both are empty, then so
is~$\Sigma A$, because every point of $X+Y$ is of the form
$\mathrm{inl}\,x$ or $\mathrm{inr}\,y$.

(\labelcref{item:compact-sigma}) For each $x : X$, the family
$y \mapsto A\,(x,y)$ on $Y\,x$ is complemented, and so the compactness
of the fibers makes the family
\[
  B\,x \;:\equiv\; \textstyle\Sigma_{y : Y\,x}\, A\,(x,y)
\]
on $X$ complemented. The compactness of $X$ then decides the sum of
$B$, which is the sum of $A$ up to reassociation, and decidability is
invariant under logical equivalence. Binary products are the case of a
constant family $Y$.

(\labelcref{item:compact-retract}) Apply the pattern to the retraction
$r : X \to Y$, whose section $s : Y \to X$ satisfies $r\,(s\,y) = y$
for every $y : Y$. A point $(y,a)$ of $\Sigma A$ gives the point
$(s\,y, a')$ of $\Sigma\,(A \circ r)$, where $a'$ is obtained from $a$
by transport along the identification $y = r\,(s\,y)$, and so
$\Sigma A$ is empty when
$\Sigma\,(A \circ r)$ is. An equivalence is in particular a retraction.

(\labelcref{item:compact-surjection}) Apply the pattern to the
surjection $f$. If $\Sigma\,(A \circ f)$ is empty, then
$\neg\,A\,(f\,x)$ holds for every $x : X$, and hence $\neg\,A\,y$ holds
for every $y : Y$ by surjection induction (\Cref{def:surjection}),
which is available because negations are propositions. Therefore
$\Sigma A$ is empty. The image of a map $g$ out of a compact type is
covered by applying this to the corestriction of $g$, which is a
surjection.
\end{proof}

\begin{proposition}
\label{prop:compact-reflection}
If $X$ is compact, then so is its totally separated
reflection~$\tsreflection X$ constructed in \Cref{thm:ts-reflection}.
\AgdaRefs{\Agda{TypeTopology.CompactTypes}{CompactTypesPT.surjection-Compact},
  applied to \Agda{TypeTopology.TotallySeparated}{ηᵀ-is-surjection}.}
\end{proposition}

\begin{proof}
The reflector $\etaT$ is a surjection
(\Cref{thm:ts-reflection}(\labelcref{item:ts-reflection-surjection})),
and compactness passes to surjective images
(\Cref{thm:compact-closure}(\labelcref{item:compact-surjection})).
\end{proof}

\begin{corollary}
\label{cor:compact-pt-closure}
Pointed compactness is closed under
\begin{enumerate}
\item\label{item:pt-retract} retracts, and hence equivalence
  \AgdaRefs{\Agda{TypeTopology.CompactTypes}{retract-is-compact∙};
  \texttt{compact∙-types-are-closed-under-equiv}},
\item\label{item:pt-sigma} sums, and hence binary products and coproducts
  \AgdaRefs{\Agda{TypeTopology.CompactTypes}{Σ-is-compact∙};
  \texttt{binary-Tychonoff}; \texttt{×-is-compact∙}; \texttt{+-is-compact∙};
  \texttt{𝟙+𝟙-is-compact∙}},
\item\label{item:pt-image} surjective images
  \AgdaRefs{\Agda{TypeTopology.CompactTypes}{codomain-of-surjection-is-compact∙};
  \texttt{image-is-compact∙}}.
\end{enumerate}
\end{corollary}

\begin{proof}
A compact pointed type is a compact type with a point
(\Cref{lem:compact-pt}), and each of the three constructions preserves
compactness by~\Cref{thm:compact-closure}. What remains is to produce a
point in each case.

(\labelcref{item:pt-retract}) A retraction $X \to Y$ sends a point of
$X$ to a point of $Y$, and an equivalence is in particular a
retraction.

(\labelcref{item:pt-sigma}) A point $x$ of the base together with a
point of the fiber $Y\,x$ is a point of $\Sigma\,Y$. Binary products
are the case of a constant family. Binary coproducts are the case of
the family on $\Two$ with fibers $X_0$ and $X_1$, since $\Two$ is
compact pointed (\Cref{ex:compact-pt-basic}) and $X_0 + X_1$ is
equivalent to the sum of that family, pointed compactness being
invariant under equivalence by~(\labelcref{item:pt-retract}).

(\labelcref{item:pt-image}) A surjection $X \to Y$ sends a point of $X$
to a point of $Y$.
\end{proof}

\begin{proposition}
\label{prop:compact-complemented}
Let $X$ be a compact type.
\begin{enumerate}
\item\label{item:complemented-subtype-compact} If $A$ is a
  complemented, proposition valued family on $X$, then
  the subtype $\Sigma_{x:X}\,A\,x$ is compact.
  \AgdaRefs{\Agda{TypeTopology.CompactTypes}{complemented-subset-of-compact-type}.}
\item\label{item:decide-sum-or-product} If $A$ and $B$ are
  families on $X$ and
  $\prod_{x:X}\,(A\,x + B\,x)$ holds, then we can decide between
  $\Sigma_{x:X}\,A\,x$ and $\prod_{x:X}\,B\,x$.
  \AgdaRefs{\Agda{TypeTopology.CompactTypes}{compact-gives-Σ+Π}.}
\item\label{item:complemented-choice} If $A$ is a complemented
  family on $X$ and $\Sigma_{x:X}\,A\,x$
  is non-empty, then it is pointed.
  \AgdaRefs{\Agda{TypeTopology.CompactTypes}{Complemented-choice},
  \texttt{Σ-Compactness-gives-Complemented-choice}.}
\end{enumerate}
\end{proposition}

\begin{proof}
(\labelcref{item:complemented-subtype-compact}) Let $B$ be a
complemented family on the subtype $\Sigma_{x:X}\,A\,x$, so that what
has to be decided is the sum of $B$. The family
\[
  C\,x \;:\equiv\; \textstyle\Sigma_{a : A\,x}\, B\,(x,a)
\]
on $X$ is complemented. Indeed, decide $A\,x$. If $A\,x$ is empty then
so is $C\,x$. If $a : A\,x$, decide $B\,(x,a)$. A point $b$ of
$B\,(x,a)$ gives the point $(a,b)$ of $C\,x$, and if $B\,(x,a)$ is
empty then so is $C\,x$, because $A\,x$ is a proposition, so that any
$a' : A\,x$ is equal to $a$ and a point of $B\,(x,a')$ transports to a
point of $B\,(x,a)$. The compactness of $X$ now decides the sum of $C$,
which is the sum of $B$ up to reassociation.

(\labelcref{item:decide-sum-or-product}) Write $q$ for the given point
of $\prod_{x:X}\,(A\,x + B\,x)$ and let $p : X \to \Two$ be its
indicator, with $p\,x = 0$ when $q\,x$ lands in $A\,x$ and $p\,x = 1$
when it lands in $B\,x$. The compactness of $X$ decides whether $p$ has
a root. A root $x$ comes with a point of $A\,x$, and so gives a point
of $\Sigma_{x:X}\,A\,x$. If there is no root, then $p\,x = 1$, and
hence $B\,x$, for every $x$.

(\labelcref{item:complemented-choice}) The compactness of $X$ decides
the sum of $A$, and a decidable type is $\neg\neg$-stable.
\end{proof}

\begin{remark}
\label{rem:tychonoff-independence}
Tychonoff's theorem in classical topology says that arbitrary products
of compact spaces are compact. In our case,
\Cref{thm:compact-closure}(\labelcref{item:compact-sigma}) gives this
for binary, and hence finite, products of compact types. Without
further axioms, it is not possible to prove an infinite generalization
in our type theory, not even a countable one. With our type
theoretical counterparts of the topological notions, the compactness
of the Cantor type $\Cantor$, the countable product of copies of
$\Two$, is independent, as it holds in the
topological topos and in the Kleene--Vesley
topos~\cite{vanOosten2008}, but fails in the effective
topos~\cite{Hyland1982} due to the presence of Kleene trees (infinite
binary trees with no computable paths, which give recursive
counterexamples to K\"onig's Lemma and hence to the compactness of the
Cantor space~\cite{KleeneVesley1965}).
\end{remark}
However, under the Brouwerian principle that all maps of the Cantor
type~$\Cantor$ into a discrete type are uniformly continuous, which we
are not adopting here (cf.~\Cref{sec:constructive-mathematics}), the
Cantor type becomes compact in our sense. We say that a number
$n : \N$ is a \emph{modulus of uniform continuity} of a map $p$ of
$\Cantor$ into a discrete type if $p\,\alpha = p\,\beta$ whenever the
sequences $\alpha$ and $\beta$ agree at their first $n$ digits, and
when such a number is given we say that $p$ is \emph{uniformly
  continuous}.
\claim{uniform-continuity-gives-cantor-compact}
\AgdaRefs{\Agda{TypeTopology.Cantor}{\_is-a-modulus-of-uc-of\_},
  \texttt{uniformly-continuous};
  \Agda{TypeTopology.CantorSearch}{having-root-is-decidable}, applied
  to the modulus the principle gives.}
\begin{proposition}
\label{prop:cantor-uniform-search}
For any uniformly continuous map $p : \Cantor \to \Two$ we can find a
universal witness for $p$, in the sense of~\Cref{lem:compact-pt}, and
hence decide whether $p$ has a root.
\AgdaRefs{\Agda{TypeTopology.CantorSearch}{Cantor-uniformly-searchable},
  \texttt{having-root-is-decidable}.}
\end{proposition}
\begin{proof}
For $b : \Two$ and $\alpha : \Cantor$, write $b\alpha$ for the sequence
that begins with~$b$ and continues as~$\alpha$, and write $p_b$ for the
map $\alpha \mapsto p\,(b\alpha)$. If $n{+}1$ is a modulus of
uniform continuity of $p$, then $n$ is one of $p_b$, because $b\alpha$
and $b\beta$ agree at their first $n{+}1$ digits whenever $\alpha$ and
$\beta$ agree at their first $n$ digits.

Define $\varepsilon_n\,p : \Cantor$ for every $p$ by induction on $n$.
For the base case $\varepsilon_0\,p$ we can take any sequence we
please, for example the constantly~$0$ one, and for the induction step
we define
\[
  \varepsilon_{n+1}\,p = b_0\,(\varepsilon_n\,p_{b_0}),
\]
where $b_0$ is the value at~$0$ of the function
$b \mapsto p_b\,(\varepsilon_n\,p_b)$ on~$\Two$, which is a universal
witness for it (\Cref{ex:compact-pt-basic}).

We show by induction on $n$ that, for every $p$, if $n$ is a modulus
of uniform
continuity of $p$ and $p\,(\varepsilon_n\,p) = 1$, then
$p\,\alpha = 1$ for every $\alpha$. For $n = 0$, the map $p$ is
constant, and so it takes the value $p\,(\varepsilon_0\,p) = 1$
everywhere. For~$n{+}1$, the definition gives
$p\,(\varepsilon_{n+1}\,p) = p_{b_0}\,(\varepsilon_n\,p_{b_0})$, so
that the universality of $b_0$ gives $p_b\,(\varepsilon_n\,p_b) = 1$
for both values of $b$. The induction hypothesis, applied to the
maps $p_b$ with modulus $n$, then gives $p\,(b\alpha) = 1$ for
every $b$ and $\alpha$, and every sequence is of the form $b\alpha$.
So $\varepsilon_n\,p$ is a universal witness for $p$.
\end{proof}

\subsection{\texorpdfstring{$\exists$}{∃}-compactness and \texorpdfstring{$\Pi$}{Π}-compactness}
\label{sec:weakly-compact}

\begin{proposition}
\label{prop:weakly-compact-basic}
\leavevmode
\begin{enumerate}
\item\label{item:weak-compact-are-props} Each of $\exists$-compactness
  and $\Pi$-compactness is a proposition.
  \AgdaRefs{\Agda{TypeTopology.WeaklyCompactTypes}{∃-compactness-is-prop};
  \texttt{Π-compactness-is-prop}.}
\item\label{item:compact-gives-exists-compact} Any compact type is
  $\exists$-compact.
  \AgdaRefs{\Agda{TypeTopology.WeaklyCompactTypes}{compact-types-are-∃-compact}.}
\item\label{item:exists-gives-pi-compact} Any $\exists$-compact type is
  $\Pi$-compact.
  \AgdaRefs{\Agda{TypeTopology.WeaklyCompactTypes}{∃-compact-types-are-Π-compact}.}
\item\label{item:exists-compact-stability} If $X$ is
  $\exists$-compact then $\exists_{x:X}\, p\,x = 0$ is
  $\neg\neg$-stable.
  \AgdaRefs{\Agda{TypeTopology.WeaklyCompactTypes}{∃-compactness-gives-Markov}.}
\item\label{item:pi-compact-isolated} A type $X$ is $\Pi$-compact if
  and only if the constant map $\lambda x.\,1$ is isolated
  in~$\Two^X$.
  \AgdaRefs{\Agda{TypeTopology.WeaklyCompactTypes}{is-Π-compact'};
  \texttt{Π-compact'-types-are-Π-compact}; \texttt{Π-compact-types-are-Π-compact'}.}
\item\label{item:weakly-compact-reflection} A type $X$ is
  $\exists$-compact, respectively $\Pi$-compact, if and only if its
  totally separated reflection $\tsreflection X$
  constructed in \Cref{thm:ts-reflection} is.
  \AgdaRefs{\Agda{TypeTopology.WeaklyCompactTypes}{∃-compact-types-are-∃-compact-𝕋},
  \texttt{∃-compact-𝕋-types-are-∃-compact},
  \texttt{Π-compact-types-are-Π-compact-𝕋},
  \texttt{Π-compact-𝕋-types-are-Π-compact}.}
\end{enumerate}
\end{proposition}

\begin{proof}
(\labelcref{item:weak-compact-are-props}) For each $p$, the types
$\exists_x p\,x=0$ and $\prod_x p\,x=1$ are propositions, and so is the
decidability of a proposition.

(\labelcref{item:compact-gives-exists-compact}) Compactness gives, for
each map $p$, either a root, whose truncation
gives~$\exists_x p\,x=0$, or a proof that $p$ is constantly~$1$, which
refutes it.

(\labelcref{item:exists-gives-pi-compact}) Decide $\exists_x p\,x = 0$.
In the positive case, $\prod_x p\,x=1$ is refuted, because a root of
$p$ contradicts it, and this conclusion, being a negation, may be drawn
from the truncated existence. In the negative case, no $x$ has
$p\,x=0$, and hence $p\,x=1$ for every $x$, as $\Two$ has only two
points.

(\labelcref{item:exists-compact-stability}) The proposition
$\exists_x p\,x=0$ is decidable, and decidable types are
$\neg\neg$-stable.

(\labelcref{item:pi-compact-isolated}) The $\Pi$-compactness of $X$
says that $\prod_x p\,x=1$ is decidable for every~$p$, while the
isolatedness of $\lambda x.\,1$ says that $p = \lambda x.\,1$ is
decidable for every~$p$. The type $\prod_x p\,x=1$ is pointwise
equality of $p$ with the constant map, so the two types are equivalent
and decidability is invariant under logical equivalence.

(\labelcref{item:weakly-compact-reflection}) Let $q$ and
$p = q \circ \etaT$ be corresponding functions on $\tsreflection X$
and on~$X$, using
\Cref{thm:ts-reflection}(\labelcref{item:ts-reflection-universal}) with
$A = \Two$. Decidability transfers along a logical equivalence, so it
is enough to prove
\[
  \exists_{x:X}\, p\,x = 0 \iff \exists_{y:\tsreflection X}\, q\,y = 0,
  \qquad
  \prod_{x:X}\, p\,x = 1 \iff \prod_{y:\tsreflection X}\, q\,y = 1 .
\]
A root $x$ of $p$ gives the root $\etaT\,x$ of $q$. Conversely, the
reflector is a surjection, so a root of $q$ has a preimage in an
unspecified way, and any such preimage is a root of $p$. If $q\,y = 1$
for every $y$, then $p\,x = 1$ for every $x$ by composition with the
reflector. Conversely, if $p\,x = 1$ for every $x$, then $q\,y = 1$ for
every $y$ by surjection induction along the reflector, the equations
being propositions.
\end{proof}

The following are some closure properties of the two weaker notions of
compactness.
\begin{proposition}
\label{prop:weakly-compact-closure}
\leavevmode
\begin{enumerate}
\item\label{item:weak-compact-retracts} Both $\exists$-compactness and
  $\Pi$-compactness pass to surjective images, and hence to retracts.
  \AgdaRefs{\Agda{TypeTopology.WeaklyCompactTypes}{codomain-of-surjection-is-∃-compact};
    \texttt{image-is-∃-compact}; \texttt{retract-∃-compact};
    \texttt{codomain-of-surjection-is-Π-compact};
    \texttt{image-is-Π-compact}; \texttt{retract-is-Π-compact}.}
\item\label{item:weak-compact-sigma} If $X$ is $\Pi$-compact and
  $Y\,x$ is $\Pi$-compact for every $x : X$, then the sum $\Sigma\,Y$
  is $\Pi$-compact.
  \AgdaRefs{\Agda{TypeTopology.WeaklyCompactTypes}{Π-compact-closed-under-Σ}.}
\item\label{item:weak-compact-subtype} Both notions pass to
  complemented subtypes.
  \AgdaRefs{\Agda{TypeTopology.WeaklyCompactTypes}{detachable-subset-∃-compact};
  \texttt{complemented-subtype-is-Π-compact}.}
\end{enumerate}
\end{proposition}

\begin{proof}
(\labelcref{item:weak-compact-retracts}) Let $f : X \to Y$ be a
surjection and let $q : Y \to \Two$, and apply the compactness
hypothesis on $X$ to the map $q \circ f$.

For $\exists$-compactness, a root of $q \circ f$ gives a root of $q$,
and so the truncated existence of the former gives that of the latter.
If $q \circ f$ has no root, then neither has $q$, because a root of $q$
has a preimage under $f$ in an unspecified way, and that preimage is a
root of $q \circ f$. This conclusion, being a negation, may be drawn
from the truncated existence.

For $\Pi$-compactness, if $q\,(f\,x) = 1$ for every $x : X$, then
$q\,y = 1$ for every $y : Y$ by surjection induction along $f$, the
equations being propositions, and conversely by composition with $f$.
So the two universal statements are logically equivalent, and
decidability transfers along a logical equivalence.

Images are the case of the corestriction, which is a surjection, and a
retraction is a surjection.

(\labelcref{item:weak-compact-sigma}) Let $p : \Sigma\,Y \to \Two$. For
each $x : X$, the $\Pi$-compactness of the fiber $Y\,x$ decides
$\prod_y p\,(x,y) = 1$, and so there is a map $q : X \to \Two$
with $q\,x = 1$ if and only if $\prod_y p\,(x,y) = 1$. The
$\Pi$-compactness of $X$ decides $\prod_x q\,x = 1$, which is
logically equivalent to $\prod_x \prod_y p\,(x,y) = 1$, and this is
$\prod_{(x,y)} p\,(x,y) = 1$ after currying.

(\labelcref{item:weak-compact-subtype}) Let $S$ be a complemented
subset of $X$, with subtype $\Sigma_{x:X}\,x \in S$.

If $X$ is $\exists$-compact, apply $\exists$-compactness to the
characteristic function of~$S$, which decides whether the subtype is
inhabited. If it is not, then the subtype is empty, and an empty type
is $\exists$-compact. If it is, then a point of the subtype makes it a
retract of $X$, the points outside $S$ being sent to that point, and so
the subtype is $\exists$-compact
by~(\labelcref{item:weak-compact-retracts}). This conclusion may be
drawn from the inhabitedness of the subtype, because
$\exists$-compactness is a proposition
(\Cref{prop:weakly-compact-basic}%
(\labelcref{item:weak-compact-are-props})).

If $X$ is $\Pi$-compact, a function on the subtype extends to a
function on $X$ that takes the value $1$ outside $S$, and the
extension is constantly~$1$ if and only if the given function is. So
the $\Pi$-compactness of $X$ decides whether the extension is
constantly~$1$, and hence whether the given function is.
\end{proof}

\begin{proposition}
\label{prop:weakly-compact-props}
\leavevmode
\begin{enumerate}
\item\label{item:exists-compact-prop-decidable} A proposition is
  $\exists$-compact if and only if it is decidable.
  \AgdaRefs{\Agda{TypeTopology.WeaklyCompactTypes}{∃-compact-propositions-are-decidable};
  \texttt{decidable-propositions-are-∃-compact}.}
\item\label{item:exists-compact-support} The truncation $\|X\|$ of an
  $\exists$-compact type $X$ is decidable.
  \AgdaRefs{\texttt{∃-compact-types-have-decidable-support}.}
\item\label{item:non-empty-exists-compact} A non-empty
  $\exists$-compact type is inhabited.
  \AgdaRefs{\texttt{∃-compact-non-empty-types-are-inhabited}.}
\item\label{item:pi-compact-negation} A $\Pi$-compact type has decidable
  negation.
  \AgdaRefs{\texttt{negations-of-Π-compact-types-are-decidable}.}
\end{enumerate}
\end{proposition}

\begin{proof}
(\labelcref{item:exists-compact-prop-decidable})
and~(\labelcref{item:exists-compact-support}) Apply
$\exists$-compactness to the constantly~$0$ map. Every point of
$X$ is a root of it, so that what is decided is the truncation
$\|X\|$, which gives~(\labelcref{item:exists-compact-support}). When
$X$ is a proposition, the truncation is equivalent to~$X$, and so an
$\exists$-compact proposition is decidable. Conversely, a decidable
proposition is $\exists$-compact, by the argument
of~\Cref{lem:compact-decidable} with $\exists$ in place of~$\Sigma$.

(\labelcref{item:non-empty-exists-compact}) The truncation $\|X\|$ is
decidable by~(\labelcref{item:exists-compact-support}). Its negation is
logically equivalent to $\neg X$, which is ruled out by the
non-emptiness of $X$, and so the truncation holds, which says that $X$
is inhabited.

(\labelcref{item:pi-compact-negation}) Apply $\Pi$-compactness to the
constantly~$0$ map. What is decided is $\prod_x 0 = 1$, which is
logically equivalent to $\neg X$, because the type $0 = 1$ is~empty.
\end{proof}

The following truncated form of
\Cref{lem:compact-pt}(\labelcref{item:universal-witness}), being a
product of propositions, is a proposition.
\begin{proposition}
\label{prop:exists-compact-pt}
The following are equivalent for any type $X$.
\begin{enumerate}
\item\label{item:exists-compact-inhabited} $X$ is $\exists$-compact
  and inhabited.
\item\label{item:truncated-universal-witness} Every $p : \Two^X$ has a
  universal witness in an unspecified way:
  \[
    \textstyle\prod_{p:\Two^X}\, \exists_{x_0:X}\,
    \big(p\,x_0 = 1 \to \forall x.\, p\,x = 1\big).
  \]
\end{enumerate}
Moreover, the type $X$ is $\exists$-compact if and only if it is
empty or satisfies~(\labelcref{item:truncated-universal-witness}).
\AgdaRefs{\Agda{TypeTopology.WeaklyCompactTypes}{is-∃-compact∙};
\texttt{∃-compactness∙-is-prop};
\texttt{∃-compact∙-types-are-inhabited-and-compact};
\texttt{inhabited-and-compact-types-are-∃-compact∙};
\texttt{∃-compact∙-or-empty-types-are-∃-compact};
\texttt{∃-compact-types-are-∃-compact∙-or-empty}.}
\end{proposition}

\begin{proof}
\emph{(\labelcref{item:truncated-universal-witness}) implies
(\labelcref{item:exists-compact-inhabited}).} Taking $p$ constantly~$0$
gives a point of $X$ in an unspecified way, so that $X$ is inhabited.
For $\exists$-compactness, let $p : X \to \Two$ and let $x_0$ be a
universal witness for $p$, available in an unspecified way. Deciding
the value of $p\,x_0$ either produces a root of $p$ or shows that $p$
is constantly~$1$, and so decides $\exists_x p\,x = 0$. This decision
may be made under the truncation, because the decidability of a
proposition is a proposition.

\emph{(\labelcref{item:exists-compact-inhabited}) implies
(\labelcref{item:truncated-universal-witness}).} Let $p : X \to \Two$.
By $\exists$-compactness, either $p$ has a root in an unspecified way,
and that root is a universal witness, the implication holding
vacuously, or $p$ is constantly~$1$, and then any point of $X$ is a
universal witness, one being available in an unspecified way from the
inhabitedness of~$X$.

\emph{The final claim.} An $\exists$-compact type has decidable
truncation by~\Cref{prop:weakly-compact-props}%
(\labelcref{item:exists-compact-support}), and so is either inhabited,
in which case it satisfies~(\labelcref{item:truncated-universal-witness})
by the equivalence just proved, or empty. Conversely, a type
satisfying~(\labelcref{item:truncated-universal-witness}) is
$\exists$-compact by that equivalence, and so is an empty type, no
function on it having a root.
\end{proof}

\begin{proposition}
\label{prop:pi-compact-infs}
A type $X$ is $\Pi$-compact if and only if every $p : X \to \Two$ has
an infimum, that is, a greatest $n : \Two$ such that $n \leq p\,x$ for
every $x : X$.
\AgdaRefs{\Agda{TypeTopology.WeaklyCompactTypes}{has-infs};
\texttt{Π-compact-has-infs}; \texttt{has-infs-Π-compact}; \texttt{inf};
\texttt{inf-property}; \texttt{inf₁}; \texttt{inf₁-converse}.}
\end{proposition}

\begin{proof}
Let $X$ be $\Pi$-compact and let $p : X \to \Two$. Decide whether $p$
is constantly~$1$. If it is, then $1$ is a lower bound of its values,
and it is the greatest element of $\Two$, so that it is the infimum. If
it is not, then $0$ is a lower bound, being the least element of
$\Two$, and $1$ is not a lower bound, as that would make $p$
constantly~$1$, so that $0$ is the infimum.

Conversely, suppose that every $p : X \to \Two$ has an infimum, and
consider the two possible values of $\inf p$. If it is $1$, then
$1 \leq p\,x$, and hence $p\,x = 1$, for every $x$. If it is $0$, then
$p$ is not constantly~$1$, because that would make $1$ a lower bound of
its values and so force $1 \leq \inf p$.
\end{proof}

The infimum operator sends $p$ to a boolean $\inf p$ that is $1$
precisely when $p$ has value $1$ everywhere, so that it defines a
boolean universal quantifier $\mathsf{A} : \Two^X \to \Two$ with
$\mathsf{A}\,p = \inf p$. This is compactness in the sense of
synthetic topology with dominance $\Two$. In fact, writing
$\mathsf{K} : \Two \to \Two^X$ for the map sending a boolean to the
constant function at it, an infimum operator is precisely an order
theoretic right adjoint $\mathsf{A}$ to $\mathsf{K}$, since the
adjunction says that $\mathsf{K}\,n \le p$ pointwise if and
only if $n \leq \mathsf{A}\,p$, which is to say that $\mathsf{A}\,p$
is the greatest lower bound of the values of $p$.
\claim{right-adjoint-characterization}
\AgdaRefs{\Agda{TypeTopology.WeaklyCompactTypes}{Κ⊣-charac};
  \texttt{Π-compact-iff-Κ-has-right-adjoint}.}

In classical topology a space $X$ is compact if and only if every
projection $Y\times X \to Y$ is a closed map. The notion appropriate
here is that of a \emph{clopen} map, one sending complemented subsets
to complemented subsets, with the image of a subset along a map
as in~\Cref{def:image}.
\begin{proposition}
\label{prop:clopen-projections}
A type $X$ is $\exists$-compact if and only if every projection
$Y \times X \to Y$ is a clopen map.
\AgdaRefs{\Agda{TypeTopology.WeaklyCompactTypes}{is-clopen-map};
\texttt{∃-compact-clopen-projections}; \texttt{clopen-projections-∃-compact}.}
\end{proposition}

\begin{proof}
Let $X$ be $\exists$-compact, let $p : Y \times X \to \Two$ be the
characteristic function of a complemented subset, and let $y : Y$.
Membership of $y$ in the image of that subset along the projection
amounts to $\exists_{x:X}\,p\,(y,x) = 0$, which $\exists$-compactness
decides.

Conversely, take $Y$ to be the one-element type and let
$p : X \to \Two$. Consider the subset of $\One \times X$ with
characteristic function $p \circ \mathrm{pr}_2$. Membership of the
point of $\One$ in the image of that subset along the projection
$\One \times X \to \One$ amounts to $\exists_{x:X}\, p\,x = 0$, so that
the clopenness of this single projection gives the
$\exists$-compactness of $X$.
\end{proof}

\subsection{Interplay between the notions}
\label{sec:compact-interplay}

We collect results that combine compactness, discreteness and
total separatedness.
\begin{lemma}
\label{lem:sigma-isolated}
If $Y\,x$ is compact for every $x : X$ and the point $(x,y)$ of
$\Sigma\,Y$ is isolated, then $x$ is isolated.
\AgdaRefs{\Agda{TypeTopology.CompactTypes}{Σ-isolated-left}.}
\end{lemma}

\begin{proof}
Let $x' : X$, so that what has to be decided is $x = x'$. The family
\[
  A\,y' :\equiv \big((x,y) = (x',y')\big)
\]
on $Y\,x'$ is complemented, because $(x,y)$ is isolated, and so the
compactness of $Y\,x'$ decides the sum of $A$. If some $y'$ has
$(x,y) = (x',y')$, then $x = x'$ by taking first components. If none
does, then $x \neq x'$, because an identification $x = x'$ would give
such a $y'$ by transporting $y$ along it.
\end{proof}

\begin{proposition}
\label{prop:compact-to-discrete}
\leavevmode
\begin{enumerate}
\item\label{item:pi-compact-discrete-power} If $X$ is $\Pi$-compact and
  $Y\,x$ is discrete for every $x:X$, then $\prod_x Y\,x$ is discrete.
\item\label{item:compact-apart-or-equal} If $X$ is compact, the
  discreteness of $\prod_x Y\,x$ can be
  strengthened to
  $\big(\Sigma_{x:X}\, f\,x \neq g\,x\big) + \big(f = g\big)$ for all
  $f,g:\prod_x Y\,x$.
\end{enumerate}
\AgdaRefs{\Agda{TypeTopology.CompactTypes}{apart-or-equal};
\texttt{discrete-to-power-compact-is-discrete};
\texttt{discrete-to-power-compact-is-discrete'};
\texttt{discrete-to-power-Compact-is-discrete};
\texttt{discrete-to-power-Compact-is-discrete'}.}
\end{proposition}

\begin{proof}
Given $f, g : \prod_x Y\,x$, the discreteness of each $Y\,x$ decides
$f\,x = g\,x$, so there is an indicator $p : X \to \Two$ with
$p\,x = 0$ if $f\,x \neq g\,x$ and $p\,x = 1$ otherwise.

(\labelcref{item:pi-compact-discrete-power}) $\Pi$-compactness decides
whether $p\,x = 1$ for every $x$, which amounts to deciding $f = g$,
by function extensionality.

(\labelcref{item:compact-apart-or-equal}) Compactness decides whether
$p$ has a root, and a root of $p$ is a point~$x$ with
$f\,x \neq g\,x$, and the absence of one gives $f\,x = g\,x$
everywhere, and hence $f = g$.
\end{proof}

\begin{examples}
\label{ex:discrete-exponentials}
\leavevmode
\begin{enumerate}
\item\label{item:NInf-exponentials-discrete} The types $\Two^{\NInf}$
  and $\N^{\NInf}$ are discrete,
  by~\Cref{prop:compact-to-discrete}(\labelcref{item:compact-apart-or-equal}),
  because $\NInf$ is compact
  (\Cref{ex:compact-pt-basic}(\labelcref{item:NInf-compact})) and
  $\Two$ and $\N$ are discrete (\Cref{ex:isolated}).
  \AgdaRefs{\Agda{TypeTopology.GenericConvergentSequenceCompactness}{ℕ∞→𝟚-is-discrete},
    \texttt{ℕ∞→ℕ-is-discrete}.}
\item\label{item:cantor-discrete-converse} The principle $\WLPO$, the
  $\Pi$-compactness of $\N$
  (\Cref{ex:compact-basic}(\labelcref{item:N-pi-compact})), gives the
  discreteness of the Cantor type $\Cantor$
  by~\Cref{prop:compact-to-discrete}%
  (\labelcref{item:pi-compact-discrete-power}). Together
  with~\Cref{ex:tot-sep}(\labelcref{item:cantor-not-discrete}), this
  makes the discreteness of $\Cantor$ equivalent to $\WLPO$.
\end{enumerate}
\end{examples}

\begin{definition}[Disconnectedness]
\label{def:disconnected}
A type $X$ is \emph{disconnected} if $\Two$ is a retract of it, or,
equivalently, if some $p : X \to \Two$ takes both values.
\AgdaRefs{\Agda{TypeTopology.DisconnectedTypes}{is-disconnected}, which
is \texttt{is-disconnected₀}, with \texttt{is-disconnected₁},
\texttt{is-disconnected₂}, \texttt{is-disconnected₃} and
\texttt{is-disconnected-eq} for equivalent formulations;
\texttt{discrete-types-with-two-different-points-are-disconnected}.}
\end{definition}

\Cref{prop:compact-to-discrete} has a partial converse.

\begin{proposition}
\label{prop:discrete-exponential-gives-compact}
If $Y$ is disconnected and $Y^X$ is discrete, then $X$ is
$\Pi$-compact.
\AgdaRefs{\Agda{TypeTopology.WeaklyCompactTypes}{power-of-two-discrete-gives-compact-exponent};
\texttt{discrete-power-of-disconnected-gives-compact-exponent}.}
\end{proposition}

\begin{proof}
Take first $Y = \Two$. The $\Pi$-compactness of $X$ is the isolatedness
of the constant map $\lambda x.\,1$ in $\Two^X$
(\Cref{prop:weakly-compact-basic}(\labelcref{item:pi-compact-isolated})),
and in a discrete type every point is isolated. A disconnected $Y$
comes with a retraction onto $\Two$, and composing with it and with
its section makes $\Two^X$ a retract of $Y^X$, and discreteness passes
to retracts
(\Cref{prop:discrete-closure}(\labelcref{item:discrete-retract})).
\end{proof}

\begin{proposition}
\label{prop:tot-sep-compact-exponential}
Let $Y$ be disconnected.
\begin{enumerate}
\item\label{item:tscd-disconnected} If $X$ is totally separated and
  $Y^X$ is $\Pi$-compact, then $X$ is discrete.
\item\label{item:tscd-reflection} For arbitrary $X$, if $Y^X$ is
  $\Pi$-compact, then the totally separated reflection
  $\tsreflection X$ of~\Cref{thm:ts-reflection} is discrete.
\end{enumerate}
\AgdaRefs{\Agda{TypeTopology.WeaklyCompactTypes}{tscd}; \texttt{tscd₀};
\texttt{tscd₁}; and, recorded independently in essentially the same
words, \Agda{TypeTopology.CompactTypes}{compact-power-of-𝟚-has-discrete-exponent}.}
\end{proposition}

\begin{proof}
(\labelcref{item:tscd-disconnected}) Take first $Y = \Two$, and let
$x, y : X$. Since $\Two$ is discrete, there is a map
$q : \Two^X \to \Two$ with $q\,p = 1$ when $p\,x = p\,y$ and
$q\,p = 0$ when the two values differ. Applying the $\Pi$-compactness
of $\Two^X$ to $q$ decides $\prod_p q\,p = 1$, which says that $x$ and
$y$ agree at every $\Two$-valued map, that is, that $x =_2 y$. If
$x =_2 y$, then $x = y$ by total separatedness, and if not, then
$x \neq y$, because $x = y$ would give $p\,x = p\,y$ for every $p$. So
$X$ is discrete.

For a disconnected $Y$, a retraction of $Y$ onto $\Two$ makes $\Two^X$
a retract of $Y^X$, and $\Pi$-compactness passes to retracts
(\Cref{prop:weakly-compact-closure}%
(\labelcref{item:weak-compact-retracts})), so that the case just proved
applies.

(\labelcref{item:tscd-reflection}) A retraction of $Y$ onto $\Two$
makes $\Two^X$ a retract of $Y^X$, as
in~(\labelcref{item:tscd-disconnected}), and $\Two^{\tsreflection X}$
is equivalent to $\Two^X$ because $X$ and $\tsreflection X$ have the
same $\Two$-valued maps
(\Cref{thm:ts-reflection}(\labelcref{item:ts-reflection-universal})
with $A = \Two$). So $\Two^{\tsreflection X}$ is a retract of $Y^X$ and
hence $\Pi$-compact
(\Cref{prop:weakly-compact-closure}(\labelcref{item:weak-compact-retracts})).
Now apply~(\labelcref{item:tscd-disconnected}) to $\tsreflection X$,
which is totally separated
(\Cref{thm:ts-reflection}(\labelcref{item:ts-reflection-surjection})).
\end{proof}

\begin{examples}
\label{ex:not-compact}
\leavevmode
\begin{enumerate}
\item\label{item:simple-not-compact} A $\Pi$-compact simple type
  (\Cref{ex:tot-sep}(\labelcref{item:simple-types})) gives $\WLPO$.
  \AgdaRefs{\Agda{TypeTopology.SimpleTypes}{WLPO'}, which is the
  formalization's name for the $\Pi$-compactness of $\N$;
  \texttt{stcwlpo}; \texttt{stcwlpo'}.}
\item\label{item:NInf-power-not-compact} The $\Pi$-compactness of
  $\Two^{\NInf}$ gives $\WLPO$.
  \AgdaRefs{\Agda{TypeTopology.WeaklyCompactTypes}{[ℕ∞→𝟚]-compact-implies-WLPO}.}
\end{enumerate}
\end{examples}

\begin{proof}
(\labelcref{item:simple-not-compact}) The type $\N$ is a retract of
every simple type (\Cref{ex:tot-sep}(\labelcref{item:simple-types})),
and $\Pi$-compactness is inherited by retracts
(\Cref{prop:weakly-compact-closure}(\labelcref{item:weak-compact-retracts})),
so $\N$ is
$\Pi$-compact, which is $\WLPO$
(\Cref{ex:compact-basic}(\labelcref{item:N-pi-compact})).

(\labelcref{item:NInf-power-not-compact}) The type $\NInf$ is totally
separated, being a retract of $\Cantor$ (\Cref{ex:NInf}), which is
totally separated
(\Cref{ex:tot-sep}(\labelcref{item:cantor-not-discrete})), and total
separatedness passing to retracts
(\Cref{prop:tot-sep-closure}(\labelcref{item:tot-sep-retract})), so
\Cref{prop:tot-sep-compact-exponential}%
(\labelcref{item:tscd-disconnected}) turns
the $\Pi$-compactness of $\Two^{\NInf}$ into the discreteness of
$\NInf$, which is $\WLPO$
(\Cref{ex:isolated}(\labelcref{item:NInf-discrete-wlpo})).
\end{proof}

\subsection{The micro-Tychonoff theorem}
\label{sec:micro-tychonoff}

We now formulate and prove the smallest possible instance of Tychonoff's
theorem for our notion of compactness
(cf.~\Cref{rem:tychonoff-independence}) by restricting the index type
to be a subsingleton and the indexed compact types to be pointed. A
crucial application is the compactness of extended sums
(\Cref{thm:extended-sum-compact}), on which the main results
of~\Cref{sec:brouwer-standard,sec:inductive-recursive} are based.

\begin{theorem}[Micro-Tychonoff]
\label{thm:micro-tychonoff}
If $X$ is a proposition and $Y\,x$ is compact pointed for every $x:X$,
then $\prod_x Y\,x$ is compact pointed.
\AgdaRefs{\Agda{TypeTopology.MicroTychonoff}{micro-tychonoff};
\texttt{micro-tychonoff-corollary}; \texttt{micro-tychonoff-corollary'}.}
\end{theorem}
\begin{proof}
Let $X$ be a proposition, let $\selection_x$ be a selection function
witnessing that $Y\,x$ is compact pointed for each $x : X$, and let
$p : \prod_x Y\,x \to \Two$.
For each $x : X$, evaluation at $x$,
\[
  f_x : \textstyle\prod_{x'} Y\,x' \to Y\,x, \qquad f_x\,\varphi = \varphi\,x,
\]
is an equivalence, because a product indexed by a proposition is the
value at any point of that proposition. Carry the map across by
setting $q_x\,y = p\,(f_x^{-1}\,y)$, and define the candidate
witness pointwise:
\[
  \varphi_0\,x \;=\; \selection_x\,q_x .
\]
We claim $\varphi_0$ is a universal witness for $p$. So assume
$p\,\varphi_0 = 1$, and we must show $p\,\varphi = 1$ for an arbitrary
$\varphi$.

\emph{(1)~The conclusion follows from $X$.} Suppose $x : X$. The
universality of $\varphi_0\,x$ for $q_x$ says that
$p\,(f_x^{-1}(\varphi_0\,x)) = 1$ implies $p\,(f_x^{-1}\,y) = 1$ for
every $y : Y\,x$. Since $f_x\,\varphi_0 = \varphi_0\,x$ and $f_x^{-1}$
is a retraction of $f_x$, the hypothesis is our assumption
$p\,\varphi_0 = 1$. Instantiating $y = \varphi\,x$ and using
$f_x^{-1}(\varphi\,x) = \varphi$, we get $p\,\varphi = 1$.
This conclusion no longer mentions the point~$x$, so that
$X \to p\,\varphi = 1$.

\emph{(2)~The conclusion also follows from $\neg X$.} If $X$ is
empty then any two points of $\prod_x Y\,x$ are equal, in particular
$\varphi = \varphi_0$, and so
$p\,\varphi = p\,\varphi_0 = 1$.

\emph{(3)~Combining (1) and (2).} Excluded middle, if it were
available, would now finish the argument, since $X$ is a proposition,
but it is not needed. Suppose $p\,\varphi = 0$. Then
$p\,\varphi \neq 1$, so contraposing~(1) gives $\neg X$, and then
(2)~gives $p\,\varphi = 1$. Thus $p\,\varphi = 0$ implies
$p\,\varphi = 1$. The value of $p\,\varphi$ is $0$ or $1$, and in the
first case it is also $1$, so $p\,\varphi = 1$.
\end{proof}

Specializing to a constant family, this says that $Y^X$ is compact
pointed for any proposition $X$ and any type $Y$ such that the
pointedness of $X$ implies the compact pointedness of $Y$, whether or
not $X$ holds.  \claim{micro-tychonoff-constant-family}
The following example, not needed for our purposes, uses a taboo as a
type in its own right rather than as a statement to be avoided.

\begin{example}
\label{ex:LPO-to-N}
The function type $\N^{\LPO}$ is compact pointed constructively. Indeed $\LPO$ is a
proposition that implies the compact pointedness of $\N$,
so~\Cref{thm:micro-tychonoff} applies. The exponential is compact
whether or not $\LPO$ holds, even though the compactness of $\N$ is
$\LPO$ itself, which is independent of our type theory
(\Cref{ex:compact-basic}(\labelcref{item:N-compact-lpo})).
\AgdaRefs{\Agda{Taboos.LPO}{[LPO→ℕ]-is-compact∙};
  \texttt{[LPO→ℕ]-is-compact}; \texttt{LPO-is-prop};
  \texttt{LPO-gives-compact-ℕ}.}
\end{example}

\begin{remark}
\label{rem:micro-tychonoff}
\leavevmode
\begin{enumerate}
\item The proof of~\Cref{thm:micro-tychonoff} does not go through the
  equivalence $\prod_x Y\,x \simeq Y\,x_0$, which holds once a point
  $x_0 : X$ is given and would transfer compact pointedness directly.
  Such a point is at hand only under the hypothesis that $X$ holds,
  whereas the universal witness has to be produced without deciding
  $X$.
\item\label{item:pointedness-essential} The pointedness hypothesis is
  essential, since the corresponding statement for plain compactness
  fails constructively. Indeed, take $Y\,x :\equiv\nobreak \Zero$, which is
  compact but not pointed. Then $\prod_x Y\,x$ is $\neg X$, so the
  unpointed statement would make $\neg X$ compact, and hence decidable
  (\Cref{lem:compact-decidable}), for every proposition $X$, which is
  the principle of weak excluded middle. Pointedness enters the proof
  of~\Cref{thm:micro-tychonoff} in the definition of $\varphi_0$, which
  has to be a point of $\prod_x Y\,x$ whether or not $X$ holds, and
  plain compactness provides no such point.
  \AgdaRefs{\Agda{TypeTopology.MicroTychonoff}{compact-micro-tychonoff-gives-WEM}.}
\end{enumerate}
\end{remark}

The hypothesis on the index type can be generalized from a proposition
to a subtype of a finite type, because such a subtype is a finite sum
of propositions.

\begin{corollary}
\label{cor:subfinite-tychonoff}
Let $F$ be a finite type and $j : X \hookrightarrow F$ an embedding. If
the type~$Y\,x$ is compact pointed for every $x : X$, then so is
$\prod_x Y\,x$.
\AgdaRefs{\Agda{TypeTopology.MicroTychonoff}{subfinite-tychonoff}, via
\Agda{Fin.Topology}{finitely-indexed-product-compact∙} and
\texttt{finite-product-compact∙}.}
\end{corollary}

\begin{proof}
The embedding decomposes $X$ as the sum of its fibers, and currying the
product over that sum gives
\[
  \textstyle\prod_{x : X} Y\,x \;\simeq\;
  \prod_{i : F}\ \prod_{w : \fib ji} Y\,(\mathrm{pr}_1 w) .
\]
Each fiber is a proposition, $j$ being an embedding, so each inner
product is compact pointed by~\Cref{thm:micro-tychonoff}. The outer
product is a finite product of compact pointed types, which is compact
pointed by~\Cref{cor:compact-pt-closure}(\labelcref{item:pt-sigma}),
and pointed compactness is invariant under equivalence
(\Cref{cor:compact-pt-closure}(\labelcref{item:pt-retract})).
\end{proof}

\subsection{Extension of families of compact types along embeddings}
\label{sec:extension}

The above closure properties of compact types allow us to construct
new compact types starting from examples such as finite types and
$\NInf$. However, they do not reach far. To construct more compact
types, we develop a generalization of a certain \emph{squashed sum}
constructed in~\cite{EscardoOmniscientJSL2013}, with the aid of
the theory of algebraically injective types developed
in~\cite{EscardoInjective2021}. It works by \emph{extending} a family
of types along an embedding and then summing the extended family. The
extension operation witnesses the algebraic injectivity of the
universe, and the sum is compact because the fibers of an embedding
are propositions, so that the
micro-Tychonoff~\Cref{thm:micro-tychonoff} is applicable.

\begin{definition}[$\Pi$-extension]
\label{def:extension}
Let $Y : X \to \UU$ be a family of types and $j : X \to K$ an arbitrary
map. The \emph{$\Pi$-extension} of $Y$
along $j$ is the family $Y \Pext j : K \to \UU$ given by
\[
  (Y \Pext j)\,k \;:\equiv\; \prod_{(x,p)\,:\,\fib{j}{k}} Y\,x,
\]
where $\fib{j}{k}$ is the fiber of $j$ over $k$.
\AgdaRefs{\Agda{InjectiveTypes.Blackboard}{Π-extension}; \texttt{\_/\_};
\texttt{2nd-Π-extension-formula}.}
\end{definition}
This is illustrated by the following diagrams, with the general case on
the left, used in~\Cref{sec:inductive-recursive}, and the special case
of the embedding $\N \hookrightarrow \NInf$ on the right, used
in~\Cref{sec:brouwer-standard}:
\[
\begin{tikzcd}[row sep=large, column sep=large]
  X \arrow[r, hook, "j"]       \arrow[dr, "Y"'] & K     \arrow[d, dashed, "Y \Pext j"] \\
  & \UU
\end{tikzcd}
\qquad
\begin{tikzcd}[row sep=large, column sep=large]
  \N \arrow[r, hook, "\iota"] \arrow[dr, "Y"'] & \NInf \arrow[d, dashed, "Y \Pext \iota"] \\
  & \UU.
\end{tikzcd}
\]
The point of the above construction is that we can extend $Y$ from
$\N$ to $\NInf$ \emph{without} checking whether the argument of the
desired extension $Y \Pext \iota$ is $\infty$ or not, which is not
available in our constructive setting, as it amounts to $\WLPO$
(\Cref{ex:isolated}(\labelcref{item:infty-isolated-wlpo})).
\begin{lemma}
\label{lem:extension-property}
Let $Y : X \to \UU$ and $j : X \to K$.
\begin{enumerate}
\item\label{item:ext-restricts} If $j$ is an embedding, the
  $\Pi$-extension restricts back to $Y$ along $j$ up to type
  equivalence: for every $x : X$ we have
  $(Y \Pext j)(j\,x) \simeq Y\,x$.
\item\label{item:ext-off-image} If $j\,x \neq k$ for every $x:X$, then
  $(Y \Pext j)\,k \simeq \One$.
\end{enumerate}
\AgdaRefs{\Agda{InjectiveTypes.Blackboard}{Π-extension-property};
\texttt{Π-extension-out-of-range}.}
\end{lemma}

\begin{proof}
(\labelcref{item:ext-restricts}) If $j$ is an embedding then
$\fib{j}{j\,x}$ is a proposition, and it contains the point
$(x,\mathrm{refl})$, so it is a singleton. A product indexed by a
singleton is equivalent to its value at the center, whence $(Y \Pext
j)(j\,x) \simeq Y\,x$.

(\labelcref{item:ext-off-image}) If $j\,x \neq k$ for every $x : X$
then $\fib jk$ is empty, and a product indexed by the empty type is
equivalent to the type~$\One$.
\end{proof}

We refer to $\Sigma (Y \Pext j)$ as the \emph{extended sum} of $Y$
along~$j$.
\begin{theorem}[Compactness of extended sums]
\label{thm:extended-sum-compact}
For any embedding $j : X \to K$ into a compact pointed type and
any family $Y : X \to \UU$ of compact pointed types, the extended sum
$\Sigma (Y \Pext j)$ is compact pointed.
\AgdaRefs{\Agda{TypeTopology.ExtendedSumCompact}{extended-sum-compact∙}.}
\end{theorem}

\begin{proof}
By the closure of pointed compactness under $\Sigma$
(\Cref{cor:compact-pt-closure}(\labelcref{item:pt-sigma})) applied to
the family $Y \Pext j$ over the compact pointed base $K$, it suffices to
show that $(Y \Pext j)\,k$ is compact pointed for every $k : K$. By
definition this is the product of the types $Y\,x$ over the index type
$\fib jk$, all of which are compact pointed by hypothesis, and $\fib jk$
is a proposition precisely because $j$ is an embedding.
So the micro-Tychonoff~\Cref{thm:micro-tychonoff} applies.
\end{proof}

\begin{remark}
\Cref{thm:extended-sum-compact} is a type theoretic generalization of
the fact that the squashed sum constructed
in~\cite[Section~10]{EscardoOmniscientJSL2013} produces compact
sets. In \emph{loc.~cit.}, given countably many compact pointed
subsets $X_n \subseteq \Cantor$, we rescaled $X_n$ by prefixing the
finite sequence $1^n0$, making the pieces disjoint, smaller in
diameter, close together, and converging to the point
$\infty = 1^\omega$ as $n$ increases, and defined their squashed sum
to be the set of those $\alpha$ such that if $1^n0$ is a prefix of
$\alpha$ then $\alpha \in 1^n0X_n$. This set is the closure of
$\bigcup_n 1^n0X_n$, which classically is this union together with
$\{\infty\}$. Constructively, it can be described as the
union of the $\NInf$-indexed family
\[
  Y_x = \{\alpha \in \Cantor \mid \forall n \in \N\,(x = n \Rightarrow
  \alpha \in 1^n0X_n) \text{ and } (x = \infty \Rightarrow \alpha = \infty)\}.
\]
The defining formula for $Y_x$ corresponds to the $\Pi$-extension
formula of~\Cref{def:extension}, for the family $n \mapsto 1^n0X_n$
along the embedding $\N \hookrightarrow \NInf$. Thus, the extended sum
generalizes the squashed sum to any embedding $X \hookrightarrow K$
into a compact pointed type and any $X$-indexed family of compact
pointed types.
\end{remark}

\begin{lemma}
\label{lem:extension-tot-sep}
For any map $j : X \to K$, not necessarily an embedding,
and any family $Y : X \to \UU$ of totally separated types,
the type $(Y \Pext j)\,k$ is totally separated for every $k : K$.
\AgdaRefs{\Agda{TypeTopology.ExtensionTotallySeparated}{/-is-totally-separated}.}
\end{lemma}

\begin{proof}
The type $(Y \Pext j)\,k$ is the product of the types $Y\,x$ over the
index type $\fib jk$, each of which is totally separated by hypothesis,
and total separatedness is closed under arbitrary products
(\Cref{prop:tot-sep-closure}(\labelcref{item:tot-sep-pi})).
\end{proof}

\begin{lemma}
\label{lem:extension-retract}
For any map $j : X \to K$ and any families $Y, Z : X \to \UU$ such
that $Y\,x$ is a retract of $Z\,x$ for every $x : X$,
\begin{enumerate}
\item\label{item:ext-retract-fiber} the type $(Y \Pext j)\,k$ is a
  retract of $(Z \Pext j)\,k$ for every $k : K$,
\item\label{item:ext-retract-sum} the extended sum $\Sigma\,(Y \Pext j)$
  is a retract of the extended sum $\Sigma\,(Z \Pext j)$.
\end{enumerate}
\AgdaRefs{\Agda{InjectiveTypes.Blackboard}{retract-extension};
  \Agda{UF.Retracts}{Σ-retract}.}
\end{lemma}

\begin{proof}
(\labelcref{item:ext-retract-fiber}) By~\Cref{def:extension}, the two
types are products over the index type $\fib jk$, and a product of
retractions is a retraction of the corresponding products, the
retraction and the section acting pointwise.
(\labelcref{item:ext-retract-sum}) follows
by~\Cref{lem:sum-map}(\labelcref{item:sum-map-retraction}), applied to
the identity on $K$ and to the retractions
of~(\labelcref{item:ext-retract-fiber}).
\end{proof}

Classically, every point of $\Cantor$ is of one of the forms
$1^n0\alpha$, with $\alpha : \Cantor$, or~$1^\omega$.  The following
gives a constructive version of this.

\begin{lemma}[Delayed binary sequences]
\label{lem:delayed-sequences} \leavevmode
\begin{enumerate}
\item \label{item:delayed-sequences-1}
The extended sum of the constant family at $\Cantor$ along the
embedding $\N\hookrightarrow\NInf$ is the type of \emph{delayed
binary sequences}, consisting of a time $u : \NInf$ together with a
binary sequence made available when $u$ is finite,
\[
  \textstyle\Sigma_{u\,:\,\NInf}\,
  \big(\text{$u$ is finite} \to \Cantor\big).
\]
\item \label{item:delayed-sequences-2}
This type is equivalent to $\Cantor$.
\end{enumerate}
\AgdaRefs{\Agda{TypeTopology.DelayMonad}{𝔻};
  \Agda{TypeTopology.SquashedCantor}{Head},
  \texttt{Tail}, \texttt{Cons}, \texttt{Cons⁻¹},
  \texttt{Cons⁻¹-Cons}, \texttt{Cons-Cons⁻¹},
  \texttt{𝔻-Cantor-≃-Cantor}.}
\end{lemma}

\begin{proof}
(\labelcref{item:delayed-sequences-1}) By~\Cref{def:extension}, the
extension of the constant family at $\Cantor$ along the embedding
$\iota : \N \hookrightarrow \NInf$ evaluates at a point $u : \NInf$ to
the type of functions from the fiber $\fib{\iota}{u}$ to $\Cantor$, so
that the extended sum is the sum of these function types over
$u : \NInf$. A point of the fiber is a natural number $n$ together
with an identification $\underline n = u$, that is, a witness of the
finiteness of $u$ (\Cref{ex:NInf}).

(\labelcref{item:delayed-sequences-2}) In one direction, we define a map
\[
  \Cons \;:\; \big(\textstyle\Sigma_{u\,:\,\NInf}\,
  \big(\text{$u$ is finite} \to \Cantor\big)\big) \to \Cantor.
\]
A delayed binary sequence is encoded as the run of $1$s prescribed by
its time, followed by a $0$ and the delayed value. A conatural number
is either $\underline 0$ or the result of prefixing a~$1$ to a
conatural number, decided by its first digit, and a binary sequence is
a digit followed by a binary sequence, so that corecursion on
$\Cantor$ defines $\Cons$ by the two clauses: the sequence
$\Cons(\underline 0, \pi)$ begins with~$0$ and continues as $\pi$
applied to the witness that $\underline 0$ is finite, and, for $u$
prefixing a~$1$ to $v$, the sequence $\Cons(u,\pi)$ begins with~$1$
and continues as $\Cons(v,\pi')$, where $\pi'$ is $\pi$ composed with
the passage from finiteness of $v$ to finiteness of~$u$.

In the other direction, we first define two maps
\[
\begin{array}{lll}
  \Head & : & \Cantor \to \NInf \\
  \Tail & : & \textstyle\prod_{\alpha:\Cantor}\,
                \big(\text{$\Head\,\alpha$ is finite} \to \Cantor\big)
\end{array}
\]
as follows.  A binary sequence is turned into a delayed one by taking
its initial run of $1$s as the time and what follows as the delayed
value. For a sequence~$\alpha$, let $\Head\,\alpha$ be that initial
run, read as a conatural number and defined by the final-coalgebra
presentation of~\Cref{ex:NInf}, as $\underline 0$ when $\alpha$
begins with~$0$, and as the result of prefixing a $1$ to the head of
the tail of $\alpha$ when $\alpha$ begins with~$1$. When
$\Head\,\alpha$ is
finite, witnessed by $n : \N$ together with an identification $p$ of
$\underline n$ with $\Head\,\alpha$, the delayed value is the suffix
of $\alpha$ after that run,
\[
  \Tail\,\alpha\,(n,p)\,k \;=\; \alpha\,(k{+}n{+}1).
\]
The two together give the map
\[
  \Consinv\,\alpha = (\Head\,\alpha, \Tail\,\alpha)
\]
from $\Cantor$ to the type of delayed binary sequences.

\emph{The composite $\Consinv \circ \Cons$ is the identity.} Its first
component asks for $\Head(\Cons(u,\pi)) = u$, which holds for
$u = \underline n$ by induction on $n$ from the two clauses, and for
$u = \infty$ because $\Cons(\infty,\pi)$ begins with~$1$ and continues
as $\Cons(\infty,\pi')$, where $\pi' = \pi$, both being maps out of the
empty type of witnesses of the finiteness of $\infty$, so that
$\Head(\Cons(\infty,\pi))$ is a fixed point of prefixing a~$1$, and
$\infty$ is the unique such point. This
proves the two cases $u = \underline n$ and $u = \infty$. These
suffice by~\Cref{lem:dense-into-nn-separated-rc}, applied to the two
functions of $u : \NInf$ that send $\pi$ to $\Head(\Cons(u,\pi))$
and to $u$, because the map $\N + \One \to \NInf$ is extremely dense
(\Cref{ex:isolated}(\labelcref{item:iota1-embedding-dense})) and the
values of the two functions at $u$ lie in a $\neg\neg$-separated
type. That type is the type of functions from
$(\text{$u$ is finite} \to \Cantor)$ to
$\NInf$, which is totally separated, since $\NInf$ is, being a
retract of $\Cantor$ (\Cref{ex:NInf},
\Cref{ex:tot-sep}(\labelcref{item:cantor-not-discrete}),
\Cref{prop:tot-sep-closure}(\labelcref{item:tot-sep-retract},
\labelcref{item:tot-sep-pi})), and totally separated types are
$\neg\neg$-separated
(\Cref{lem:tot-sep-basic}(\labelcref{item:ts-gives-nnsep})). The
second component asks that $\Tail(\Cons(u,\pi))$, transported along
the identification just established, be $\pi$, and this follows from
the two clauses by cases on the finiteness witness, the case of a
positive witness using that
$\Cons(\underline{n}, \pi)$ is $\pi$ applied to the witness of
finiteness of $\underline{n}$ once its first $n{+}1$ digits are
removed.

\emph{The composite $\Cons \circ \Consinv$ is the identity as well.} When
$\alpha$ begins with~$0$, the head of $\alpha$ is $\underline 0$, and
the first clause gives $\Cons(\Consinv\,\alpha) = \alpha$, because the
finiteness witness obtained by transporting along
$\Head\,\alpha = \underline 0$ has size $0$, which makes $\Tail\,\alpha$
at that witness the tail of $\alpha$. When $\alpha$ begins with~$1$,
the head of $\alpha$ prefixes a $1$ to the head of the tail of
$\alpha$, and the second clause gives that $\Cons(\Consinv\,\alpha)$
begins with~$1$ and continues as $\Cons(\Consinv(\text{the tail of
$\alpha$}))$, the transported witness again having the size that makes
$\Tail\,\alpha$ agree with $\Tail$ of the tail of $\alpha$.
Induction on the digit index, with~$\alpha$ generalized, settles all
digits, the first case directly, and the second by reducing the digit
of index $k{+}1$ to the digit of index $k$ of the tail. The
two composites being the identity, the map $\Cons$ is an equivalence.
\end{proof}

\begin{remark}
\label{rem:delay-monad}
Notice that, more generally, for any type $X$, the extended sum of the
constant family at $X$ along the embedding $\N\hookrightarrow\NInf$
is the type $ \textstyle\Sigma_{u\,:\,\NInf}\, \big(\text{$u$ is
  finite} \to X\big).$ This type satisfies the universal property of
the delay monad~\cite{Capretta_2005}, and, moreover, $\Cantor$ is the
final coalgebra of the functor part of this monad, but we will
not prove these two facts here. Hence the extended sum
along $\N\hookrightarrow\NInf$ of an arbitrary countable family
$X : \N \to \UU$ generalizes the delay monad, the type $X\,n$ of the
result being allowed to depend on the time $u = \underline n$ at which
it arrives.
\AgdaRefs{\Agda{TypeTopology.DelayMonad}{𝔻-is-final-coalgebra};
  \Agda{TypeTopology.SquashedCantor}{Cantor-is-final-𝔻-coalgebra}.}
\end{remark}

\begin{example}
\label{ex:bounded-decreasing-sequences}
For every natural number~$k$, the type of decreasing sequences of
natural numbers bounded by~$k$,
\[
  B_k \;:\equiv\; \textstyle\Sigma_{\beta\,:\,\N^\N}\,
  \big(\prod_{i:\N}\, \beta\,i \le k\big) \times
  \big(\prod_{i:\N}\, \beta\,(i{+}1) \le \beta\,i\big),
\]
is compact and totally separated.
\AgdaRefs{\Agda{TypeTopology.BoundedDecreasingSequences}{𝔹},
  \texttt{is-bounded-by}, \texttt{is-decreasing'};
  \texttt{𝔹-is-compact∙}; \texttt{𝔹-is-totally-separated}.}
\end{example}

\begin{proof}
We show that $B_k$ is equivalent to the type $D^k\,\One$, where $D$
sends a type~$X$ to the extended sum of the constant family at~$X$
along the embedding $\N \hookrightarrow \NInf$, and that $D^k\,\One$ is
compact pointed and totally separated. This suffices, because pointed
compactness and total separatedness are closed under equivalence
(\Cref{cor:compact-pt-closure}(\labelcref{item:pt-retract}),
\Cref{prop:tot-sep-closure}(\labelcref{item:tot-sep-retract})). We
first establish the properties of $D$ that we need, and then construct
equivalences $B_0 \simeq \One$ and $B_{k+1} \simeq D\,B_k$, which give
$B_k \simeq D^k\,\One$ by induction on~$k$, since $D$ preserves
equivalences.
Explicitly, for a type~$X$, we define
\[
  D\,X \;:\equiv\; \textstyle\Sigma_{u\,:\,\NInf}\,
  \big(\text{$u$ is finite} \to X\big).
\]
If $X$ is compact pointed, then so is~$D\,X$,
by~\Cref{thm:extended-sum-compact}, because $\NInf$ is compact pointed
(\Cref{ex:compact-pt-basic}(\labelcref{item:NInf-compact})). If $X$ is
totally separated, then so is~$D\,X$, by~\Cref{lem:sigma-NInf-tot-sep},
because the type of functions from the finiteness of~$u$ to~$X$ is
totally separated for every $u : \NInf$ (\Cref{lem:extension-tot-sep})
and is a proposition for $u = \infty$, since $\infty$ is not finite.
Hence $D^k\,\One$ is compact pointed and totally separated for
every~$k$, by induction on~$k$. An equivalence $e : X \simeq Y$ gives
an equivalence $D\,X \simeq D\,Y$, sending $(u,\pi)$ to
$(u, e \circ \pi)$.
\AgdaRefs{\Agda{TypeTopology.DelayMonad}{𝔻};
  \Agda{TypeTopology.SquashedSum}{Σ¹-is-totally-separated};
  \Agda{TypeTopology.BoundedDecreasingSequences}{𝔻-preserves-compactness∙},
  \texttt{𝔻-preserves-total-separatedness}, \texttt{𝓑},
  \texttt{𝓑-is-compact∙}, \texttt{𝓑-is-totally-separated},
  \texttt{𝔻-preserves-≃}.}

The two conditions on~$\beta$ in the definition of $B_k$ are
propositions, so two points of $B_k$ with the same sequence are equal,
and we denote a point of $B_k$ by its sequence. The type $B_0$ is a
singleton, its only point being the constantly~$0$ sequence, and so
$B_0 \simeq \One$.
\AgdaRefs{\Agda{TypeTopology.BoundedDecreasingSequences}{𝔹-base}.}

It remains to construct an equivalence $B_{k+1} \simeq D\,B_k$. In
one direction, we define a map $f : B_{k+1} \to D\,B_k$. Let
$\beta : B_{k+1}$, and let $u$ be the binary sequence with $u_i = 1$
if $\beta\,i = k{+}1$ and $u_i = 0$ otherwise. It is decreasing, and
hence a point of~$\NInf$, because if $u_{i+1} = 1$ then
$k{+}1 = \beta\,(i{+}1) \le \beta\,i \le k{+}1$, so that $u_i = 1$.
Now let $(n,p)$ witness the finiteness of~$u$, so that $p$ identifies
$\underline n$ with~$u$. Then $u_n = 0$, because $\underline n$ has
digit~$0$ at~$n$ (\Cref{ex:NInf}), so $\beta\,n \neq k{+}1$, and hence
$\beta\,n \le k$. The sequence $i \mapsto \beta\,(n{+}i)$ is
decreasing and bounded by~$k$, because $\beta$ is decreasing, so that
$\beta\,(n{+}i) \le \beta\,n \le k$. This defines a function $\pi$
from the finiteness of~$u$ to~$B_k$, and we set
$f\,\beta = (u,\pi)$.

In the other direction, we define a map $g : D\,B_k \to B_{k+1}$. Let
$(u,\pi) : D\,B_k$. When $u_i = 0$, the point $u$ is finite
(\Cref{ex:NInf}), with a witness $\varphi$ identifying $\underline m$
with~$u$ for some~$m$, and $m \le i$, because $\underline m$ has
digit~$1$ at every index below~$m$. Define a sequence $\beta$ by
\[
  \beta\,i \;=\;
  \begin{cases}
    k{+}1 & \text{if } u_i = 1, \\
    \pi\,\varphi\,(i{-}m) & \text{if } u_i = 0.
  \end{cases}
\]
It is bounded by $k{+}1$, because $\pi\,\varphi$ is bounded by~$k$. It
is decreasing, by cases on the digits $u_{i+1}$ and~$u_i$. If
$u_{i+1} = 1$ then $u_i = 1$, because $u$ is decreasing, and
$\beta\,(i{+}1) = k{+}1 = \beta\,i$. If $u_{i+1} = 0$ and $u_i = 1$
then $\beta\,(i{+}1) \le k < \beta\,i$. If $u_{i+1} = u_i = 0$ then
the witnesses of finiteness used at $i$ and at $i{+}1$ are equal, the
finiteness of~$u$ being a proposition, and so
$\beta\,(i{+}1) = \pi\,\varphi\,(i{+}1{-}m) \le \pi\,\varphi\,(i{-}m)
= \beta\,i$, because $\pi\,\varphi$ is decreasing. We set
$g\,(u,\pi) = \beta$.

\emph{The composite $g \circ f$ is the identity.} Let
$f\,\beta = (u,\pi)$ and $\beta' = g\,(u,\pi)$. If $u_i = 1$ then
$\beta'\,i = k{+}1 = \beta\,i$, by the definition of~$u$. If $u_i = 0$
then, with $\varphi$ and $m$ as in the definition of~$g$,
$\beta'\,i = \pi\,\varphi\,(i{-}m) = \beta\,(m + (i{-}m)) = \beta\,i$,
by the definition of~$\pi$.

\emph{The composite $f \circ g$ is the identity as well.} Let
$\beta = g\,(u,\pi)$ and $f\,\beta = (u',\pi')$. If $u_i = 1$ then
$\beta\,i = k{+}1$, and if $u_i = 0$ then $\beta\,i \le k$, so that
$u'_i = u_i$ in both cases. A point of $\NInf$ is determined by its
digits (\Cref{ex:NInf}), and so $u' = u$. It remains to show that the
transport of $\pi'$ along this identification is~$\pi$. Let $(n,p)$
witness the finiteness of~$u$. At $(n,p)$, the transported function
takes the value of $\pi'$ at a witness of the finiteness of $u'$ with
the same number~$n$, which is the point of $B_k$ with sequence
$i \mapsto \beta\,(n{+}i)$, by the definition of~$f$. Since
$u = \underline n$, we have $u_{n+i} = 0$, and the witness used in the
definition of $\beta\,(n{+}i)$ is $(n,p)$, the finiteness of~$u$ being
a proposition, so that $\beta\,(n{+}i) = \pi\,(n,p)\,i$ for every~$i$,
and the two functions agree at~$(n,p)$.
\AgdaRefs{\Agda{TypeTopology.BoundedDecreasingSequences}{𝔹-step};
  \texttt{𝔹-and-𝓑-are-equivalent}.}
\end{proof}

\section{Equipping compact types with well-orders}
\label{sec:ordinals-prelim}

The compact types we construct in this paper all come equipped with
well-orders, in the sense of~\cite[Section 10.3]{HoTTBook}, because the
operations used to build them, discussed in the previous section,
preserve well-orders.

\subsection{Ordinals in type theory}
\label{sec:ordinals-in-type-theory}

Let ${\lt{}}$ be a binary relation on a type $X$. We say that an
element $x : X$ is \emph{accessible} whenever every $y \lt{} x$ is
accessible. In particular, if an element has no predecessors, then it
is vacuously accessible. This is an inductive definition with a
constructor of type
\[
   \textstyle\prod_{x : X} \big(\textstyle\prod_{y : X} (y \lt{} x \to \text{$y$ is accessible})\big) \to \text{$x$ is accessible},
 \]
which we leave unnamed.
We say that the relation ${\lt{}}$ is \emph{well-founded} if every
element is accessible.
\AgdaRefs{\Agda{Ordinals.Notions}{is-accessible}, \texttt{acc},
  \texttt{accessible-induction}; \texttt{is-well-founded}.}

\begin{lemma}
\label{lem:accessibility-transfinite-induction}
A binary relation ${\lt{}}$ on a type $X$ is well-founded if and only if
the principle of transfinite induction holds: for every family
$P : X \to \UU$,
\[
  \big(\textstyle\prod_{x:X}
    \big(\textstyle\prod_{y:X} y \lt{} x \to P\,y\big) \to P\,x\big)
  \to \textstyle\prod_{x:X} P\,x .
\]
\AgdaRefs{\Agda{Ordinals.Notions}{transfinite-induction} and
  \texttt{transfinite-induction-converse}.}
\end{lemma}
\begin{proof}
The accessibility elimination rule turns
a step \[ \big(\textstyle\prod_{y:X} (y \lt{} x \to P\,y)\big) \to P\,x \]
into a witness of $P\,x$ for every accessible $x$, so accessibility of
every element gives transfinite induction for every family
$P : X \to \UU$. Conversely, the constructor of the accessibility
predicate is a step for the family $P$ that sends $x$ to the type
expressing that $x$ is accessible, so transfinite induction applied to
this family gives accessibility of every element.
\end{proof}

We say that the relation ${\lt{}}$ is \emph{extensional} if any two
elements with the same predecessors are equal:
$\big(u \lt{} x \leftrightarrow u \lt{} y\text{ for all }u\big) \to x
= y$. An \emph{ordinal} is a type $X$ equipped with a
proposition valued binary relation ${\lt{}} : X \to X \to \UU$ which
is transitive, well-founded and extensional.  Such a relation is
called a \emph{well-order} on $X$. We write $\Ord$ for the type of
ordinals with underlying type in $\UU$.
\claim{ordinal-well-order}
\AgdaRefs{\Agda{Ordinals.Notions}{is-well-order};
  \Agda{Ordinals.Type}{Ordinal}.}  The HoTT Book additionally requires
the type~$X$ to be a set, but this is automatic:

\begin{lemma}
\label{lem:ext-gives-set}
Any type equipped with a proposition valued extensional
relation is a set.
\AgdaRefs{\Agda{Ordinals.Notions}{extensionally-ordered-types-are-sets};
used via \Agda{Ordinals.Type}{underlying-type-is-set}.}
\end{lemma}
\begin{proof}
  Let ${\lt{}}$ be a proposition valued extensional relation on a type
  $X$. Consider the derived relation
  $x \preccurlyeq y :\equiv \prod_z (z \lt{} x \to z \lt{} y)$, that
  every predecessor of $x$ is a predecessor of $y$. It is reflexive,
  and it is proposition valued because $\lt{}$ is. Extensionality of
  $\lt{}$ says exactly that it is antisymmetric, since
  $x \preccurlyeq y$ and $y \preccurlyeq x$ together say that $x$ and
  $y$ have the same predecessors. Now a type carrying a
  reflexive, antisymmetric, proposition valued relation is a set, since
  the relation lets us define, from an identification $x = y$, a weakly
  constant endofunction of $x = y$: transport the reflexivity of
  $\preccurlyeq$ along the identification to get $x \preccurlyeq y$ and
  $y \preccurlyeq x$, and apply antisymmetry, the constancy coming from
  $\preccurlyeq$ being proposition valued. So Hedberg's~\Cref{thm:hedberg}
  applies.
  \AgdaRefs{\Agda{Ordinals.Notions}{extensional-po-is-prop-valued};
    \Agda{UF.HedbergApplications}{type-with-prop-valued-refl-antisym-rel-is-set}.}
\end{proof}

\begin{examples}
  \label{ex:standard-ordinals}
  \leavevmode
  \begin{enumerate}
  \item The type of natural numbers under its natural order is
    well-ordered, giving rise to an ordinal written $\omega$.
  \item Any proposition ordered by the relation with empty graph is an
    ordinal, and in particular we get ordinals $\ZeroO$ and $\OneO$.
  \item The two-element type $\TwoO$ ordered by $0 \lt{} 1$ is an
    ordinal.
  \item The type $\OmegaO$ of propositions is an ordinal under the
    order $p \lt{} q$ when $p = \bot$ and $q = \top$.
\end{enumerate}
\AgdaRefs{\Agda{Ordinals.Arithmetic}{ω}, \texttt{prop-ordinal},
  \texttt{𝟘ₒ}, \texttt{𝟙ₒ}, \texttt{𝟚ₒ};
  \Agda{Ordinals.OrdinalOfTruthValues}{Ωₒ}.}
\end{examples}
\begin{notation}
\label{not:ordinals}
An ordinal is written with the symbol of its underlying type, so that
$x : \alpha$ is a point of the underlying type of the
ordinal~$\alpha$, a map $\alpha \to \beta$ is a map between the
underlying types, and a notion such as compactness is a notion about
the underlying type. This is the same convention by which a group is
identified by its carrier. The order of an ordinal is written
${\lt{}}$, the same symbol for every ordinal, with a subscript naming
the ordinal only when several orders occur together, as in the
definitions of the arithmetic operations below. The subscript on
$\Ord[]$, and on the variants of it introduced later, names the universe
of the underlying types, as it does on $\Omega_{\UU}$, and is left out
where the context determines it.  \AgdaRefs{the
  underlying type of $\alpha$ is \texttt{⟨ α ⟩}.}
\end{notation}

By univalence, the type $\Ord[]$ is itself an ordinal in the next
universe, under the relation
\[
  \alpha \olt \beta \;:\equiv\; \Sigma_{b : \us\beta}\, \alpha =
  \down\beta b ,
\]
where $\down\beta b$ is the initial segment (lower set, or downward
closure) of $\beta$ at~$b$, which is again an ordinal, with the points
of $\beta$ below $b$ as its underlying type and the order inherited
from $\beta$:
\[
  \down\beta b \;:\equiv\; \textstyle\Sigma_{y : \us\beta}\, y \lt\beta b,
  \qquad
  (y,l) \lt{} (z,m) \;:\equiv\; y \lt\beta z .
\]
In particular, being well-ordered, $\Ord[]$ is a set
by~\Cref{lem:ext-gives-set}. Passing to initial segments preserves the
order, in the sense that $y \lt\beta y'$ implies
$\down\beta y \olt \down\beta{y'}$, and is left cancellable, in the
sense that $\down\beta y = \down\beta{y'}$ implies $y = y'$.
\claim{ordinal-of-ordinals}
\AgdaRefs{\Agda{Ordinals.OrdinalOfOrdinals}{OO}; the relation $\olt$
  is \texttt{\_⊲\_} and the initial segment is \texttt{\_↓\_}, both
  defined loc.\ cit.; \texttt{↓-preserves-order}, \texttt{↓-lc}.} We
also use the
\emph{initial-segment embedding} order
\[
  \alpha \oemb \beta \;:\equiv\; \Sigma_{f:\us\alpha\to\us\beta}\,
  \text{$f$ is a simulation of $\alpha$ into $\beta$},
\]
where a \emph{simulation} of $\alpha$ into $\beta$ is a map
$f : \us\alpha \to \us\beta$ that is \emph{order-preserving} and
\emph{initial-segment-preserving}, the latter meaning that every point
of $\beta$ below $f\,x$ is the image of a point of $\alpha$ below $x$:
\[
  x \lt\alpha y \to f\,x \lt\beta f\,y
  \qquad\text{and}\qquad
  z \lt\beta f\,x \to \textstyle\Sigma_{x' : \us\alpha}\,
  (x' \lt\alpha x) \times (f\,x' = z) .
\]
The two conditions together say that $f$ maps the initial segment of
$x$ onto the initial segment of $f\,x$. Moreover, $\alpha \olt \beta$
implies
$\alpha \oemb \beta$, while $\alpha \olt \beta \oemb \gamma$ implies
$\alpha \olt \gamma$. There is at most one simulation of $\alpha$ into
$\beta$, so that $\alpha \oemb \beta$ is a proposition. Two ordinals
$\alpha,\beta$ are \emph{order equivalent}, written
$\alpha \iseq \beta$, when there is an
order-preserving equivalence $\us\alpha \simeq \us\beta$ with
order-preserving inverse. This relation is a proposition, as there is
at most one such equivalence.
\AgdaRefs{\Agda{Ordinals.OrdinalOfOrdinals}{\_⊴\_},
  \texttt{⊲-⊴-gives-⊲};
  \Agda{Ordinals.Maps}{is-simulation}, \texttt{is-order-preserving},
  \texttt{is-initial-segment}, \texttt{at-most-one-simulation};
  \Agda{Ordinals.OrdinalOfOrdinals}{⊴-is-prop-valued}.}
\claim{ordinal-comparisons}
\AgdaRefs{\Agda{Ordinals.Equivalence}{\_≃ₒ\_},
  \texttt{≃ₒ-is-prop-valued}.}

\begin{lemma}
\label{lem:surjective-simulations-are-order-equivs}
If $f : \us\alpha \to \us\beta$ is a surjective simulation, then $f$
is an equivalence with order-preserving inverse.
\AgdaRefs{\Agda{Ordinals.Equivalence}{surjective-simulations-are-order-equivs}.}
\end{lemma}

\begin{proof}
A simulation is left cancellable, which is proved by transfinite
induction on the two points
(\Cref{lem:accessibility-transfinite-induction}). Let $f\,x = f\,x'$
and $u \lt{} x$. Order preservation gives $f\,u \lt{} f\,x'$, and
initial-segment preservation gives $v \lt{} x'$ with $f\,v = f\,u$, so
$u = v \lt{} x'$ by the induction hypothesis. Symmetrically, every
point below $x'$ is below $x$, and extensionality identifies $x$
with~$x'$. The underlying type of an ordinal is a set
by~\Cref{lem:ext-gives-set}, and a left-cancellable map into a set is
an embedding (\Cref{def:embedding}), so the fibers of $f$ are
propositions. Surjectivity makes them inhabited
(\Cref{def:surjection}), and an inhabited proposition is a singleton,
so $f$ is an equivalence (\Cref{def:equivalence}).

The map $f$ is also order-reflecting. If $f\,x \lt{} f\,y$, then
$f\,x$ lies in the initial segment of $f\,y$, which initial-segment
preservation identifies with the image of the initial segment of $y$,
so $f\,x = f\,x'$ for some $x' \lt{} y$, and left cancellability gives
$x = x' \lt{} y$. The inverse of an order-preserving order-reflecting
equivalence is order-preserving.
\end{proof}

\subsection{Trichotomy}

\begin{definition}[Trichotomy]
\label{def:trichotomy}
An ordinal $\alpha$ is \emph{trichotomous} if for all $x,y:\us\alpha$ we
have $x \lt{} y$ or $x = y$ or $x \gt{} y$. We write $\OrdThree[]$ for
the type of trichotomous ordinals in a universe left implicit.
\AgdaRefs{\Agda{Ordinals.Type}{is-trichotomous};
\Agda{Ordinals.TrichotomousType}{Ordinal₃}.}
\end{definition}

\begin{remark}
\label{rem:trichotomy-EM}
Asking that every discrete ordinal be trichotomous is equivalent to
excluded middle.
\end{remark}
\begin{proof}
Excluded middle decides
  $x \lt{} y$ for all $x,y : \us\alpha$, and a well-order whose
  order is decidable is trichotomous, by transfinite induction on both
  points. If neither $x \lt{} y$ nor $y \lt{} x$, the induction
  hypotheses show that every predecessor of either point is a
  predecessor of the other, and extensionality gives $x = y$.
  \AgdaRefs{\Agda{Ordinals.Notions}{trichotomy₃}, via
  \texttt{trichotomy-from-decidable-order}.}

 Conversely, let $P$ be a
  proposition with $\neg\neg P$, and order the two-element type by
  \[
    0 \lt{} 1 \;:\equiv\; P, \qquad
    0 \lt{} 0 \;:\equiv\; 1 \lt{} 0 \;:\equiv\; 1 \lt{} 1 \;:\equiv\; \Zero.
  \]
  This order is proposition valued, transitive and well-founded with no
  assumption on $P$, and $\neg\neg P$ makes it extensional, so it is an
  ordinal, whose underlying type $\Two$ is discrete. Trichotomy at $0$
  and $1$ then gives $P$, since $0 = 1$ is false and $1 \lt{} 0$ is
  empty. This is double negation elimination, and hence excluded
  middle.
  \AgdaRefs{\Agda{Ordinals.Taboos}{EM-if-Every-Discrete-Ordinal-Is-Trichotomous},
  via \texttt{DNE-if-Every-Discrete-Ordinal-Is-Trichotomous}.}
\end{proof}

Trichotomous ordinals are nevertheless abundant without excluded
middle, and~\Cref{sec:brouwer-four} interprets every Brouwer code as a
trichotomous ordinal under an interpretation that dominates the
standard one.
\begin{lemma}
\label{lem:trichotomy-gives-discrete}
  Trichotomy implies discreteness.
\end{lemma}
\begin{proof}
Given $x$ and $y$, the two order cases of trichotomy decide $x \neq y$,
because a well-founded relation is irreflexive, and the middle case
decides $x = y$.
\AgdaRefs{\Agda{Ordinals.Notions}{trichotomous-gives-discrete}.}
\end{proof}

\subsection{Ordinal arithmetic}
\label{sec:ordinal-arithmetic}

\begin{definition}[Ordinal addition]
\label{def:ordinal-addition}
The sum $\alpha \oplusO \beta$ has as underlying type the binary sum of
those of $\alpha$ and $\beta$, ordered so that every element of
$\alpha$ precedes every element of $\beta$:
\begin{align*}
  \mathrm{inl}\,x \lt{} \mathrm{inl}\,x' &:\equiv x \lt\alpha x', &
  \mathrm{inl}\,x \lt{} \mathrm{inr}\,y' &:\equiv \One, \\
  \mathrm{inr}\,y \lt{} \mathrm{inr}\,y' &:\equiv y \lt\beta y', &
  \mathrm{inr}\,y \lt{} \mathrm{inl}\,x' &:\equiv \Zero.
\end{align*}
\AgdaRefs{\Agda{Ordinals.Arithmetic}{\_+ₒ\_}, whose well-order is the
  module \Agda{Ordinals.WellOrderArithmetic}{plus}.}
\end{definition}

\begin{lemma}
\label{lem:addition-well-order}
The order of~\Cref{def:ordinal-addition} is a well-order, so that
$\alpha \oplusO \beta$ is an ordinal.
\AgdaRefs{\Agda{Ordinals.WellOrderArithmetic}{plus.well-order}.}
\end{lemma}

\begin{proof}
It is proposition valued because its four cases are, two of them by
hypothesis and the other two being $\One$ and $\Zero$.

For transitivity, suppose $x \lt{} y \lt{} z$. If $z$ lies in the left
summand, then so do $x$ and $y$, no point of the left summand being
preceded by a point of the right, and transitivity is inherited from
$\alpha$. If $z$ lies in the right summand and $x$ in the left, then
$x \lt{} z$ holds by definition. If $x$ and $z$ both lie in the right
summand, then so does $y$, no point of the right summand preceding a
point of the left, and transitivity is inherited from $\beta$.

For well-foundedness, each $\mathrm{inl}\,x$ is accessible by induction
on the accessibility of $x$, its predecessors being the
$\mathrm{inl}\,x'$ with $x' \lt{} x$. Each $\mathrm{inr}\,y$ is
then accessible by induction on the accessibility of $y$, its
predecessors being all the $\mathrm{inl}\,x$, accessible as shown
above, together with the $\mathrm{inr}\,y'$ with $y' \lt{} y$.

For extensionality, let two points have the same predecessors. They
cannot be one from each summand, for $\mathrm{inl}\,x$ precedes every
$\mathrm{inr}\,y$, so that $\mathrm{inl}\,x$ would precede itself,
which well-foundedness forbids. If both are of the form
$\mathrm{inl}$, all of their predecessors are, and the extensionality
of $\alpha$ applies, and if both are of the form $\mathrm{inr}$, their
predecessors in the right summand agree and the extensionality of
$\beta$ applies.
\end{proof}

\begin{definition}[Lexicographic order on arbitrary sums]
\label{def:lexicographic-order}
For $\tau : \Ord[]$ and a family $\upsilon : \tau \to \Ord[]$, the
\emph{lexicographic order} on $\Sigma_{x:\tau}\,\upsilon\,x$ has the
index as the more significant coordinate:
\[
  (x,a) \lt{} (x',a') \;:\equiv\; (x \lt\tau x') \;+\;
  \textstyle\Sigma_{r\,:\,x = x'}\,
  \big(\mathrm{transport}^{\upsilon}\,r\,a \lt{\upsilon\,x'} a'\big).
\]
We write $\osum\tau\upsilon$ for the type $\Sigma_{x:\tau}\,\upsilon\,x$
equipped with this order.
\AgdaRefs{\Agda{Ordinals.LexicographicOrder}{slex-order}.}
\end{definition}

\begin{lemma}
\label{lem:lex-order}
The lexicographic order is
proposition valued, transitive and well-founded. It is moreover
extensional when the family is constant.
\AgdaRefs{\Agda{Ordinals.WellOrderArithmetic}{sum.prop-valued},
  \texttt{sum.transitive}, \texttt{sum.well-founded};
  \texttt{times.well-order} for the constant family.}
\end{lemma}

\begin{proof}
It is proposition valued because each of its two cases is, the second
because $\tau$ is a set, so that $r$ is unique, and the fiber orders
are proposition valued. The two cases exclude each other, because
$x \lt{} x'$ together with $x = x'$ would make $x$ precede itself. A
binary sum of two propositions that exclude each other is a
proposition.

For transitivity, if either of two consecutive steps moves the index
then so does their composite, by the transitivity of $\lt{}$ after
transporting, and if neither does, the three indices agree and the
fiber order is transitive.

For well-foundedness, we show that $(x,a)$ is accessible by induction
on the accessibility of $x$ and then on that of $a$. Its predecessors
are the $(x',a')$ with $x' \lt{} x$, accessible by the induction
hypothesis on $x$, and the $(x,a')$ with $a' \lt{} a$,
accessible by the induction hypothesis on $a$.

Now let the family be constant, say $\upsilon\,x = \alpha$ for every
$x$, and let $(x,a)$ and $(x',a')$ have the same predecessors. Given
$u \lt{} x$ we have $(u,a) \lt{} (x,a)$, so $(u,a) \lt{} (x',a')$,
which gives either $u \lt{} x'$, as wanted, or $u = x'$ together
with $a \lt{} a'$. In the second case $x' \lt{} x$, so
$(x',a') \lt{} (x,a)$ and hence $(x',a') \lt{} (x',a')$, which
well-foundedness forbids. By symmetry $x$ and $x'$ have the same
predecessors, so $x = x'$ by the extensionality of $\tau$, and we may
assume $x = x'$. Then for $b \lt{} a$ we have $(x,b) \lt{} (x,a)$,
hence $(x,b) \lt{} (x,a')$, which gives $b \lt{} a'$, since
$x \lt{} x$ is impossible, and by symmetry and the extensionality of the
fiber, we get $a = a'$.
\end{proof}

\begin{definition}[Ordinal multiplication]
\label{def:ordinal-multiplication}
Ordinal multiplication $\alpha \otimesO \beta$ is the lexicographic
order of~\Cref{def:lexicographic-order} in the case of a constant
family: the sum over~$\beta$ of the family constantly~$\alpha$, so that
$\beta$ is the more significant coordinate. The operands are thus
written in the opposite order from their sum, $\alpha$ before $\beta$,
as in the classical definition. It is an ordinal by~\Cref{lem:lex-order}.
\AgdaRefs{\Agda{Ordinals.Arithmetic}{\_×ₒ\_}, whose well-order is the
  module \Agda{Ordinals.WellOrderArithmetic}{times}.}
\end{definition}

\begin{lemma}[Maps of lexicographic sums]
\label{lem:sum-map-order}
For $\tau, \tau' : \Ord[]$ with families $\upsilon : \tau \to \Ord[]$
and $\upsilon' : \tau' \to \Ord[]$, a map $f : \us\tau \to
\us{\tau'}$ and maps $g\,x : \us{\upsilon\,x} \to
\us{\upsilon'\,(f\,x)}$, the map
$\summap fg : \osum\tau\upsilon \to \osum{\tau'}{\upsilon'}$
of~\Cref{def:sum-map} is
\begin{enumerate}
\item\label{item:sum-map-order-preserving} order-preserving if $f$ is
  and every $g\,x$ is,
\item\label{item:sum-map-order-reflecting} order-reflecting if $f$ is
  and every $g\,x$ is and $f$ is in addition an embedding.
\end{enumerate}
\AgdaRefs{\Agda{Ordinals.Closure}{pair-fun-is-order-preserving},
  \texttt{pair-fun-is-order-reflecting}.}
\end{lemma}

\begin{proof}
A comparison in the lexicographic order
of~\Cref{def:lexicographic-order} is of one of two kinds, a comparison
of the indices, or an identification of the indices together with a
comparison in the fiber.

(\labelcref{item:sum-map-order-preserving}) A comparison of the first
kind is preserved by $f$. One of the second kind is preserved by
$g\,x$, after transporting the fiber component along the image under
$f$ of the identification of the indices.

(\labelcref{item:sum-map-order-reflecting}) A comparison of the images
of the first kind is reflected by $f$. One of the second kind comes
with an identification $f\,x = f\,y$ of the images of the indices, and
$f$, being an embedding, is left cancellable, which identifies $x$
with $y$. The underlying type of an ordinal is a set
(\Cref{lem:ext-gives-set}), so the given identification is the image
of that one, and the two transports agree. The comparison in the fiber
is then reflected by $g\,x$.
\end{proof}

\begin{remark}
\label{rem:lex-not-extensional}
For arbitrary $\tau : \Ord[]$ and $\upsilon : \us\tau \to \Ord[]$, the
lexicographic order on $\Sigma_{x:\us\tau}\,\us{\upsilon\,x}$ fails
constructively to be extensional, in that excluded middle follows if
all such orders are extensional.
\end{remark}
\begin{proof}
Take $\tau$ to be the ordinal $\OmegaO$ and $\upsilon$ to be the family
that sends a truth value $p$ to the proposition $p \ne \bot$. Every
proposition is an ordinal in exactly one way, under the empty order,
so this is a family of ordinals. No comparison in a fiber is possible,
the fibers carrying the empty order, so the lexicographic sum is the
subtype
\[
  X \;:\equiv\; \textstyle\Sigma_{p : \OmegaO}\,p \ne \bot
\]
of $\OmegaO$ with the restricted order. No element of $X$ has a
predecessor in $X$, since a predecessor would have to be $\bot$, which
$X$ excludes. So the hypothesis of extensionality is vacuous, and
extensionality of $X$ says that $X$ is a proposition. Let $P$ be a
proposition with $\neg\neg P$. Both $P$ and $\neg\neg P$ are elements
of~$X$, so they are equal, and $P$ follows. This is double negation
elimination, and hence excluded middle. This
counterexample is due to Mike Shulman (personal communication).
\AgdaRefs{\Agda{Ordinals.ShulmanTaboo}{shulmans-taboo};
\Agda{Ordinals.ToppedArithmetic}{extensionality-of-ordinal-indexed-sums-gives-EM}
for the reduction of lexicographic sums to it.}
\end{proof}

One sufficient condition for the lexicographic order to be extensional
is that the family be constant, as we have seen above. Two further
sufficient conditions are that the index ordinal be trichotomous, and
that the summands have top elements.

\begin{lemma}
\label{lem:trichotomous-sum}
If $\tau$ is trichotomous and $\upsilon$ is a $\tau$-indexed family of
ordinals, then the lexicographic order is extensional, so that
$\osum\tau\upsilon$ is an ordinal.
\AgdaRefs{\Agda{Ordinals.TrichotomousArithmetic}{∑³}, whose well-order
  is the module \Agda{Ordinals.WellOrderArithmetic}{sum-cotransitive};
  \Agda{Ordinals.Notions}{tricho-gives-cotrans}.}
\end{lemma}

\begin{proof}
By~\Cref{lem:lex-order}, only extensionality is left to prove.
Trichotomy makes the order of $\tau$ cotransitive, in the sense that
$x \lt{} y$ implies, for any~$z$, that $x \lt{} z$ or $z \lt{} y$.
Indeed, compare $z$ with $y$. In the case $z \lt{} y$ there is nothing
to do, in the case $z = y$ we have $x \lt{} z$, and in the case
$y \lt{} z$ transitivity gives $x \lt{} z$.

Now let $(a,b)$ and $(x,y)$ have the same predecessors. Given
$u \lt{} a$, cotransitivity applied to the point $x$ gives $u \lt{} x$
or $x \lt{} a$. In the second case $(x,y) \lt{} (a,b)$, so that
$(x,y) \lt{} (x,y)$, which the well-foundedness of the lexicographic
order forbids. So every predecessor of $a$ is one of $x$, and by
symmetry the two have the same predecessors, which gives $a = x$ by
the extensionality of~$\tau$.

So we may take $a$ to be $x$. For $v \lt{} b$ we have
$(x,v) \lt{} (x,b)$ and hence $(x,v) \lt{} (x,y)$, which gives
$v \lt{} y$, the alternative $x \lt{} x$ being impossible. By symmetry
and the extensionality of $\upsilon\,x$, we get $b = y$.
\end{proof}

\begin{definition}[Topped ordinals]
\label{def:topped}
A \emph{top} of an ordinal $\alpha$ is a point that no element exceeds,
and $\OrdT[]$ is the type of ordinals equipped with a top:
\[
  \OrdT[] :\equiv \textstyle\Sigma_{\alpha:\Ord[]}\,\Sigma_{t:\alpha}\,
  \prod_{y:\alpha}\,\neg\,(t \lt{} y).
\]
\AgdaRefs{\Agda{Ordinals.ToppedType}{Ordinalᵀ}.}
\end{definition}
A largest element is a top, because the order is well-founded and
hence irreflexive, but not conversely: in the ordinal $\OmegaO$ no
element exceeds a truth value $P$ with $\neg\neg P$, since such a $P$
differs from $\bot$, whereas $\top$ is the largest element. This shows
that having a top element is in general data rather than property.
Classically, the topped ordinals are precisely the successor ordinals
(ordinals of the form $\alpha \oplusO \OneO$).

\begin{example}
\label{ex:NInf-ordinal}
The type $\NInf$ of~\Cref{ex:NInf} is an ordinal under its natural
order, in which $u$ precedes $v$ when $u$ is a finite point
$\underline n$ whose index is below~$v$ in the sense of~\Cref{ex:NInf}:
\[
  u \lt{} v \;:\equiv\;
  \textstyle\Sigma_{n:\N}\,\big((u = \underline n) \times (n \sqlt v)\big).
\]
It is topped by its largest element $\infty$, and we write $\NInfO$ for
it, as an ordinal or as a topped one according to the context.
\AgdaRefs{\Agda{TypeTopology.GenericConvergentSequence}{\_≺ℕ∞\_},
  \texttt{ℕ∞-ordinal}, \texttt{∞-top}, \texttt{∞-largest};
  \Agda{Ordinals.Arithmetic}{ℕ∞ₒ};
  \Agda{Ordinals.ToppedArithmetic}{ℕ∞ᵒ}.}
\end{example}

\begin{proof}
The predecessors of a point $v$ are exactly the finite points
$\underline n$ with $n \sqlt v$. One direction is the definition, which
asks a predecessor to be finite, and for the other, the comparison
$\underline n \lt{} v$ holds whenever $n \sqlt v$. Moreover, the index
$n$ in the definition is determined by $u$, because $\underline{(-)}$
is left cancellable.
\AgdaRefs{\Agda{TypeTopology.GenericConvergentSequence}{⊏-gives-≺},
  \texttt{≺-gives-⊏}, \texttt{ℕ-to-ℕ∞-lc}.}

The order is proposition valued, being a sum over $n$ of a product of
two equations in the sets $\NInf$ and $\Two$, with the index $n$
determined as observed~above.

For transitivity, let $u \lt{} v$ and $v \lt{} w$, say $u = \underline
m$ with $m \sqlt v$ and $v = \underline n$ with $n \sqlt w$. From
$m \sqlt \underline n$ we get $m < n$, and $w$ is decreasing, so that
$n \sqlt w$ gives $m \sqlt w$. This says that $u \lt{} w$.

For well-foundedness, each finite point $\underline n$ is accessible by
course-of-values induction on $n$. Its predecessors are the
$\underline m$ with $m \sqlt \underline n$, that is, with $m < n$, and
these are accessible by the induction hypothesis. Any point $v$ is then
accessible, its predecessors being finite points.

For extensionality, let $u$ and $v$ have the same predecessors. By the
first paragraph this says that $n \sqlt u$ if and only if $n \sqlt v$
for every $n$, and two binary digits that are $1$ in the same cases are
equal, so $u$ and $v$ have the same digits. The decreasing condition is
a proposition, so a point of $\NInf$ is determined by its digits, and
$u = v$.

Finally, the point $\infty$ is a top, because $\infty \lt{} v$ would
make $\infty$ a finite point~$\underline n$, and $\infty$ is not
finite, having $n \sqlt \infty$ for every~$n$ whereas
$\underline n \sqleq n$. It is moreover the largest element, since
every predecessor $\underline n$ of a point has $n \sqlt \infty$, and
so is a predecessor of $\infty$.
\end{proof}

Constructively, $\NInfO$ is topped by~$\infty$ but fails to be a
successor ordinal, because the top of a successor ordinal
$\alpha \oplusO \OneO$ is isolated, whereas $\infty$ fails to be
isolated constructively
(\Cref{ex:isolated}(\labelcref{item:infty-isolated-wlpo})).

The operations on topped ordinals and on trichotomous ordinals
(\Cref{def:trichotomy}) are written like those on plain ordinals.
Addition and multiplication of ordinals restrict to topped ordinals,
and the ordinal-indexed sum $\osum\tau\upsilon$ of a
topped-ordinal-indexed family of topped ordinals is again topped.
\claim{topped-arithmetic}
For topped ordinals we take $\tau \oplusT \upsilon$ to be the
ordinal-indexed sum over the topped ordinal $\TwoT$ of the family with
members $\tau$ and $\upsilon$, whose underlying type is
$\Sigma_{i : \One+\One}$ of these, and which is order equivalent to
the sum of~\Cref{def:ordinal-addition}.
\AgdaRefs{\Agda{Ordinals.ToppedArithmetic}{\_+ᵒ\_};
  \Agda{Ordinals.ToppedAdditionProperties}{alternative-plus}.}
They restrict to trichotomous ordinals as well, where the sum of a
trichotomous-ordinal-indexed family of trichotomous ordinals is an
ordinal by~\Cref{lem:trichotomous-sum} and is again trichotomous.
\claim{trichotomous-arithmetic}

\begin{lemma}
\label{lem:topped-sum}
If $\upsilon$ is a $\tau$-indexed family of topped ordinals, then the
lexicographic order is extensional, so that $\osum\tau\upsilon$ is an
ordinal. If moreover $\tau$ is topped, then so is $\osum\tau\upsilon$.
\AgdaRefs{\Agda{Ordinals.ToppedArithmetic}{∑ₒ} for the ordinal and
  \texttt{∑} for the topped one, whose well-order is the module
  \Agda{Ordinals.WellOrderArithmetic}{sum-top}.}
\end{lemma}

\begin{proof}
By~\Cref{lem:lex-order}, only extensionality is left to prove. Write
$t_u$ for the top of $\upsilon\,u$, and let $(a,b)$ and $(x,y)$ have
the same predecessors. Given $u \lt{} a$, we have
$(u,t_u) \lt{} (a,b)$ and hence $(u,t_u) \lt{} (x,y)$, which gives
$u \lt{} x$ or $u = x$ together with a comparison exceeding $t_u$ in
the fiber, and the latter is impossible. So every predecessor of $a$
is one of $x$, and by symmetry the two have the same predecessors,
which gives $a = x$ by the extensionality of~$\tau$. The fiber
components are then identified as in the proof
of~\Cref{lem:trichotomous-sum}, which for that step uses only the
irreflexivity of the two orders.

We take the top of the sum to be the pair of the top $t$ of $\tau$
with the top of $\upsilon\,t$. A point exceeding it would either
exceed $t$ in $\tau$ or exceed the top of $\upsilon\,t$ in that fiber,
and neither is possible.
\end{proof}

\subsection{Extended sums}
\label{sec:extended-sums}

The extended sum of~\Cref{sec:extension} becomes a construction on
ordinals once the extension of a family carries an order, which we use
to interpret the ordinal notations
of~\Cref{sec:brouwer-standard,sec:inductive-recursive}.
For the following, recall the $\Pi$-extension operation on families
given in~\Cref{def:extension}.
\begin{lemma}
\label{lem:ordinal-extension}
Let $j : X \hookrightarrow K$ be an embedding and
$\alpha : X \to \Ord[]$ a family of ordinals. For any $k : K$, give the
type $(\us{\alpha} \Pext j)\,k$ the relation
\[
  u \lt{} v \;:\equiv\; \Sigma_{p\,:\,\fib{j}{k}}
  u\,p \lt{\alpha\,(\mathrm{pr}_1 p)} v\,p .
\]
This is a well-order, so we obtain the extension
$\alpha \extend j : K \to \Ord[]$.
\AgdaRefs{\Agda{Ordinals.WellOrderExtension}{extension}, whose
\texttt{well-order} is an instance of the prop-indexed product of
ordinals.}
\end{lemma}

\begin{proof}
Since $j$ is an embedding, the fiber $\fib jk$ is a proposition, so
$(\us{\alpha} \Pext j)\,k$ is a product indexed by a proposition, and
$\lt{}$ is the order that such a product carries, asking the component
order to hold at some point of the fiber. For each $p : \fib jk$,
evaluation at $p$ is an equivalence from the product to the component
$\us{\alpha\,(\mathrm{pr}_1 p)}$, whose inverse sends a point of the
component to the section transporting it along the identifications of
the fiber.

The order is proposition valued, being a sum over a proposition of
proposition valued components.

For extensionality, suppose that $u$ and $v$ have the same
predecessors, and let $p : \fib jk$. A predecessor $x$ of $u\,p$
determines a section, which is a predecessor of $u$ in the product,
witnessed at $p$. By the hypothesis it is a predecessor of $v$,
witnessed at some point of the fiber, which is $p$ because the fiber
is a proposition, so $x$ is a predecessor of $v\,p$. Symmetrically,
$u\,p$ and $v\,p$ have the same predecessors, and the extensionality
of $\alpha\,(\mathrm{pr}_1 p)$ gives $u\,p = v\,p$. Since $p$ was
arbitrary, $u = v$.

For transitivity, two consecutive comparisons come with points of the
fiber, which are equal because the fiber is a proposition, so that the
transitivity of the component order at that point applies.

For well-foundedness, a point $u$ of the product is accessible as soon
as every $v \lt{} u$ is. A witness of $v \lt{} u$ supplies a point $p$
of $\fib jk$ together with $v\,p \lt{} u\,p$, and $u\,p$ is accessible
because the order of $\alpha\,(\mathrm{pr}_1 p)$ is well-founded, so
that $v\,p$ is accessible as one of its predecessors. Accessibility
transfers from the component back to the product along evaluation at
$p$, because a predecessor in the product of the section determined by
a point $x$ of the component has its value at $p$ below $x$.
\end{proof}

\begin{lemma}
\label{lem:extension-top}
If every $\alpha\,x$ is topped, then so is every
$(\alpha \extend j)\,k$, so the construction extends families of
topped ordinals.
\AgdaRefs{\Agda{Ordinals.WellOrderExtension}{top-preservation}.}
\end{lemma}

\begin{proof}
Suppose each $\alpha\,x$ has a top element $t_x$. Then
$p \mapsto t_{\mathrm{pr}_1 p}$ is a top element of $\lt{}$, since, if
it were below some $\psi$, we would have a point $p : \fib jk$ with
$t_{\mathrm{pr}_1 p} \lt{} \psi\,p$, contradicting the topness of
$t_{\mathrm{pr}_1 p}$.
\end{proof}

\begin{lemma}
\label{lem:extension-restricts}
The construction restricts back to $\alpha$ along $j$, giving
\[
  (\alpha \extend j)(j\,x) \iseq \alpha\,x
\]
for every $x : X$, as illustrated in the diagram below.
\[
  \begin{tikzcd}[row sep=large, column sep=large]
  X \arrow[r, hook, "j"] \arrow[dr, "\alpha"'] & K \arrow[d, dashed, "\alpha \extend j"] \\
  & \Ord[]
\end{tikzcd}
\]
\AgdaRefs{\Agda{Ordinals.Injectivity}{\_↗\_}, \texttt{↗-propertyₒ}.}
\end{lemma}

\begin{proof}
For $k = j\,x$ the fiber $\fib{j}{j\,x}$ is a singleton, being an
inhabited proposition, so evaluation at its center $(x,\mathrm{refl})$
is an equivalence $(\alpha \extend j)(j\,x) \to \alpha\,x$
by~\Cref{lem:extension-property}(\labelcref{item:ext-restricts}). It is
order-preserving, because a comparison $u \lt{} v$ is witnessed at some
point of the fiber, and that point is the center, the fiber being a
proposition, so that the comparison holds at the center. It is
order-reflecting, because a comparison at the center is by definition a
witness of $u \lt{} v$. The inverse of an order-preserving
order-reflecting equivalence is order-preserving, so this is an order
equivalence.
\end{proof}

Under univalence, \Cref{lem:extension-restricts} and, for total
separatedness, \Cref{lem:extension-tot-sep} show that the type of
ordinals and the type of totally separated ordinals are both
algebraically
injective in the sense of~\cite{EscardoInjective2021}, as univalence
turns order equivalences into identifications.
\claim{ordinals-are-ainjective}
\AgdaRefs{\Agda{Ordinals.Injectivity}{Ordinal-is-ainjective}, in
module \texttt{ordinals-injectivity}, and
\texttt{topped-ordinals-injectivity} for the topped variant;
\Agda{Ordinals.TotallySeparated}{TSOrdinal-is-ainjective};
\Agda{Ordinals.Equivalence}{eqtoidₒ} for the passage from an order
equivalence to an identification, and \texttt{UAₒ} for the equivalence
of the two.}

\begin{definition}[Countable extended sum and successor sum]
\label{def:extended-and-successor-sum}
Let $\upsilon : \N \to \OrdT[]$ be a family of topped ordinals.
\begin{enumerate}
\item\label{item:extended-sum} The \emph{extended sum} of $\upsilon$
  is the extension of~$\upsilon$ along the embedding
  $\iota : \N \hookrightarrow \NInf$, summed over the
  ordinal~$\NInfO$:
  \[
    \ssumone\upsilon \;:\equiv\; \osum{\NInfO}{}\,\big(\upsilon \extend
    \iota\big).
  \]
\item\label{item:successor-sum} The \emph{successor sum} of $\upsilon$
  is the extension of~$\upsilon$ along the left inclusion
  $\mathrm{inl} : \N \hookrightarrow \N + \One$, summed over the
  successor $\succT{\omegaO}$, taken as a topped ordinal with the added
  point as top:
  \[
    \ssumsub\upsilon \;:\equiv\; \osum{\succT{\omegaO}}{}\,\big(\upsilon
    \extend \mathrm{inl}\big).
  \]
\end{enumerate}
\AgdaRefs{\Agda{Ordinals.ToppedArithmetic}{∑¹} and \texttt{∑₁}; at the
level of types, \Agda{TypeTopology.SquashedSum}{Σ¹} and
\texttt{Σ₁}, the left inclusion being \texttt{over}.}
\end{definition}
Both are topped ordinals, by~\Cref{lem:extension-top}
and~\Cref{lem:topped-sum}, the index ordinals $\NInfO$ and
$\succT{\omegaO}$ being topped.

\begin{lemma}
\label{lem:successor-sum}
Let $\upsilon : \N \to \OrdT[]$.
\begin{enumerate}
\item\label{item:successor-sum-is-successor} The successor sum is order
  equivalent to the successor of the countable sum:
  \[
    \ssumsub\upsilon \;\iseq\; \Big(\osum{\omegaO}{}\,\upsilon\Big) \oplusO \One.
  \]
\item\label{item:successor-sum-retract} If $\upsilon\,n$ is a retract of
  $\N$ for every $n$, then so is $\ssumsub\upsilon$, and if
  $\upsilon\,n$ is discrete for every $n$, then so is
  $\ssumsub\upsilon$.
\end{enumerate}
\AgdaRefs{\Agda{TypeTopology.SquashedSum}{Σ₁-explicitly},
  \Agda{Ordinals.Closure}{∑₁-top-is-over-inr};
  \Agda{Ordinals.ToppedAdditionProperties}{∑₁-is-successorₒ}, and
  \Agda{Ordinals.Injectivity}{↗-propertyₒ},
  \texttt{↗-out-of-range}, \Agda{Ordinals.Closure}{∑-≃ₒ};
  \Agda{TypeTopology.SquashedSum}{Σ₁-is-discrete}, via
  \texttt{over-is-discrete};
  \Agda{Ordinals.Closure}{Σ₁-ℕ-retract}.}
\end{lemma}

\begin{proof}
The sum is taken over $\N + \One$, so its fibers are the extension of
$\upsilon$ along $\mathrm{inl}$. Since $\mathrm{inl}$ is an embedding,
that extension gives $\upsilon\,n$ back at $\mathrm{inl}\,n$, up to
order equivalence (\Cref{lem:extension-restricts}). At
$\mathrm{inr}\,\star$ the fiber of $\mathrm{inl}$ is empty, no natural
number having $\mathrm{inl}\,n = \mathrm{inr}\,\star$, so the extension
is the one-point ordinal there
(\Cref{lem:extension-property}(\labelcref{item:ext-off-image})). Write
$\upsilon'$ for the family on $\N + \One$ so described. Taking
underlying types, that of $\ssumsub\upsilon$ is the binary sum
$\big(\Sigma_{n:\N}\,\upsilon\,n\big) + \One$, whose top is the one
supplied by~\Cref{lem:topped-sum}, namely the pair of the top
$\mathrm{inr}\,\star$ of $\succT{\omegaO}$ with the point of its fiber.

(\labelcref{item:successor-sum-is-successor}) Order equivalent families
have order equivalent sums, by~\Cref{lem:sum-map}%
(\labelcref{item:sum-map-equiv}) and~\Cref{lem:sum-map-order} with the
identity on the index, so $\ssumsub\upsilon$ is order equivalent to
$\osum{\succT{\omegaO}}{}\,\upsilon'$. The map sending
$(\mathrm{inl}\,n, x)$ to $\mathrm{inl}\,(n,x)$ and
$(\mathrm{inr}\,\star, \star)$ to $\mathrm{inr}\,\star$ is an
equivalence from the underlying type of that sum to that of
$\big(\osum{\omegaO}{}\,\upsilon\big) \oplusO \One$, with the evident
inverse. The map and its inverse are order-preserving, as we see by
comparing the lexicographic order of the sum with the order of the
addition on each of the two summands. Composing the two order
equivalences gives the desired equivalence.

(\labelcref{item:successor-sum-retract}) By the description of the
underlying type above, it is enough to treat
$\big(\Sigma_{n:\N}\,\upsilon\,n\big) + \One$. If each $\upsilon\,n$ is
a retract of $\N$, then $\Sigma_{n:\N}\,\upsilon\,n$ is a retract of
$\N \times \N$ and hence of $\N$, by a pairing function, so that its
binary sum with $\One$ is a retract of $\N + \One$, which is again a
retract of $\N$. If each $\upsilon\,n$ is discrete, then so is
$\Sigma_{n:\N}\,\upsilon\,n$, discreteness being closed under $\Sigma$
over the discrete index $\N$
(\Cref{prop:discrete-closure}(\labelcref{item:discrete-sigma})), and so
is its binary sum with $\One$, discreteness being closed under $+$
(\Cref{prop:discrete-closure}(\labelcref{item:discrete-plus})).
\end{proof}

\begin{definition}[Comparison map]
\label{def:sum-comparison}
We now define a \emph{comparison map}
\[
  \cmap : \ssumsub\upsilon \to \ssumone\upsilon
\]
where $\upsilon : \N \to \OrdT[]$. We first need some preparation. The
underlying type of $\ssumsub\upsilon$ is
$\Sigma_{z\,:\,\N+\One}\,\us{(\upsilon \extend \mathrm{inl})\,z}$ and
that of $\ssumone\upsilon$ is
$\Sigma_{u\,:\,\NInf}\,\us{(\upsilon \extend \iota)\,u}$. By
\Cref{lem:ordinal-extension} the two fibers are the products
\[
  \us{(\upsilon \extend \mathrm{inl})\,z}
    = \textstyle\prod_{(n,r)\,:\,\fib{\mathrm{inl}}{z}}
      \us{\upsilon\,n},
  \qquad
  \us{(\upsilon \extend \iota)\,u}
    = \textstyle\prod_{(n,q)\,:\,\fib{\iota}{u}} \us{\upsilon\,n} .
\]
Write $\iota_1 : \N + \One \to \NInf$ for the map that sends
$\mathrm{inl}\,n$ to $\underline n$ and the added point to $\infty$,
so that $\iota_1 \circ \mathrm{inl} = \iota$. The map $\iota_1$ is
left cancellable: if
$\iota_1\,(\mathrm{inl}\,m) = \iota_1\,(\mathrm{inl}\,n)$ then
$\underline m = \underline n$ and so $m = n$. If $\iota_1$ takes the
same value at a finite point and at the added point, in either order,
then $\underline n = \infty$, which is impossible. And if both are the
added point, they are already equal, $\One$ having a single point.
Hence each $z : \N + \One$ gives the map
$\varphi_z : \fib{\iota}{\iota_1\,z} \to \fib{\mathrm{inl}}{z}$
defined~by
\[
  \varphi_z\,(n,q) = (n,r_q) ,
\]
where $r_q : \mathrm{inl}\,n = z$ is obtained by cancelling $\iota_1$
in $q$, which has type $\iota_1\,(\mathrm{inl}\,n) = \iota_1\,z$
because $\iota_1 \circ \mathrm{inl} = \iota$. Since $\varphi_z$ keeps the
first component, precomposition with it gives the map
$e_z : \us{(\upsilon \extend \mathrm{inl})\,z} \to
\us{(\upsilon \extend \iota)\,(\iota_1\,z)}$ defined~by
\[
  e_z\,\psi\,(n,q) = \psi\,(\varphi_z\,(n,q)) .
\]
With this notation, we define the \emph{comparison map}
$\cmap : \ssumsub\upsilon \to \ssumone\upsilon$ to be the map of sums
of~\Cref{def:sum-map} induced by $\iota_1$ and the family $e$,
\[
  \cmap \;=\; \summap{\iota_1}{e},
\]
which amounts to $\cmap\,(z,\psi) = (\iota_1\,z,\, e_z\,\psi)$.
\AgdaRefs{\Agda{TypeTopology.SquashedSum}{Σ-up}, which is
  \Agda{UF.PairFun}{pair-fun} applied to
  \Agda{TypeTopology.GenericConvergentSequence}{ι𝟙}, our $\iota_1$, and
  to \Agda{TypeTopology.SquashedSum}{over-ι-map}, our $e_z$, itself
  precomposition with
  \texttt{over-ι-fiber}, our $\varphi_z$, which uses
  \texttt{ι𝟙-lc} for the left cancellability of $\iota_1$;
  \Agda{Ordinals.Closure}{∑-up} for the ordinal version.}
\end{definition}

This comparison map is the countable-sum case of the comparison
between the discrete and the compact interpretation of Brouwer ordinal
codes, which is the subject of~\Cref{sec:discrete-compact}.

\begin{lemma}
\label{lem:sum-comparison}
The comparison map $\cmap$ is an extremely dense embedding and both
order-preserving and order-reflecting.
\AgdaRefs{\Agda{TypeTopology.SquashedSum}{Σ-up-embedding},
  \texttt{Σ-up-dense};
  \Agda{Ordinals.Closure}{∑-up-is-order-preserving},
  \texttt{∑-up-is-order-reflecting}, via
  \texttt{ι𝟙ᵒ-is-order-preserving} and
  \texttt{ι𝟙ᵒ-is-order-reflecting} for $\iota_1$.}
\end{lemma}

\begin{proof}
The map $\iota_1$ is an embedding and is extremely dense
(\Cref{ex:isolated}(\labelcref{item:iota1-embedding-dense})). Each
$e_z$ is an equivalence, being precomposition with the equivalence
$\varphi_z$. That $\varphi_z$ is an equivalence holds because the two
fibers are propositions, $\mathrm{inl}$ and $\iota$ being embeddings,
and because $\varphi_z$ has a converse, sending $(n,r)$ to $(n,q)$ with
$q$ obtained by applying $\iota_1$ to $r$. An equivalence is in
particular an embedding and extremely dense, so
that~\Cref{lem:sum-map}(\labelcref{item:sum-map-embedding},\labelcref{item:sum-map-dense}),
applied to $\iota_1$ on the indices and to the maps $e_z$ on the
fibers, makes $\cmap$ an extremely dense embedding.

Each $e_z$ is moreover order-preserving and order-reflecting. A
comparison in either of the two products is a point of the fiber
indexing it together with a comparison of the values there, and
$\varphi_z$ and its converse map such data in each product to such
data in the other. The map
$\iota_1$ is order-preserving and order-reflecting from
$\succT{\omegaO}$ to $\NInfO$. Among the finite points,
$\underline m \lt{} \underline n$ says that $m \sqlt \underline n$,
which holds exactly when $m < n$. A finite point precedes the added
point of $\succT{\omegaO}$, and $\underline n \lt{} \infty$ because
$n \sqlt \infty$. And $\infty$ precedes nothing, since
$\infty \lt{} v$ would make $\infty$ finite.
So~\Cref{lem:sum-map-order} applies, its embedding hypothesis
for~(\labelcref{item:sum-map-order-reflecting}) being again
that~$\iota_1$ is an embedding.
\end{proof}

\subsection{Extended suprema}
\label{sec:extended-suprema}

Every family of ordinals indexed by a small type has a least upper
bound, assuming set quotients, or, equivalently, set replacement
(\Cref{def:set-replacement}).

\begin{theorem}
\label{thm:sup}
Let $I : \UU$ and $\alpha : I \to \Ord[]$, and
consider the map
\[
  \sigma : \textstyle\Sigma_{i:I}\,\us{\alpha\,i} \to \Ord[],
  \qquad \sigma\,(i,x) = \down{\alpha\,i}{x} ,
\]
which sends a point of a member of the family to its initial segment in~$\Ord[]$.
\begin{enumerate}
\item\label{item:sup-is-ordinal} The image of the map $\sigma$ is an
  ordinal, written $\sup\alpha$, under the inherited order
  \[
    (\beta,w) \lt{} (\beta',w') \;:\equiv\; \beta \olt \beta' .
  \]
\item\label{item:sup-surjection} The map
  $\Sigma_{i:I}\,\us{\alpha\,i} \to \us{\sup\alpha}$ sending
  $(i,x)$ to the image of $x$ is a surjection.
\item\label{item:sup-lub} $\sup\alpha$ is the least upper bound of
  the family~$\alpha$.
\end{enumerate}
\AgdaRefs{\Agda{Ordinals.OrdinalOfOrdinalsSuprema}{sup}, constructed in
module \texttt{suprema} from module \texttt{construction-using-image},
with \texttt{construction-using-quotient} for the other construction;
\texttt{sup-is-upper-bound}; \texttt{sup-is-lower-bound-of-upper-bounds};
\texttt{initial-segment-of-sup-at-component}; \texttt{sum-to-sup},
\texttt{sum-to-sup-is-surjection}.}
\end{theorem}

The proofs of these clauses are in~\cite{deJongThesis2023}. Set
replacement makes the image of the map $\sigma$ equivalent to a type
in our base universe, although $\Ord[]$ is a type in
the next universe.  Whether suprema of families of ordinals exist
constructively was left open by Kraus, Nordvall Forsberg and
Xu~\cite{ForsbergKrausXu2021}, who established the least upper bound
property for weakly increasing families indexed by $\N$. Tom de Jong
answered it with a construction by set quotients, extending and
formalizing the one of~\cite[Lemma~10.3.22]{HoTTBook}, which is
claimed there only to be an upper bound, and the author of this paper
simultaneously answered it with the above construction. Both are
formalized by de Jong and are presented in his
thesis~\cite{deJongThesis2023}.

\begin{corollary}
\label{cor:sup-compact}
If $I$ is compact pointed and
$\us{\alpha\,i}$ is compact pointed for every $i : I$, then
$\us{\sup\alpha}$ is compact pointed.
\AgdaRefs{\Agda{Ordinals.CompactnessOfSuprema}{sup-is-compact∙}.}
\end{corollary}

\begin{proof}
The type $\Sigma_{i:I}\,\us{\alpha\,i}$ is compact pointed
by~\Cref{cor:compact-pt-closure}(\labelcref{item:pt-sigma}), and
$\us{\sup\alpha}$ is a surjective image of it
by~\Cref{thm:sup}(\labelcref{item:sup-surjection}),
so~\Cref{cor:compact-pt-closure}(\labelcref{item:pt-image}) applies.
\end{proof}

\begin{theorem}[Compactness of extended suprema]
\label{thm:extended-sup-compact}
For any embedding $j : X \hookrightarrow K$ into a compact pointed type
and any family $\alpha : X \to \Ord[]$ of compact pointed ordinals, the
supremum of the extended family $\alpha \extend j$ is itself compact
pointed.
\AgdaRefs{\Agda{Ordinals.CompactnessOfSuprema}{sup-is-compact∙};
  \Agda{TypeTopology.MicroTychonoff}{micro-tychonoff}.}
\end{theorem}

\begin{proof}
By~\Cref{cor:sup-compact} applied to the family $\alpha \extend j$ over
the compact pointed index $K$, it suffices to show that
$\us{(\alpha \extend j)\,k}$ is compact pointed for every $k : K$. By
\Cref{lem:ordinal-extension} this is the product of the types
$\us{\alpha\,x}$ over the index type $\fib jk$, all of which are
compact pointed by hypothesis, and $\fib jk$ is a proposition precisely
because $j$ is an embedding. So the
micro-Tychonoff~\Cref{thm:micro-tychonoff} applies.
\end{proof}

For use in~\Cref{sec:brouwer-standard}, we introduce notation for the
special case of the embedding $\iota : \N \hookrightarrow \NInf$.
\begin{definition}[Countable extended supremum]
\label{def:extended-supremum}
For a family $\alpha : \N \to \Ord[]$ of ordinals, the
\emph{extended supremum} is
\[
  \ssupone\alpha \;:\equiv\; \sup\big(\alpha \extend \iota\big),
\]
the supremum of the extension of $\alpha$ along the embedding
$\iota : \N \hookrightarrow \NInf$. The superscript has the same
meaning as in the extended sum, that the index type is $\NInf$ rather than
$\N + \One$.
\AgdaRefs{\Agda{Ordinals.OrdinalOfOrdinalsSuprema}{sup} applied to the
  extension \Agda{Ordinals.Injectivity}{\_↗\_}, as
  in \Agda{Ordinals.BrouwerCodesInterpretations}{⟦\_⟧₂}.}
\end{definition}

\section{Compact totally separated ordinals from Brouwer codes}
\label{sec:brouwer-standard}

In the remainder of this paper we give two notation systems, each with
several interpretations of the notations as ordinals. The first
occupies this section and takes a traditional system, the Brouwer
ordinal codes, whose interpretation $\compactsepint{-}$ is compact and
totally separated and classically dominates the standard
interpretation $\stdint{-}$, whose compactness fails constructively.
The subscript on an interpretation
names the operation interpreting the limit constructor, a sum or a
supremum, and a superscript~$1$ indicates that the family is first
extended along the embedding $\N \hookrightarrow \NInf$, as in the
extended sum of~\Cref{def:extended-and-successor-sum} and the extended
supremum of~\Cref{def:extended-supremum}.

We give five interpretations to Brouwer codes, obtaining different
combinations of compactness, total separatedness, trichotomy and
discreteness.  The last of the five is the discrete interpretation
of~\Cref{sec:discrete-compact}. It assigns to each code $b$ a
discrete ordinal~$\alttrichint{b}$, and the same code $b$ denotes a compact
totally separated ordinal $\compactsepint b$. There is an
extremely dense order-embedding of the discrete ordinal into
the compact totally separated one, which classically is an
equivalence but fails to be an equivalence constructively.

\begin{definition}[Brouwer codes]
\label{def:brouwer-codes}
The type $\Brouwer$ of Brouwer ordinal codes (countably branching
trees) is inductively generated by the constructors
\[
  \bZ : \Brouwer, \qquad
  \bS : \Brouwer \to \Brouwer, \qquad
  \bL : (\N \to \Brouwer) \to \Brouwer .
\]
\AgdaRefs{\Agda{Ordinals.BrouwerCodes}{B}.}
\end{definition}

\subsection{The standard interpretation}
\label{sec:std-interpretation}

\begin{definition}
  \label{def:std-interp} We define an interpretation function
\[
  \stdint{-} : \Brouwer \to \Ord[] \qquad \text{(standard interpretation)}
\]
by interpreting the constructors as indicated in the following table:

\begin{center}
\renewcommand{\arraystretch}{1.6}
\begin{tabular}{|c|c|c|c|}
  \hline
  & $\bZ$ & $\bS$ & $\bL$
  \\ \hline
  $\stdint{-}$ & $\ZeroO$ & $(-) \oplusO \OneO$ & $\sup$ \\
\hline
\end{tabular}
\end{center}

\noindent
Here $\sup$ is the supremum of an $\N$-indexed family of ordinals
(\Cref{thm:sup}).
\AgdaRefs{\Agda{Ordinals.BrouwerCodesInterpretations}{⟦\_⟧₀}, using
\Agda{Ordinals.OrdinalOfOrdinalsSuprema}{sup}, which requires
univalence and a form of set replacement to be constructed.}
\end{definition}
For the standard interpretation, compactness, trichotomy and
discreteness fail constructively, which motivates the alternative
interpretations given in~\Cref{sec:brouwer-four}.
\begin{proposition}
\label{prop:std-compact-lpo}
If the standard interpretation $\us{\stdint b}$ is compact for every
Brouwer code~$b$, then $\LPO$ holds.
\AgdaRefs{\Agda{Ordinals.BrouwerCodesInterpretations}{⟦\_⟧₀-compact-gives-LPO}.}
\end{proposition}

\begin{proof}
The code $\bL(n\mapsto \bS^n\,\bZ)$ is interpreted as the supremum of the
finite ordinals, which is $\omegaO$, and the compactness of its
underlying type $\N$ is $\LPO$
(\Cref{ex:compact-basic}(\labelcref{item:N-compact-lpo})).
\AgdaRefs{\Agda{Ordinals.BrouwerCodesInterpretations}{ω-is-a-standard-interpretation},
using \Agda{Ordinals.Omega}{ω-is-sup-of-Fin}.}
\end{proof}

\begin{proposition}
\label{prop:failure-trichotomy}
\leavevmode
\begin{enumerate}
\item\label{item:std-trichotomous-lpo} If the standard interpretation
  $\stdint{b}$ is trichotomous for every Brouwer code~$b$, then $\LPO$
  holds.
\item\label{item:std-discrete-wlpo} If $\us{\stdint{b}}$ is discrete
  for every Brouwer code~$b$, then $\WLPO$ holds.
\end{enumerate}
\AgdaRefs{\Agda{Ordinals.FailureOfTrichotomy}{trichotomy-of-the-standard-interpretation-gives-LPO},
\texttt{discreteness-of-the-standard-interpretation-gives-WLPO}.}
\end{proposition}

\begin{proof}
It suffices to decide, for an arbitrary $u : \NInf$, whether $u$ is
finite for~(\labelcref{item:std-trichotomous-lpo}), and whether
$u = \infty$ for~(\labelcref{item:std-discrete-wlpo}), because
deciding finiteness for every conatural number is equivalent to
$\LPO$, and deciding equality with $\infty$ is equivalent to $\WLPO$
(\Cref{ex:NInf}).
\AgdaRefs{\Agda{Taboos.LPO}{LPO'}, \texttt{LPO'-gives-LPO}.}
So fix $u : \NInf$. Both decisions come from one ordinal $\beta$ built
from~$u$. By~\Cref{ex:NInf},
the point $u$ is finite precisely when $u \sqleq i$ for some $i$.

\emph{The codes.} Put
\begin{align*}
  g\,i &\;=\; \bS\big(\text{if $u_i = 0$ then $\bS\,\bZ$ else $\bZ$}\big),
  \\
  h\,0 &= \bS(\bS\,\bZ),
  \\
  h\,1 &= \bS(\bL\,g),
  \\
  h\,(n{+}2) &= \bZ,
\end{align*}
the last clause being arbitrary. Write $\alpha\,i =
\stdint{g\,i}$ and $\beta = \stdint{\bL\,h} = \sup_i
\stdint{h\,i}$.
\AgdaRefs{in \Agda{Ordinals.FailureOfTrichotomy}{} the name $\beta$ is
given to the
family $i \mapsto \stdint{h\,i}$ instead, and the supremum is written
$\sup\beta$.}
Unfolding~\Cref{def:std-interp} and using that
$\ZeroO$ is left neutral for $\oplusO$,
\begin{align*}
  \alpha\,i &=
  \begin{cases}
    \TwoO & \text{if } u_i = 0, \\
    \OneO & \text{if } u_i = 1,
  \end{cases}
  \\
  \stdint{h\,0} &= \TwoO,
  \\
  \stdint{h\,1} &= (\sup_i \alpha\,i) \oplusO \OneO .
\end{align*}
The code we associate to $u$ is $\bL\,h$. The hypothesis
of~(\labelcref{item:std-trichotomous-lpo}) makes $\beta$ trichotomous,
and that of~(\labelcref{item:std-discrete-wlpo}) makes its underlying
type discrete.

\emph{Two elements of $\beta$.} We have $\OneO \olt \TwoO$, and
$\stdint{h\,0} \oemb \beta$ by~\Cref{thm:sup}(\labelcref{item:sup-lub}),
so $\OneO \olt \beta$. Also
$\sup_i \alpha\,i \olt (\sup_i \alpha\,i)
\oplusO \OneO = \stdint{h\,1} \oemb \beta$, an ordinal being an
element of its own successor, so $\sup_i \alpha\,i \olt \beta$. Let
$y_0, y_1 : \us\beta$ be the witnesses of these two facts, so that
\[
  \down\beta{y_0} = \OneO, \qquad \down\beta{y_1} = \sup_i \alpha\,i .
\]

\emph{Three lemmas.} First, the supremum $\sup_i\alpha\,i$ has a
point: each $\alpha\,i$, being of the form $\gamma \oplusO \OneO$, has
the point contributed by the~$\OneO$ summand, and the map
$\alpha\,i \to \sup_i\alpha\,i$ carries it to a point of the supremum.
Hence the inequality $\sup_i \alpha\,i \olt \OneO$ is impossible,
because an initial segment of $\OneO$ is $\ZeroO$, which would make
$\sup_i\alpha\,i$ empty.

Second, if $u$ is finite then $\OneO \olt \sup_i \alpha\,i$. To see
this, say $u = \underline n$, so that $u_n = 0$ and hence $\alpha\,n =
\TwoO$. Now $\OneO \olt \TwoO = \alpha\,n \oemb \sup_i\alpha\,i$.

Third, conversely, the inequality $\OneO \olt \sup_i\alpha\,i$ implies
that $u$ is finite. The finiteness of $u$ is a proposition, so we may
use the joint surjectivity of the maps $\alpha\,i \to \sup_i\alpha\,i$,
which is the surjectivity of the map out of the sum
(\Cref{thm:sup}(\labelcref{item:sup-surjection})). The element of
$\sup_i\alpha\,i$ whose initial segment is $\OneO$ therefore comes
from some $x : \us{\alpha\,i}$, and these maps are simulations
(\Cref{thm:sup}(\labelcref{item:sup-lub})), hence preserve initial
segments, so that $\OneO = \down{\alpha\,i}{x}$, that is,
$\OneO \olt \alpha\,i$.
Decide the value of $u_i$. If $u_i = 0$ then
$u$ is bounded by $i$, hence finite. If $u_i = 1$ then
$\alpha\,i = \OneO$, giving $\OneO \olt \OneO$, which is impossible
since $\olt$ is irreflexive.

\emph{Deciding finiteness.}
For~(\labelcref{item:std-trichotomous-lpo}), apply trichotomy of
$\beta$ to $y_0$ and $y_1$. Since passing to initial segments preserves
the order, the three cases for trichotomy are:
\begin{enumerate}
\item $y_0 \lt{} y_1$: then $\OneO \olt \sup_i \alpha\,i$, so $u$
  is finite by the third lemma.
\item $y_0 = y_1$: then $\OneO = \sup_i\alpha\,i$. So $u$ is not
  finite, for were it finite, the second lemma would give
  $\OneO \olt \OneO$.
\item $y_1 \lt{} y_0$: then $\sup_i\alpha\,i \olt \OneO$, which the
  first lemma rules out, so this case does not arise.
\end{enumerate}
\AgdaRefs{\Agda{Ordinals.FailureOfTrichotomy}{main-lemma},
\texttt{brouwer-code},
\texttt{trichotomy-of-the-standard-interpretation-gives-LPO}.}

\emph{Deciding whether $u = \infty$.}
For~(\labelcref{item:std-discrete-wlpo}), discreteness decides whether
$y_0 = y_1$. If $y_0 = y_1$, then $u$ is not finite, as in the second
case above, and a conatural number that is not finite is~$\infty$. If
$y_0 \neq y_1$, then $u \neq \infty$. For if $u = \infty$, then
$u_i = 1$, and hence $\alpha\,i = \OneO$, for every~$i$, so that
$\sup_i \alpha\,i = \OneO$, because the supremum of a constant family
indexed by an inhabited type is its constant value
by~\Cref{thm:sup}(\labelcref{item:sup-lub}), the constant value being
an upper bound of the family and a member of it. Then
$\down\beta{y_0} = \down\beta{y_1}$, and hence $y_0 = y_1$, because
passing to initial segments is left cancellable.
\AgdaRefs{\Agda{Ordinals.FailureOfTrichotomy}{discreteness-lemma},
\texttt{y₀-equal-to-y₁-gives-not-finite},
\texttt{∞-gives-sup-α-is-𝟙ₒ};
\Agda{Ordinals.OrdinalOfOrdinalsSuprema}{sup-of-constant-family};
\Agda{Ordinals.OrdinalOfOrdinals}{↓-lc}.}
\end{proof}

\begin{remark}
\label{rem:trichotomy-hypothesis-use}
The proof uses the hypotheses of trichotomy and of discreteness only at
the single pair $y_0, y_1$ of the single ordinal $\beta$, and $\beta$
is built from $u$ by a fixed construction. So the statement can be
sharpened to give a function assigning to each conatural number a
Brouwer code such that trichotomy of the standard interpretation of the
code decides the finiteness of the conatural number, and its
discreteness decides whether the conatural number is~$\infty$.
\end{remark}

Whether the discreteness of the standard interpretation for every
Brouwer code gives $\LPO$ is open.

\subsection{Three non-standard interpretations}
\label{sec:brouwer-four}

We now give three further interpretations of Brouwer codes, each
obtaining a different combination of compactness, total separatedness
and trichotomy, and all of them dominating the standard interpretation
classically (\Cref{thm:comparisons}).

\begin{definition}[Three non-standard interpretations]
  \label{def:non-standard-interpretations} We define three
  interpretation functions
\[
  \begin{array}{rcll}
    \trichint{-} & :  & \Brouwer \to \OrdThree[] & \text{(trichotomous interpretation)}\\
    \compactsepint{-}  & : & \Brouwer \to \OrdT[] & \text{(compact totally separated interpretation)} \\
    \compactint{-} & : & \Brouwer \to \Ord[] & \text{(compact interpretation)}
  \end{array}
\]
by interpreting the constructors as indicated in the following table:

\begin{center}
\renewcommand{\arraystretch}{1.6}
\begin{tabular}{|c|c|c|c|}
  \hline
  & $\bZ$ & $\bS$ & $\bL$
  \\ \hline
  $\trichint{-}$ & $\ZeroO$ & $(-) \oplusO \OneO$ & $\osumthree{}{}$
  \\\hline
  $\compactsepint{-}$ & $\OneT$ & $(-) \oplusT \OneT$ & $\ssumone{}$
  \\\hline
  $\compactint{-}$ & $\OneO$ & $(-) \oplusO \OneO$ & $\ssupone{}$ \\
\hline
\end{tabular}
\end{center}

\noindent
Here $\osumthree{}{}$ is the lexicographic sum
over~$\omegaO$ (\Cref{def:lexicographic-order}), which is an ordinal
by~\Cref{lem:trichotomous-sum}, the index $\omegaO$ being
trichotomous, $\ssumone{}$ is the
extended sum
(\Cref{def:extended-and-successor-sum}(\labelcref{item:extended-sum})),
and $\ssupone{}$ is the
extended supremum (\Cref{def:extended-supremum}).
\AgdaRefs{\Agda{Ordinals.BrouwerCodesInterpretations}{⟦\_⟧₃}
  for $\trichint{-}$, \texttt{⟦\_⟧₁} for $\compactsepint{-}$ and
  \texttt{⟦\_⟧₂} for $\compactint{-}$;
  \Agda{Ordinals.TrichotomousArithmetic}{\_+₃\_} and \texttt{∑³} for
  the successor and countable-sum operations of $\trichint{-}$.}
\end{definition}

In the compact totally separated and the compact interpretation we
deliberately map $\bZ$ to~$\One$ rather than $\Zero$, so that every
denoted ordinal is pointed and hence the
micro-Tychonoff~\Cref{thm:micro-tychonoff} is applicable to get the
compactness of the ordinals in their images.

\begin{theorem}
\label{thm:four-interp-props}
For every Brouwer code~$b$,
\begin{enumerate}
\item\label{item:trich-interp} the ordinal $\trichint b$ is
  trichotomous,
\item\label{item:compactsep-interp} the underlying type
  of $\compactsepint{b}$ is compact
  and totally separated,
\item\label{item:kappa-retract-cantor} the underlying type of
  $\compactsepint b$ is a retract of $\Cantor$,
\item\label{item:compact-interp} the underlying type
  of $\compactint b$ is compact.
\end{enumerate}
\AgdaRefs{\Agda{Ordinals.BrouwerCodesInterpretations}{⟦\_⟧₁-is-compact∙};
\texttt{⟦\_⟧₁-is-totally-separated}; \texttt{⟦\_⟧₂-is-compact∙};
\texttt{⟦\_⟧₃}, which lands in $\OrdThree[]$ by
construction;
\Agda{Ordinals.BrouwerCodesDiscreteAndCompactInterpretations}{Κ-Cantor-retract},
stated for \texttt{Κ} (\Cref{def:delta-kappa}).}
\end{theorem}

\begin{proof}
All by induction on $b$, with compactness throughout in the pointed
sense.

(\labelcref{item:trich-interp}) The interpretation lands in
$\OrdThree[]$ by construction, each of its clauses preserving
trichotomy: $\ZeroO$ is trichotomous, so is the successor of a
trichotomous ordinal, and so is a sum over $\omegaO$ of a family of
trichotomous ordinals, such a sum being an ordinal
by~\Cref{lem:trichotomous-sum}.

(\labelcref{item:compactsep-interp}) For compactness, the type $\One$ is
compact pointed. The successor clause is a sum over $\One+\One$,
so~\Cref{cor:compact-pt-closure}(\labelcref{item:pt-sigma}) applies. The
limit clause is~\Cref{thm:extended-sum-compact}. For total
separatedness, the type $\One$ is totally separated. The successor
clause is a $\Sigma$ whose index type $\One+\One$ is discrete, and total
separatedness is closed under $\Sigma$ over a discrete index
(\Cref{prop:tot-sep-closure}(\labelcref{item:tot-sep-sigma})). For the
limit clause,
apply~\Cref{lem:sigma-NInf-tot-sep}: each fiber is totally separated
by~\Cref{lem:extension-tot-sep} and the induction hypothesis, and the
fiber at $\infty$ is $\One$
by~\Cref{lem:extension-property}(\labelcref{item:ext-off-image}).
\AgdaRefs{\Agda{TypeTopology.SquashedSum}{Σ¹-is-totally-separated}.}

(\labelcref{item:kappa-retract-cantor}) We show by induction on $b$
that $\us{\compactsepint b}$ is a retract of $\Cantor$. For $\bZ$, the
type $\One$ is a retract of any pointed type. For $\bS$, we use that
$\Cantor + \Cantor$ is a retract of $\Cantor$
(split on the first digit and pass to the tail), and the retracts of
$\Cantor$ are closed under $+$
\AgdaRefs{\Agda{Ordinals.Closure}{+-retract-of-Cantor};
\Agda{TypeTopology.SquashedCantor}{+-Cantor-retract}}. For $\bL$, the
induction hypothesis gives, for each $i : \N$, a section
$s_i : \us{\compactsepint{b\,i}} \to \Cantor$ with retraction
$\rho_i$. By~\Cref{lem:extension-retract}, the extension along
$\N\hookrightarrow\NInf$ of the family
$i \mapsto \us{\compactsepint{b\,i}}$ is pointwise a retract of the
extension of the constant family at $\Cantor$, and hence the extended
sum $\us{\compactsepint{\bL\,b}}$ is a retract of the constant
extended sum. By~\Cref{lem:delayed-sequences}, the constant extended
sum, the type of delayed binary sequences, is equivalent to $\Cantor$,
and in particular a retract of it. Retracts compose, so
$\us{\compactsepint{\bL\,b}}$ is a retract of $\Cantor$.

(\labelcref{item:compact-interp}) The type $\One$ is compact pointed
and the successor clause is a coproduct. The limit clause is the one place where
the argument differs from~(\labelcref{item:compactsep-interp}), because
$\compactint{\bL\,b}$ is a supremum rather than a sum, and it
is~\Cref{thm:extended-sup-compact}.
\end{proof}

Being trichotomous, $\trichint b$ is discrete
(\Cref{lem:trichotomy-gives-discrete}).
\claim{trichotomous-interpretation-is-discrete}
The compactness
in~\Cref{thm:four-interp-props}(\labelcref{item:compactsep-interp})
does not follow
from~\Cref{thm:four-interp-props}(\labelcref{item:kappa-retract-cantor}),
the compactness of $\Cantor$ being independent of our type theory
(\Cref{rem:tychonoff-independence}).

\begin{remark}
\label{rem:tot-sep-via-cantor}
The total separatedness of $\compactsepint b$, proved directly
in~\Cref{thm:four-interp-props}(\labelcref{item:compactsep-interp}),
also follows
from~\Cref{thm:four-interp-props}(\labelcref{item:kappa-retract-cantor})
without a further induction, because $\Cantor$ is totally separated
(\Cref{ex:tot-sep}(\labelcref{item:cantor-not-discrete})) and total
separatedness is closed under retracts
(\Cref{prop:tot-sep-closure}(\labelcref{item:tot-sep-retract})).
In the direct proof of total
separatedness, addition is a sum over the discrete index
$\One+\One$, and total separatedness is closed under $\Sigma$ over a
discrete index (\Cref{prop:tot-sep-closure}(\labelcref{item:tot-sep-sigma})).
The index $\NInf$ of the
extended sum is not discrete, its discreteness being $\WLPO$
(\Cref{ex:isolated}(\labelcref{item:NInf-discrete-wlpo})), but~\Cref{lem:sigma-NInf-tot-sep} applies, the
fiber at $\infty$ being~$\One$. Closure under $\Sigma$ fails in general
(\Cref{ex:tot-sep}(\labelcref{item:ts-not-sigma})), and the property of
being a retract of $\Cantor$ is stronger than total separatedness and
is closed under the constructions used here.
\AgdaRefs{\Agda{TypeTopology.SquashedSum}{Σ¹-is-totally-separated};
\Agda{Ordinals.BrouwerCodesDiscreteAndCompactInterpretations}{Κ-is-totally-separated}.}
\end{remark}

\begin{proposition}
\label{prop:trich-compact-lpo}
If $\us{\trichint b}$ is compact for every Brouwer code~$b$, then
$\LPO$ holds.
\AgdaRefs{\Agda{Ordinals.BrouwerCodesInterpretations}{⟦\_⟧₃-compact-gives-LPO}.}
\end{proposition}

\begin{proof}
Apply the hypothesis to the code $\bL(\lambda i.\,\bS\,\bZ)$, whose
interpretation is the sum over $\omegaO$ of the constant family at
$\ZeroO \oplusO \OneO$, so that its underlying type is
$\Sigma_{i:\N}\,(\Zero+\One)$. The first projection makes $\N$ a
retract of it, compactness is closed under retracts
(\Cref{thm:compact-closure}(\labelcref{item:compact-retract})), and the
compactness of $\N$ is $\LPO$
(\Cref{ex:compact-basic}(\labelcref{item:N-compact-lpo})).
\end{proof}

Unlike $\compactsepint{-}$
(\Cref{thm:four-interp-props}(\labelcref{item:compactsep-interp})), the
compact interpretation $\compactint{-}$ fails to give
totally separated ordinals constructively:

\begin{proposition}
\label{prop:failure-tot-sep}
If $\us{\compactint b}$ is totally separated for every Brouwer
code~$b$, then $\neg\neg\WLPO$ holds.
\AgdaRefs{\Agda{Ordinals.FailureOfTotalSeparatedness}{total-separatedness-of-the-sup-of-extension-interpretation-gives-¬¬WLPO}.}
\end{proposition}

\noindent
This follows from the following counterexample, which does not
mention Brouwer codes and says that compact totally separated types
fail to be closed under suprema constructively.

\begin{example}
\label{ex:tot-sep-counterexample}
There are a type $I$ and a family $\beta : I \to \Ord[]$ such that $I$
is compact pointed and totally separated, $\us{\beta\,i}$ is compact
pointed and totally separated for every $i : I$, and the total
separatedness of $\us{\sup\beta}$ implies~$\neg\neg\WLPO$.
\AgdaRefs{\Agda{Ordinals.FailureOfTotalSeparatedness}{counterexample-to-total-separatedness}.}
\end{example}

We may take $I = \NInf$ and for $\beta$ the extension $\extfam$ along
$\N\hookrightarrow\NInf$ of the constant family $\alpha : \N \to \Ord[]$
with value $\TwoO$. For the proof, which requires a number of
lemmas, we define
\[
  \sier \;:\equiv\; \sup\extfam
\]
and use the following explicit description of it.
By~\Cref{def:extension}, since $\N\hookrightarrow\NInf$
is an embedding whose fiber over $u$ is the proposition that $u$ is
finite,
\[
  \us{\extfam\,u} \;=\; \Two^{\text{$u$ is finite}} .
\]
Its elements are \emph{partial elements of $\Two$}, with the
finiteness of $u$ as domain of definition, and $\sier$ is obtained
from all such partial elements, for all $u$, by identifying those with
the same initial segment. We will show that $\sier$ is order equivalent
to the following \emph{Sierpi\'nski ordinal}.
\claim{extension-is-partial-boolean}

\begin{definition}[A Sierpi\'nski ordinal]
\label{def:sierpinski}
Call a proposition \emph{semidecidable} if it is equivalent to the
finiteness of some conatural number,
\[
  \textstyle\exists_{u:\NInf}\,
  \big(P \simeq \text{$u$ is finite}\big),
\]
which is a proposition. Let
\[
  \Sier \;:\equiv\;
  \Sigma_{P : \Omega_{\UU_0}}\,(\text{$P$ is semidecidable}),
  \qquad
  t \ltS t' \;:\equiv\; \neg\,\dom t \times \dom t'
\]
where $\dom$ is the first projection, the \emph{domain of definition}.
Write $\botS$ and $\topS$ for the elements given by $\Zero$ and
$\One$, which are semidecidable via $\infty$ and $0$ respectively.
\AgdaRefs{\Agda{Ordinals.FailureOfTotalSeparatedness}{is-semidecidable};
\texttt{𝕊}; \texttt{\_≺ₛ\_}; \texttt{⊥ₛ}; \texttt{⊤ₛ}; \texttt{𝓢}.}
\end{definition}

\begin{lemma}
\label{lem:sierpinski-ordinal}
The relation $\ltS$ is a well-order, so $\Sier$ is an ordinal, and two
of its elements are equal precisely when their domains of definition
are logically equivalent.
\AgdaRefs{\Agda{Ordinals.FailureOfTotalSeparatedness}{≺ₛ-prop-valued};
\texttt{≺ₛ-transitive}; \texttt{≺ₛ-extensional};
\texttt{≺ₛ-well-founded}; \texttt{to-𝕊-＝}.}
\end{lemma}

\begin{proof}
The semidecidability of a proposition is a proposition, so $\Sier$ is a
subtype of $\Omega_{\UU_0}$, and hence two of its elements are equal
exactly when their domains of definition are logically equivalent. In
particular $\Sier$ is a set. The relation $\ltS$ is proposition valued,
being a product of two propositions, and transitive, since $t \ltS t'$
and $t' \ltS t''$ give $\neg\,\dom t$ from the first and $\dom t''$ from
the second. For extensionality, suppose $t$ and $t'$ have the same
predecessors. If $\dom t$ holds then $\botS \ltS t$, hence $\botS \ltS
t'$, whose second component is $\dom t'$, and symmetrically. So the
domains of definition are logically equivalent, and $t = t'$. For
well-foundedness, if $t'' \ltS t' \ltS t$ then $t'' \ltS t'$ gives
$\dom t'$ while $t' \ltS t$ gives $\neg\,\dom t'$, a contradiction. So
no predecessor of $t$ has a predecessor, hence every predecessor of $t$
is accessible, and therefore so is $t$.
\end{proof}

\begin{lemma}
\label{lem:sup-is-sierpinski}
There is an order equivalence $\sier \iseq \Sier$.
\AgdaRefs{\Agda{Ordinals.FailureOfTotalSeparatedness}{𝓼-is-𝓢};
\texttt{⟨𝓼⟩-is-𝕊}; \texttt{𝕊-is-small}.}
\end{lemma}
\begin{proof}
For each $u : \NInf$ let $\theta_u : \us{\extfam\,u} \to \us\sier$ be the
canonical map into the supremum. These maps are simulations
(\Cref{thm:sup}(\labelcref{item:sup-lub})) and jointly surjective
(\Cref{thm:sup}(\labelcref{item:sup-surjection})), and they preserve
initial segments, a simulation mapping the initial segment of $\xi$
onto that of its value, so that
$\down{\sier}{\theta_u\,\xi} = \down{\extfam\,u}{\xi}$.

Send a partial element $\xi : \Two^{\text{$u$ finite}}$ to the
proposition
\[
  \mathcal F\,\xi \;:\equiv\;
  \textstyle\Sigma_{\varphi\,:\,\text{$u$ finite}}\,(\xi\,\varphi = 1),
\]
that is, the proposition that $\xi$ is defined and takes the value $1$.
This is a proposition, because the finiteness of $u$ is a proposition
and $\Two$ is a set,
and it is semidecidable, because from $\xi$ we build a decreasing
binary sequence, whose value at $i$ is $1$ before $u$ is known to
satisfy $u \sqleq i$, and is $1$ minus the value of $\xi$ at the
finiteness witness afterwards, hence a conatural number whose
finiteness is equivalent to $\mathcal F\,\xi$. Writing
$\mathbb F\,\xi$ for the resulting element of $\Sier$, the map
$\mathbb F$ is order-preserving, and it depends on $\xi$ only through
$\theta_u\,\xi$, since $\mathbb F\,\xi$ is determined by the initial
segment of $\xi$, which $\theta_u$ preserves. Since the maps $\theta_u$
are jointly surjective and $\Sier$ is a set, there is a
unique $\tau : \us\sier \to \Sier$ with $\tau\,(\theta_u\,\xi) = \mathbb
F\,\xi$. It is order-preserving, by joint surjectivity and the
corresponding property of $\mathbb F$. It preserves initial
segments, because an element of $\Sier$ below $\tau\,y$ has empty
domain of definition, so it is $\botS$, and we exhibit a predecessor of
$y$ mapping to $\botS$, namely the image under $\theta_u$ of the
constantly~$0$ partial element, which is unique because $\tau$ is left
cancellable at $\botS$, so that the required existence is a proposition
and the joint surjectivity of the $\theta_u$ may be used to produce it.
So $\tau$ is a simulation. It is
surjective, for given $(P,w)$ with $w$ a witness of
semidecidability, we may unpack $w$, because being in the image is
a truncated statement, to get $u$ with $P \simeq (\text{$u$
finite})$, and then the constantly~$1$ partial element of
$\us{\extfam\,u}$ is sent to $(P,w)$. So $\tau$ is an order
equivalence by~\Cref{lem:surjective-simulations-are-order-equivs}.
\AgdaRefs{\Agda{Ordinals.FailureOfTotalSeparatedness}{θ-is-simulation},
\texttt{θ-is-jointly-surjective}, \texttt{𝓕-is-semidecidable},
\texttt{𝔽-is-order-preserving}, \texttt{θ-to-𝔽-＝}, \texttt{T-point},
\texttt{τ-is-simulation}, \texttt{τ-is-surjection},
\texttt{τ-is-equiv}.}
\end{proof}

\begin{lemma}
\label{lem:sierpinski-separation}
If some $p : \Sier \to \Two$ has $p\,\botS \neq p\,\topS$, then
$\WLPO$ holds. Consequently, if $\Sier$ is totally separated then
$\neg\neg\WLPO$ holds.
\AgdaRefs{\Agda{Ordinals.FailureOfTotalSeparatedness}{𝕊-separation-gives-WLPO};
\texttt{𝕊-totally-separated-gives-¬¬WLPO};
\texttt{𝓼-totally-separated-gives-¬¬WLPO}.}
\end{lemma}

\begin{proof}
Every $u : \NInf$ gives an element of $\Sier$, namely its own
finiteness, which is semidecidable by definition. Write $\hat u$ for
this element.
Then $\widehat{\underline n} = \topS$ for every natural number $n$, since a
natural number is finite, and $\widehat\infty = \botS$, since $\infty$ is
not. So, given $p$ as in the statement, the composite $q = (u
\mapsto p\,\hat u) : \NInf \to \Two$ takes the constant value $p\,\topS$
on the natural numbers and the different value $p\,\botS$ at $\infty$.
Composing with the complement if necessary, a matter decided by
inspecting the two boolean values, we may assume that $q$ is
constantly~$0$ on $\N$ and takes the value $1$ at $\infty$. Such a
map decides, for an arbitrary $u : \NInf$, whether $u = \infty$,
by inspecting the value $q\,u$, which if $0$ gives $u \neq \infty$
and if $1$ makes $u$ differ from every finite point, since a conatural
number that is not finite is $\infty$. That is $\WLPO$ in the formulation
of~\Cref{ex:NInf}.
\AgdaRefs{\Agda{Taboos.BasicDiscontinuity}{basic-discontinuity-taboo}.}

For the second claim, assume $\Sier$ is totally separated and, for the
sake of contradiction, that $\neg\WLPO$ holds. For any
$p : \Sier \to \Two$ we cannot have $p\,\botS \neq p\,\topS$, by the
first part. Deciding the two values, we conclude
$p\,\botS = p\,\topS$. As this holds for every $p$, total
separatedness gives $\botS = \topS$, which is a contradiction. So
$\neg\neg\WLPO$.
\end{proof}

\begin{proof}[Proof of~\Cref{ex:tot-sep-counterexample}]
Take $I = \NInf$ and $\beta = \extfam$. Then $\NInf$ is
compact pointed (\Cref{ex:compact-pt-basic}(\labelcref{item:NInf-compact}))
and totally separated, being
a retract of $\Cantor$ (\Cref{ex:NInf}), which is totally separated
(\Cref{ex:tot-sep}(\labelcref{item:cantor-not-discrete})), and total
separatedness is inherited by retracts
(\Cref{prop:tot-sep-closure}(\labelcref{item:tot-sep-retract})). Each
$\us{\extfam\,u}$ is $\Two^{\text{$u$ finite}}$, which is compact
pointed by the micro-Tychonoff~\Cref{thm:micro-tychonoff}, and totally
separated by closure of total separatedness under products
(\Cref{prop:tot-sep-closure}(\labelcref{item:tot-sep-pi})). Finally, if $\us{\sup\extfam} = \us\sier$
is totally separated then so is $\Sier$, by~\Cref{lem:sup-is-sierpinski}
and closure of total separatedness under equivalence
(\Cref{prop:tot-sep-closure}(\labelcref{item:tot-sep-retract})), and
then~\Cref{lem:sierpinski-separation} gives $\neg\neg\WLPO$.
\end{proof}

\begin{proof}[Proof of~\Cref{prop:failure-tot-sep}]
Apply the hypothesis to the code
$b = \bL(\lambda\_.\,\bS\,\bZ)$.
By~\Cref{def:non-standard-interpretations},
$\compactint{\bS\,\bZ} = \compactint{\bZ} \oplusO \OneO = \OneO
\oplusO \OneO$, which is $\TwoO$, and therefore
$\compactint b = \sup\extfam = \sier$, since $\compactint{-}$ sends
$\bZ$ to $\OneO$ rather than to $\ZeroO$. So $\us\sier$ is totally
separated,
and~\Cref{lem:sup-is-sierpinski,lem:sierpinski-separation} give
$\neg\neg\WLPO$.
\end{proof}

\subsection{Comparing the four interpretations}
\label{sec:comparing-four}

Classically, the three alternative interpretations dominate the
standard interpretation and are related as in~\Cref{thm:comparisons}
below. Excluded middle is used twice, to make the successor operation
of $\Ord[]$ monotone and to turn an order-preserving map between
ordinals into $\oemb$. A fifth interpretation follows
in~\Cref{sec:discrete-compact}, related to $\compactsepint{-}$ in a
stronger and constructive way (\Cref{thm:iota}).

\begin{theorem}
\label{thm:comparisons}
Assuming the principle of excluded middle, the following order
relations hold for the four interpretations of any Brouwer code~$b$:
\begin{center}
\begin{tikzpicture}[every node/.style={font=\normalsize}]
  \node (stdl) at (-2,2) {standard};
  \node (stdv) at (1,2) {$\stdint b$};
  \node (le1) at (2,2) {$\oemb$};
  \node (triv) at (3,2) {$\trichint b$};
  \node (tril) at (7,2) {trichotomous};
  \node (cmpl) at (-2,0) {compact};
  \node (cmpv) at (1,0) {$\compactint b$};
  \node (le2) at (2,0) {$\oemb$};
  \node (csepv) at (3,0) {$\compactsepint b$};
  \node (csepl) at (7,0) {compact totally separated.};
  \node at ($(stdv)!0.5!(cmpv)$) {$\oembdown$};
  \node at ($(triv)!0.5!(csepv)$) {$\oembdown$};
\end{tikzpicture}
\end{center}
\AgdaRefs{\Agda{Ordinals.BrouwerCodesInterpretations}{comparison₀₃};
\texttt{comparison₀₂}; \texttt{comparison₂₁}; \texttt{comparison₃₁}.}
\end{theorem}
\noindent So the compact totally separated interpretation gives the largest
ordinal among the four.
\begin{proof}
The top row, the left-hand column and the bottom row are proved by
induction on $b$. All three use excluded middle at $\bS$, for
monotonicity of the successor operation on ordinals
\AgdaRefs{\Agda{Ordinals.AdditionProperties}{succ-monotone}}, and the
top row and the bottom row use it again at $\bL$, for the fact that a
supremum is bounded by the corresponding sum. That fact is in turn the
passage from an order-preserving map to $\oemb$, applied to the map
$y \mapsto (x,y)$ of a summand into the sum, which makes the sum an
upper bound of the family, so that the supremum, being least among upper
bounds (\Cref{thm:sup}(\labelcref{item:sup-lub})), is below it. The
right-hand column uses the passage directly.

\emph{The top row.} At $\bZ$ both sides are $\ZeroO$. At $\bS$, apply
monotonicity of the successor to the induction hypothesis.
At $\bL$, monotonicity of the supremum in its family gives
$\sup_i\stdint{b\,i} \oemb \sup_i\trichint{b\,i}$, and this supremum
is below $\trichint{\bL\,b}$, a supremum being bounded by the
corresponding sum.
\AgdaRefs{\Agda{Ordinals.SupSum}{sup-bounded-by-sum₃}.}

\emph{The left-hand column.} The same shape, with one extra step at
$\bL$, where the family whose supremum is taken on the right is the
extension of the family
$n \mapsto \compactint{b\,n}$ along $\N\hookrightarrow\NInf$, so we first
identify its value at $\underline n$ with $\compactint{b\,n}$,
by~\Cref{lem:extension-restricts} and univalence, compare suprema over
$\N$, and then
observe that the supremum over $\N$ of a family restricted along
$\iota$ is bounded by the supremum over all of $\NInf$. At $\bZ$ the
inequality is $\ZeroO \oemb \OneO$, the zero
ordinal being least.

\emph{The bottom row.} Again the same, using at $\bL$ that extension
along an embedding preserves $\oemb$ fiberwise, the extension being at
each point a product
indexed by a proposition, so that a simulation between the members gives
one between the products, and then that suprema compare and that a
supremum is bounded by the corresponding sum.
\AgdaRefs{\Agda{Ordinals.Injectivity}{↗-preserves-⊴};
\Agda{Ordinals.SupSum}{sup-bounded-by-sumᵀ}.} At $\bS$ we also need
that addition of topped ordinals agrees with addition of plain ordinals
on underlying ordinals.
\AgdaRefs{\Agda{Ordinals.ToppedAdditionProperties}{alternative-plus}.}

\emph{The right-hand column.} This one is proved differently. We define
a map
$\us{\trichint b} \to \us{\compactsepint
b}$ by induction on $b$ and show that it is order-preserving, and then
excluded middle turns an order-preserving map between ordinals into
$\oemb$. At $\bZ$ the source is empty. At
$\bS$ the two summands go to the two summands. At $\bL$ the pair
$(i,x)$ goes to $\iota\,i$ paired with the partial element that sends a
point $(j,p)$ of the fiber of $\iota$ over $\iota\,i$ to the recursively
defined value at $x$, transported along the identification $j = i$
supplied by left cancellability of $\iota$. Order preservation is a
routine induction.
\AgdaRefs{\Agda{Ordinals.BrouwerCodesInterpretations}{map₃₁},
\texttt{map₃₁-is-order-preserving};
\Agda{Ordinals.OrdinalOfOrdinals}{EM-implies-order-preserving-gives-≼}.}
\end{proof}

\begin{proposition}
  \label{prop:comparisons-taboos}
  \leavevmode
\begin{enumerate}
\item\label{item:comparison03-gives-lpo} If the relation
  $\stdint b \oemb \trichint b$ holds for every Brouwer code~$b$, then
  $\LPO$ holds.
\item\label{item:comparison21-gives-notnot-wlpo} If the relation
  $\compactint b \oemb \compactsepint b$ holds for every Brouwer
  code~$b$, then $\neg\neg\WLPO$ holds.
\end{enumerate}
\AgdaRefs{\Agda{Ordinals.FailureOfTrichotomy}{comparison₀₃-gives-LPO};
  \Agda{Ordinals.FailureOfTotalSeparatedness}{comparison₂₁-gives-¬¬WLPO}.}
\end{proposition}

\begin{proof}
(\labelcref{item:comparison03-gives-lpo}) An ordinal with a simulation
into a trichotomous ordinal is trichotomous. Since the trichotomous
interpretation is trichotomous
(\Cref{thm:four-interp-props}(\labelcref{item:trich-interp})), the
standard interpretation is trichotomous for every code,
and~\Cref{prop:failure-trichotomy}(\labelcref{item:std-trichotomous-lpo})
gives $\LPO$.
\AgdaRefs{\Agda{Ordinals.Maps}{simulation-into-trichotomous-gives-trichotomous}.}

(\labelcref{item:comparison21-gives-notnot-wlpo}) The domain of an
embedding into a totally separated type is totally separated. Since
$\us{\compactsepint b}$ is totally separated
(\Cref{thm:four-interp-props}(\labelcref{item:compactsep-interp})),
$\us{\compactint b}$ is totally separated for every code, and
\Cref{prop:failure-tot-sep} gives $\neg\neg\WLPO$.
\AgdaRefs{\Agda{TypeTopology.TotallySeparated}{embedding-into-totally-separated-gives-totally-separated}.}
\end{proof}

\begin{question}
\label{q:comparisons-taboos}
Does each of the remaining two comparisons, namely the relation
$\stdint b \oemb \compactint b$ for every code $b$, and the relation
$\trichint b \oemb \compactsepint b$ for every code $b$, imply a
constructive taboo?
\end{question}

\subsection{A fifth interpretation}
\label{sec:discrete-compact}

Our fifth interpretation is a modification of the trichotomous
interpretation $\trichint{-}$
of~\Cref{def:non-standard-interpretations}, mapping $\bZ$ to the
ordinal $\One$, rather than $\Zero$, and $\bL$ to $\ssumsub{}$ in
order to match the compact totally separated interpretation
$\compactsepint{-}$ and be able to get
\Cref{thm:iota}(\labelcref{item:iota-density}).

\begin{definition}
  \label{def:delta-kappa} We define an interpretation function
\[
  \alttrichint{-} : \Brouwer \to \OrdT[] \qquad
  \text{(discrete interpretation)}
\]
by interpreting the constructors as in the second row of the following
table, whose first row repeats the compact totally separated
interpretation of~\Cref{def:non-standard-interpretations} for
comparison:

\begin{center}
\renewcommand{\arraystretch}{1.6}
\begin{tabular}{|c|c|c|c|}
  \hline
  & $\bZ$ & $\bS$ & $\bL$
  \\ \hline
  $\compactsepint{-}$ & $\OneT$ & $(-) \oplusT \OneT$ & $\ssumone{}$
  \\\hline
  $\alttrichint{-}$ & $\OneT$ & $(-) \oplusT \OneT$ & $\ssumsub{}$ \\
\hline
\end{tabular}
\end{center}

\noindent
Here $\ssumone{}$ and $\ssumsub{}$ are the extended sum and the
successor sum of~\Cref{def:extended-and-successor-sum}.
\AgdaRefs{\Agda{Ordinals.BrouwerCodesDiscreteAndCompactInterpretations}{Δ}
for $\alttrichint{-}$; \texttt{Κ} in the same module for
$\compactsepint{-}$, repeating
\Agda{Ordinals.BrouwerCodesInterpretations}{⟦\_⟧₁} clause for clause
without the assumptions of univalence, propositional truncation and set
replacement, and the name the pointers below use for it.}
\end{definition}

\begin{theorem}
\label{thm:delta-kappa-props}
For every Brouwer code~$b$,
\begin{enumerate}
\item\label{item:delta-trichotomous} the ordinal $\alttrichint b$ is
  trichotomous,
\item\label{item:delta-retract-N} the underlying type of
  $\alttrichint b$ is a retract of $\N$, hence discrete.
\end{enumerate}
\AgdaRefs{\Agda{Ordinals.BrouwerCodesDiscreteAndCompactInterpretations}{Δ-retract-of-ℕ};
\texttt{Δ-is-discrete}; \texttt{Δ-is-trichotomous}.}
\end{theorem}
\noindent
Notice that~(\labelcref{item:delta-trichotomous}) gives another
route to the discreteness of $\alttrichint b$
(\Cref{lem:trichotomy-gives-discrete}), as does a direct induction on
$b$ using~\Cref{prop:discrete-closure}.
Item~(\labelcref{item:delta-retract-N}) is the counterpart
of~\Cref{thm:four-interp-props}(\labelcref{item:kappa-retract-cantor}),
with $\N$ in place of $\Cantor$.
\begin{proof}
(\labelcref{item:delta-trichotomous}) By induction on $b$. For $\bZ$
the one-point ordinal is trichotomous. For $\bS$, ordinal addition
$\tau \oplusT \upsilon$ is not a coproduct here but the ordinal-indexed
sum over the ordinal $\TwoT$ of the two-element family, whose index
type is $\One+\One$ (\Cref{sec:ordinal-arithmetic}), and trichotomy is
closed under ordinal-indexed sums, the index $\TwoT$ being
trichotomous. For $\bL$, the successor sum must preserve
trichotomy. Unfolding $\ssumsub{}$ as the sum over
$\succT\omegaO$ of the extension of the family along
$\N\hookrightarrow\N+\One$, the index is trichotomous and it remains
to see that the extension takes trichotomous values. At a point $z$ of
$\N+\One$ its value is a product indexed by the fiber of
$\N\hookrightarrow\N+\One$ over $z$, and such a product is trichotomous
as soon as that fiber is decidable, since if the fiber has a point, the
product is order equivalent to the member of the family at that point,
and if the fiber is empty, the product is $\One$. The fiber is
decidable, being a singleton over $\mathrm{inl}\,n$ and empty over the
added point.
\AgdaRefs{\Agda{Ordinals.ToppedArithmetic}{∑₁-is-trichotomous};
\Agda{Ordinals.Arithmetic}{𝟙ₒ-is-trichotomous};
\Agda{Ordinals.ToppedArithmetic}{+ᵒ-is-trichotomous};
\Agda{Ordinals.Injectivity}{↗-is-trichotomous};
\Agda{Ordinals.WellOrderArithmetic}{pip.decidable-index-gives-trichotomy}.}

(\labelcref{item:delta-retract-N}) We show by induction on $b$ that
$\us{\alttrichint b}$ is a retract of $\N$. For $\bZ$ the underlying
type is $\One$, a retract of $\N$. For $\bS$, use that the
retracts of $\N$ are closed under $\Sigma$, in the sense that if
$\us\tau$ and every $\us{\upsilon\,x}$ are retracts of $\N$, then so is
$\Sigma_x \us{\upsilon\,x}$, being a retract of $\N\times\N$, which is
a retract of $\N$ by a pairing function
\AgdaRefs{\Agda{Ordinals.Closure}{Σ-retract-of-ℕ}}; the index type
$\One+\One$ of the addition is a retract of $\N$ by sending $0$ to the
left point and every successor to the right one. For $\bL$, the
interpretation is a successor sum, whose underlying type is a retract
of $\N$ as soon as those of the summands are,
by~\Cref{lem:successor-sum}(\labelcref{item:successor-sum-retract}).
Discreteness follows since $\N$ is discrete and discreteness is closed
under retracts.
\end{proof}

\noindent
The trichotomy of the successor sum uses the decidability of the fibers
of $\N\hookrightarrow\N+\One$. For the extended sum the corresponding
fibers are those of $\N\hookrightarrow\NInf$, whose decidability is
$\LPO$, so the same argument is unavailable for $\compactsepint{-}$.

Each of the two interpretations fails constructively to have the
property that the other has.

\begin{proposition}
\label{prop:delta-kappa-fail}
\leavevmode
\begin{enumerate}
\item\label{item:delta-compact-lpo} The statement that
  $\us{\alttrichint b}$ is compact for every Brouwer code~$b$ is
  equivalent to $\LPO$.
\item\label{item:kappa-discrete-wlpo} If $\us{\compactsepint b}$ is
  discrete for every Brouwer code~$b$, then $\WLPO$ holds.
\end{enumerate}
\AgdaRefs{\Agda{Ordinals.BrouwerCodesDiscreteAndCompactInterpretations}{Δ-compact-iff-LPO},
combining \texttt{Δ-compact-gives-LPO} and \texttt{LPO-gives-Δ-compact};
\texttt{Κ-discrete-gives-WLPO}.}
\end{proposition}

\begin{proof}
For the forward implications, apply the hypotheses to
$b = \bL(\lambda i.\,\bZ)$, the code of~\Cref{ex:omega-plus-one},
whose discrete interpretation is $\succT{\omegaO}$ with underlying type
$\N+\One$ and whose compact one is $\NInfO$ with underlying type~$\NInf$.

Compactness of $\N+\One$ gives compactness of $\N$, a retract of it
(\Cref{thm:compact-closure}(\labelcref{item:compact-retract})), which is
$\LPO$ (\Cref{ex:compact-basic}(\labelcref{item:N-compact-lpo})), and
discreteness of $\NInf$ is $\WLPO$
(\Cref{ex:isolated}(\labelcref{item:NInf-discrete-wlpo})). Conversely,
$\LPO$ makes $\N$ compact, and $\us{\alttrichint b}$ is a retract of
$\N$ for every~$b$
(\Cref{thm:delta-kappa-props}(\labelcref{item:delta-retract-N})).
\end{proof}

\noindent
So the trichotomy
of~\Cref{thm:delta-kappa-props}(\labelcref{item:delta-trichotomous})
has no counterpart for $\compactsepint{-}$, since trichotomy implies
discreteness (\Cref{lem:trichotomy-gives-discrete}).

\begin{definition}[The comparison map]
\label{def:iota}
The map $\iemb_b : \us{\alttrichint b} \to \us{\compactsepint b}$ is
defined by induction on $b$ as follows:
\[
  \begin{array}{l}
    \iemb_{\bZ} = \mathrm{id} \\[0.4ex]
    \iemb_{\bS\,b}\,(\mathrm{inl}\,\pt, y)
      = (\mathrm{inl}\,\pt, \iemb_b\,y) \\[0.4ex]
    \iemb_{\bS\,b}\,(\mathrm{inr}\,\pt, y)
      = (\mathrm{inr}\,\pt, y) \\[0.4ex]
    \iemb_{\bL\,b}
      = \cmap \circ \ssumsub{\big(i \mapsto \iemb_{b\,i}\big)} ,
  \end{array}
\]
where $\ssumsub{(i \mapsto \iemb_{b\,i})}$ is the map between the
successor sums that applies $\iemb_{b\,i}$ in the $i$-th summand and
fixes the added point, and $\cmap$ is the comparison map
of~\Cref{def:sum-comparison} for the family
$\upsilon\,i = \compactsepint{b\,i}$.
\AgdaRefs{\Agda{Ordinals.BrouwerCodesDiscreteAndCompactInterpretations}{ι};
\Agda{TypeTopology.SquashedSum}{Σ↑}, \texttt{Σ₁-functor},
\texttt{Σ-up}; \Agda{UF.PairFun}{pair-fun}.}
\end{definition}

\begin{theorem}
\label{thm:iota}
For every Brouwer code~$b$, the map
$\iemb_b : \us{\alttrichint b} \to \us{\compactsepint b}$ is
\begin{enumerate}
\item an \emph{embedding},
\item \label{item:iota-density} \emph{extremely dense},
\item order-preserving,
\item order-reflecting.
\end{enumerate}
\AgdaRefs{\Agda{Ordinals.BrouwerCodesDiscreteAndCompactInterpretations}{ι};
\texttt{ι-is-embedding}; \texttt{ι-is-dense};
\texttt{ι-is-order-preserving}; \texttt{ι-is-order-reflecting}.}
\end{theorem}

\begin{proof}
The four properties are established by induction on $b$. At $\bZ$ the
map is the identity. At $\bS$
it is a map of sums, with the identity on the index and with the
induction hypothesis and the identity on the two fibers, so that the
embedding property and the density come
from~\Cref{lem:sum-map}(\labelcref{item:sum-map-embedding},
\labelcref{item:sum-map-dense}) and the two order properties
from~\Cref{lem:sum-map-order}, order reflection requiring that the
index map be an embedding
(\labelcref{item:sum-map-order-reflecting}), which the identity is.

At $\bL$, the family of maps $\iemb_{b\,i}$ induces a map between the
successor sums, namely $\mathrm{inl}\,(i,x) \mapsto \mathrm{inl}\,(i,
\iemb_{b\,i}\,x)$ and $\mathrm{inr}\,\pt \mapsto \mathrm{inr}\,\pt$,
which preserves all four properties fiberwise, and the comparison
map~$\cmap$ for the family $\upsilon\,i = \compactsepint{b\,i}$ has all
four (\Cref{lem:sum-comparison}), so the composite does too, extreme
density and the embedding property both being closed under
composition.
\AgdaRefs{\Agda{Ordinals.Closure}{∑↑-dense}, \texttt{∑↑-embedding},
\texttt{∑↑-is-order-preserving}, \texttt{∑↑-is-order-reflecting}.}
\end{proof}

The first code whose interpretations are both infinite shows the extent
of the failure of the comparison embedding to be an equivalence.

\begin{example}
\label{ex:omega-plus-one}
Take $b = \bL(\lambda i.\,\bZ)$, the countable sum of the
one-point code. By~\Cref{def:delta-kappa} the two interpretations are
the successor sum and the extended sum of the constant family at $\OneT$,
\[
  \alttrichint b = \ssumsub{\big(i \mapsto \OneT\big)}
    \iseq \big(\succT{\omegaO}\big), \qquad
  \compactsepint b = \ssumone{\big(i \mapsto \OneT\big)} \iseq \NInfO ,
\]
the summands being one-point ordinals and a sum of those being its own
index,
\AgdaRefs{\Agda{Ordinals.Closure}{∑₁-of-𝟙ᵒ}, \texttt{∑¹-of-𝟙ᵒ}, via
\texttt{∑-of-𝟙ᵒ} and \texttt{↗-of-𝟙ᵒ}.}
so the discrete side is $\N + \One$ with its added point isolated,
because the point of $\One$ is isolated and $\mathrm{inr}$ preserves
isolatedness
(\Cref{prop:discrete-closure}(\labelcref{item:discrete-plus})),
\AgdaRefs{\Agda{Ordinals.Closure}{∑₁-top-is-isolated}.}
the compact side is $\NInf$ with its added point $\infty$, whose
isolatedness is $\WLPO$
(\Cref{ex:isolated}(\labelcref{item:infty-isolated-wlpo})), and
$\iemb_b$ is the inclusion
$\N + \One \hookrightarrow \NInf$ that sends the added point to
$\infty$.

The ordinal $\alttrichint b$ is discrete and a retract of $\N$, but it
is compact exactly when $\LPO$ holds
(\Cref{ex:compact-basic}(\labelcref{item:N-compact-lpo})). The
ordinal $\compactsepint b$ is compact
and totally separated with no assumption
(\Cref{thm:four-interp-props}(\labelcref{item:compactsep-interp})),
but it
is discrete, equivalently its top point is isolated, exactly when
$\WLPO$ holds
(\Cref{ex:isolated}(\labelcref{item:NInf-discrete-wlpo,item:infty-isolated-wlpo})).
The map between them is an
extremely dense order-embedding, by~\Cref{thm:iota}.
Read as binary sequences, the image of $\iemb_b$ consists of the sequences
$1^n0^\omega$ and the sequence $1^\omega$. Extreme density of $\iemb_b$
says that no decreasing binary sequence lies outside this image, which
is a theorem, and surjectivity says that every decreasing binary
sequence lies in it, which for this code is exactly $\LPO$.
\end{example}

\begin{proposition}
\label{prop:brouwer-iota-equiv-LPO}
The statement that $\iemb_b$ is an equivalence for every Brouwer
code~$b$ is equivalent to $\LPO$.
\AgdaRefs{\Agda{Ordinals.BrouwerCodesDiscreteAndCompactInterpretations}{ι-is-equiv-iff-LPO},
combining \texttt{ι-is-equiv-gives-LPO}, which needs only the single
code $\bL(\lambda i.\,\bZ)$, and \texttt{LPO-gives-ι-is-equiv}, which
needs the induction on~$b$.}
\end{proposition}

\begin{proof}
For the forward direction, apply the hypothesis to the code
$\bL(\lambda i.\,\bZ)$, where $\iemb_b$ is the inclusion
$\N+\One\hookrightarrow\NInf$, which is an equivalence exactly when
$\LPO$ holds (\Cref{ex:omega-plus-one}). Conversely, assume $\LPO$ and
induct on~$b$, following the clauses of~\Cref{def:iota}. At $\bZ$ the
map is the identity. At $\bS\,b$ it is a map of sums whose index part
is the identity and whose two fiber parts are the induction hypothesis
and the identity, so it is an equivalence
by~\Cref{lem:sum-map}(\labelcref{item:sum-map-equiv}). At
$\bL\,b$ the map between the successor sums induced by the family
$\iemb_{b\,i}$ is an equivalence by the induction hypothesis, and
the map $\cmap$ is an equivalence because the maps $e_z$ are
equivalences (\Cref{lem:sum-comparison}) and $\LPO$ makes $\iota_1$ an
equivalence.
\end{proof}

The embedding $\iemb_b$ is order-preserving and order-reflecting, and
it is natural to ask whether it additionally satisfies the
initial-segment condition, so that
$\alttrichint b \oemb \compactsepint b$ for all codes~$b$. The code
$\bL(\lambda i.\,\bZ)$ does satisfy this, but the requirement that it
hold for every code fails constructively.

\begin{lemma}
\label{lem:omega-plus-two}
\leavevmode
\begin{enumerate}
\item\label{item:iota1-simulation} The map
  $\iota_1 : \N+\One \to \NInf$ is a simulation of $\succT{\omegaO}$
  into~$\NInfO$.
\item\label{item:omega-plus-two-lpo} If
  $(\succT{\omegaO}) \oplusT \OneT \oemb \NInfO \oplusT \OneT$, then
  $\LPO$ holds.
\end{enumerate}
\AgdaRefs{\Agda{Ordinals.ConvergentSequence}{ω+𝟙-is-⊴-ℕ∞};
\Agda{Ordinals.Closure}{ω+𝟚-⊴-ℕ∞+𝟙-gives-LPO}.}
\end{lemma}

\begin{proof}
(\labelcref{item:iota1-simulation}) The map is order-preserving,
because $\underline m \lt{} \underline n$ exactly when \mbox{$m < n$}, and
$\underline n \lt{} \infty$ for every~$n$ (\Cref{ex:NInf-ordinal}). It
preserves initial segments because the points of $\NInf$ below
$\iota_1\,z$ are the finite points $\underline n$ with
$n \sqlt \iota_1\,z$ (\Cref{ex:NInf-ordinal}), and such
a point is the image of $\mathrm{inl}\,n$, which is below $z$ in
$\succT{\omegaO}$ whether $z$ is finite or is the added point.

(\labelcref{item:omega-plus-two-lpo}) Let $f$ be a simulation of
$(\succT{\omegaO}) \oplusT \OneT$ into $\NInfO \oplusT \OneT$. In
both ordinals, the right summand
contributes the single point $(\mathrm{inr}\,\pt,\pt)$, and no point
is above it (\Cref{sec:ordinal-arithmetic}), the order of $\OneT$
being empty.

Every point $(\mathrm{inl}\,\pt,z)$ of the domain is below
$(\mathrm{inr}\,\pt,\pt)$, so its image is below
$f\,(\mathrm{inr}\,\pt,\pt)$ and hence is not
$(\mathrm{inr}\,\pt,\pt)$. The image therefore lies in the left
summand, giving a map $g : \N+\One \to \NInf$ with
$f\,(\mathrm{inl}\,\pt,z) = (\mathrm{inl}\,\pt,\,g\,z)$.

The map $g$ is a simulation of $\succT{\omegaO}$ into $\NInfO$. It is
order-preserving because $f$ is, and it preserves initial segments
because on both sides the points below a point of the left summand are
the points of the left summand below it. There is at most one
simulation of one ordinal into another
(\Cref{sec:ordinals-in-type-theory}), so
$g = \iota_1$ by~(\labelcref{item:iota1-simulation}),
and in particular $g$ sends the added point of $\N+\One$ to $\infty$.

The map $f$ sends $(\mathrm{inr}\,\pt,\pt)$ to $(\mathrm{inr}\,\pt,\pt)$.
Otherwise its value there is $(\mathrm{inl}\,\pt,w)$ for some
$w : \NInf$, and, the point $(\mathrm{inl}\,\pt,\mathrm{inr}\,\pt)$ of
the domain being below $(\mathrm{inr}\,\pt,\pt)$, the conatural number
$\infty = g\,(\mathrm{inr}\,\pt)$ would be below $w$, whereas $\infty$
is below no conatural number (\Cref{ex:NInf-ordinal}).

Now let $u : \NInf$. The point $(\mathrm{inl}\,\pt,u)$ of the codomain
is below $(\mathrm{inr}\,\pt,\pt) = f\,(\mathrm{inr}\,\pt,\pt)$, so by
initial-segment preservation it is in the image of~$f$, and the point
of the domain sent to it is below $(\mathrm{inr}\,\pt,\pt)$. That point
lies in the left summand, say $(\mathrm{inl}\,\pt,z)$, and then
$\iota_1\,z = g\,z = u$. The assignment $u \mapsto z$ is
a section of $\iota_1$, and deciding whether $z$ is of
the form $\mathrm{inl}\,n$ decides whether $u$ is finite, which is
$\LPO$.
\AgdaRefs{\Agda{Taboos.LPO}{ι𝟙-has-section-gives-LPO}.}
\end{proof}

\begin{proposition}
\label{prop:delta-below-kappa}
\leavevmode
\begin{enumerate}
\item\label{item:delta-below-kappa-omega} The relation
  $\alttrichint b \oemb \compactsepint b$ holds for the code
  $b = \bL(\lambda i.\,\bZ)$.
\item\label{item:delta-below-kappa-lpo} The relation
  $\alttrichint b \oemb \compactsepint b$ holds for every Brouwer
  code~$b$ if and only if $\LPO$ holds.
\end{enumerate}
\AgdaRefs{\Agda{Ordinals.BrouwerCodesDiscreteAndCompactInterpretations}{Δ𝟙-⊴-Κ𝟙};
\texttt{Δ-⊴-Κ-iff-LPO}, combining \texttt{Δ-⊴-Κ-gives-LPO}, which only
needs the single code $\bS(\bL(\lambda i.\,\bZ))$, and
\texttt{LPO-gives-Δ-⊴-Κ}.}
\end{proposition}

\begin{proof}
(\labelcref{item:delta-below-kappa-omega}) We have
$\alttrichint b \iseq \big(\succT{\omegaO}\big) \oemb \NInfO \iseq
\compactsepint b$, by~\Cref{ex:omega-plus-one}
and~\Cref{lem:omega-plus-two}(\labelcref{item:iota1-simulation}), and
an order equivalence is a simulation and simulations compose.

(\labelcref{item:delta-below-kappa-lpo}) ($\Leftarrow$) Under $\LPO$
the map $\iemb_b$ is an equivalence
(\Cref{prop:brouwer-iota-equiv-LPO}), and it is order-preserving and
order-reflecting (\Cref{thm:iota}), so it is an order equivalence, and
an order equivalence is a simulation.

($\Rightarrow$) The single code $b = \bS(\bL(\lambda i.\,\bZ))$
suffices. For this code,
$\big((\succT{\omegaO}) \oplusT \OneT\big) \iseq \alttrichint b$
and $\compactsepint b \iseq \big(\NInfO \oplusT \OneT\big)$,
by~\Cref{def:delta-kappa,ex:omega-plus-one}, sums of order equivalent
ordinals being order equivalent. An order equivalence is a simulation
and simulations compose, so a simulation of $\alttrichint b$ into
$\compactsepint b$ gives
$(\succT{\omegaO}) \oplusT \OneT \oemb \NInfO \oplusT \OneT$,
and~\Cref{lem:omega-plus-two}(\labelcref{item:omega-plus-two-lpo})
applies.
\end{proof}

In particular, under excluded middle, the discrete interpretation and
the compact totally separated one are order equivalent, and hence equal
by univalence, so that the discrete interpretation occupies the same
place as $\compactsepint b$ in the diagram of~\Cref{thm:comparisons}.

\subsection{The least element property}
\label{sec:kappa-infs}

The compact ordinals $\compactsepint b$ are not just compact. Every
complemented subset has an infimum, which is a member when the subset
is non-empty.

\begin{definition}[Infima of complemented subsets]
\label{def:has-inf}
Let $X$ carry a relation~$\leq$. A complemented subset of $X$ is given
by its characteristic function $p : X \to \Two$, whose roots are its
elements. Say that $X$ \emph{has infima of complemented subsets} if
for every $p$ there is a point $x_0 : X$ that is
\begin{enumerate}
\item\label{item:conditional-root} a \emph{conditional root}: if
  $p$ has a root then $x_0$ is a root,
\item\label{item:roots-infimum} the infimum of the roots:
  $x_0 \leq x$ for every root $x$, and $l \leq x_0$ for every $l$ that
  is below all roots.
\end{enumerate}
For an ordinal this is taken with respect to the weak
order \mbox{$x \leq y :\equiv \neg\,(y \lt{} x)$}.
\AgdaRefs{\Agda{Ordinals.InfProperty}{has-inf}, \texttt{is-conditional-root},
\texttt{is-roots-infimum};
\Agda{Ordinals.ToppedType}{has-infs-of-complemented-subsets}.}
\end{definition}

\begin{remark}
\label{rem:inf-gives-compact}
The property of having infima of complemented subsets strengthens pointed
compactness, since clause~(\labelcref{item:conditional-root}) alone
already gives it. Given
$p$, take the conditional root $x_0$ as a universal witness. If $p\,x_0
= 1$ then $p$ has no root, since a root would force $p\,x_0 = 0$, and a
$\Two$-valued function with no root is constantly~$1$.
\AgdaRefs{\Agda{Ordinals.InfProperty}{has-inf-gives-compact∙}.}
\end{remark}

\begin{lemma}[Least element property]
\label{lem:least-element}
Let $X$ have infima of complemented subsets with respect to a
relation~$\leq$. Then every non-empty complemented subset of $X$ has a
least element, that is, a member below all members.
\AgdaRefs{\Agda{Ordinals.InfProperty}{has-inf-gives-least-root},
\texttt{is-least-root}, \texttt{is-roots-lower-bound}.}
\end{lemma}

\begin{proof}
Let $p$ be the characteristic function of the subset and $x_0$ the point
supplied by~\Cref{def:has-inf}. By~\Cref{rem:inf-gives-compact}, $X$
is compact, and the subset is non-empty, so $p$ has a root by
complemented choice
(\Cref{prop:compact-complemented}(\labelcref{item:complemented-choice})).
Then $x_0$ is a root by
clause~(\labelcref{item:conditional-root}), and it is below every root
by clause~(\labelcref{item:roots-infimum}).
\end{proof}

The following strengthens the micro-Tychonoff
\Cref{thm:micro-tychonoff} from compactness to infima.

\begin{theorem}
\label{thm:micro-inf-tychonoff}
Let $X$ be a proposition and $Y : X \to \UU$ a family, and let each
$Y\,x$ carry an arbitrary relation~${\lt{}}$, with derived relation
\[
y \preceq y' \;:\equiv\; \neg\,(y' \lt{} y).
\]
Give $\prod_x Y\,x$ the pointwise relation
\[
\varphi \leq \gamma \;:\equiv\; \prod_{x:X} \varphi\,x \preceq \gamma\,x .
\]
If every $Y\,x$ has infima of complemented subsets with respect
to~$\preceq$ (\Cref{def:has-inf}), then $\prod_x Y\,x$ has infima of
complemented subsets with respect to~$\leq$.
\AgdaRefs{\Agda{TypeTopology.MicroInfTychonoff}{micro-inf-tychonoff}, which
writes the pointwise relation in the equivalent form
$\neg\,\Sigma_{x:X}\,\gamma\,x \lt{} \varphi\,x$.}
\end{theorem}

\begin{proof}
The construction of the witness is the same as
in~\Cref{thm:micro-tychonoff}. Given $p : \prod_x Y\,x \to \Two$, for
each $x : X$ transport $p$ along the evaluation equivalence $f_x$ to a
map on $Y\,x$, and let $\varphi_0\,x$ be the infimum of its roots in
$Y\,x$, which defines $\varphi_0 : \prod_x Y\,x$. Let $p$ have a root
$\gamma$. Given $x : X$, the point $\gamma\,x$ is a root of the
transported map on $Y\,x$, so the infimum $\varphi_0\,x$ is also a
root of it, and since $f_x^{-1}(\varphi_0\,x) = \varphi_0$, this gives
$p\,\varphi_0 = 0$. If $X$ is empty, then $\varphi_0 = \gamma$, the
product having a single point, and again $p\,\varphi_0 = 0$. These
combine as in step~(3) of the proof of~\Cref{thm:micro-tychonoff}. If
$p\,\varphi_0 = 1$, the first case refutes $X$, and then the second
gives $p\,\varphi_0 = 0$. So $p\,\varphi_0 = 0$, and $\varphi_0$ is a
conditional root. It remains to show that $\varphi_0$ is the infimum for
the relation~$\leq$.

To see that $\varphi_0$ is a lower bound, note that if $\gamma$ is a
root, then for each $x$ the point $\gamma\,x$ is a root of the
transported map on $Y\,x$, so $\varphi_0\,x \preceq \gamma\,x$, that
is, $\varphi_0 \leq \gamma$.

To see that it is greatest among lower bounds, let $\psi$ be below every
root. We must show $\psi\,x \preceq \varphi_0\,x$ for every $x$, so let
$x : X$ and suppose $\varphi_0\,x \lt{} \psi\,x$. Then $\psi\,x$ is
below every root $y$ of the transported map at $x$. Indeed, such a
$y$ gives the root $f_x^{-1}\,y$ of $p$, so that $\varphi_0$ is a root,
being a conditional root, and hence $\psi \leq \varphi_0$, since $\psi$
is below every root. This is refuted by the supposition, so the claim
holds vacuously. Since $\varphi_0\,x$ is the infimum of the roots of the
transported map, we get $\psi\,x \preceq \varphi_0\,x$, contradicting
the supposition.
\end{proof}

\begin{theorem}
\label{thm:kappa-infs}
For every Brouwer code~$b$, the topped ordinal $\compactsepint b$ has
infima of complemented subsets.
\AgdaRefs{\Agda{Ordinals.BrouwerCodesDiscreteAndCompactInterpretations}{Κ-has-infs-of-complemented-subsets},
using \Agda{Ordinals.ToppedType}{has-infs-of-complemented-subsets}.}
\end{theorem}

\begin{proof}
By induction on $b$, with the same three cases and the same shape
as~\Cref{thm:delta-kappa-props}, but now with closure properties for
infima in place of those for compactness.

For $\bZ$, the one-point ordinal has infima, since the unique point is a
conditional root and the infimum. For $\bS$,
use that the ordinals with infima of complemented subsets are closed under
ordinal-indexed sums, in the sense that if $\tau$ has them and every
$\upsilon\,x$ does, then so does $\osum\tau\upsilon$, the infimum
being computed lexicographically, first in the base, as the infimum of
those $x$ whose fiber contains a root, and then within that fiber. This
needs the lexicographic weak order on the sum to agree
with the weak order derived from the sum's own strict order. The
ordinal $\TwoT$ indexing the addition has infima, by case analysis on
its two points.
\AgdaRefs{\Agda{Ordinals.Closure}{∑-has-infs-of-complemented-subsets};
\texttt{𝟙ᵒ-has-infs-of-complemented-subsets};
\texttt{𝟚ᵒ-has-infs-of-complemented-subsets}.}

For $\bL$, unfold $\ssumone{}$ as the sum over the ordinal $\NInfO$ of
the extension along $\N\hookrightarrow\NInf$ and apply the same closure
property. Two things are needed. First, the ordinal $\NInfO$ has
infima of complemented subsets, which is the property of its selection
function stated in~\Cref{ex:compact-pt-basic}(\labelcref{item:NInf-compact}),
that the sequence it
returns is a conditional root and the infimum of the set of roots.
Second, each fiber of the extended family has infima.
This is exactly~\Cref{thm:micro-inf-tychonoff}, the micro-Tychonoff
theorem for infima, applied to the fiber of the embedding over
the point in question, which is a proposition.
\AgdaRefs{\Agda{Ordinals.Closure}{∑¹-has-infs-of-complemented-subsets}
and \texttt{ℕ∞ᵒ-has-infs-of-complemented-subsets};
\Agda{TypeTopology.MicroInfTychonoff}{micro-inf-tychonoff}.}
\end{proof}

\noindent
By~\Cref{lem:least-element}, we conclude the following.
\begin{corollary}
\label{cor:kappa-least}
For every Brouwer code~$b$, every non-empty complemented subset of
the ordinal $\compactsepint b$ has a least element.
\AgdaRefs{\Agda{Ordinals.BrouwerCodesDiscreteAndCompactInterpretations}{Κ-has-least-roots-of-complemented-subsets}.}
\end{corollary}

\noindent By~\Cref{rem:inf-gives-compact}, \Cref{thm:kappa-infs}
subsumes the compactness
of~\Cref{thm:four-interp-props}(\labelcref{item:compactsep-interp}).
The least element property fails constructively for the discrete, the
standard and the trichotomous interpretations, because in each of them
the ordinal $\succT{\omegaO}$ is the interpretation of a code, up to
order equivalence.

\begin{lemma}
\label{lem:omega-plus-one-least}
\leavevmode
\begin{enumerate}
\item\label{item:least-transfer} If $\alpha \iseq \beta$ and every
  non-empty complemented subset of $\alpha$ has a least element, then
  so does every non-empty complemented subset of $\beta$.
\item\label{item:omega-plus-one-least-lpo} If every non-empty
  complemented subset of $\succT{\omegaO}$ has a least element, then
  $\LPO$ holds.
\end{enumerate}
\AgdaRefs{\Agda{Ordinals.Closure}{≃ₒ-gives-has-least-roots};
\texttt{succₒ-ω-least-roots-gives-LPO}.}
\end{lemma}

\begin{proof}
(\labelcref{item:least-transfer}) Let $f$ be an order equivalence of
$\alpha$ with $\beta$, with inverse $g$, and let $q$ be the
characteristic function of a non-empty complemented subset of $\beta$.
Then $q \circ f$ is that of a non-empty complemented subset of
$\alpha$, since $g$ sends members to members. If $x_0$ is its least
element, then $f\,x_0$ is a member of the original subset and is below
every member $y$, because $y \lt{} f\,x_0$ would give
$g\,y \lt{} x_0$, as $g$ preserves the order.

(\labelcref{item:omega-plus-one-least-lpo}) The underlying type of
$\succT{\omegaO}$ is $\N+\One$. Given $\alpha : \N \to \Two$,
consider the complemented subset with characteristic function
\[
  p\,(\mathrm{inl}\,n) = \alpha\,n , \qquad
  p\,(\mathrm{inr}\,\pt) = 0 .
\]
It contains the added point, so it is non-empty, and its least element
$x_0$ is either $\mathrm{inl}\,n$ for some~$n$, in which case
$\alpha\,n = 0$, or the added point, in which case no
$\mathrm{inl}\,n$ is a member, since membership would give
$x_0 \leq \mathrm{inl}\,n$, that is,
$\neg(\mathrm{inl}\,n \lt{} \mathrm{inr}\,\pt)$. Inspection of $x_0$
tells the two cases apart, and they decide whether $\alpha$ has a
root, which is compactness of~$\N$, equivalently $\LPO$
(\Cref{ex:compact-basic}(\labelcref{item:N-compact-lpo})).
\end{proof}

\begin{proposition}
\label{prop:delta-least-lpo}
The statement that every non-empty complemented subset of the discrete
interpretation $\alttrichint b$ has a least element, for every Brouwer
code~$b$, is equivalent to $\LPO$.
\AgdaRefs{\Agda{Ordinals.BrouwerCodesDiscreteAndCompactInterpretations}{Δ-least-roots-iff-LPO},
combining \texttt{Δ-least-roots-gives-LPO} and
\texttt{LPO-gives-Δ-least-roots}, for the property
\Agda{Ordinals.ToppedType}{has-least-roots-of-complemented-subsets},
which is \Agda{Ordinals.InfProperty}{has-least-roots} for the
underlying weak order.}
\end{proposition}

\begin{proof}
For the forward implication, take $b = \bL(\lambda i.\,\bZ)$, whose
discrete interpretation satisfies
$\alttrichint b \iseq \big(\succT{\omegaO}\big)$
(\Cref{ex:omega-plus-one}), so
that~\Cref{lem:omega-plus-one-least}(\labelcref{item:least-transfer},\labelcref{item:omega-plus-one-least-lpo})
applies.

Conversely, $\LPO$ makes $\iemb_b$ an equivalence
(\Cref{prop:brouwer-iota-equiv-LPO}), and it is order-preserving and
order-reflecting (\Cref{thm:iota}), so $\alttrichint b$ and
$\compactsepint b$ are order equivalent. Every non-empty complemented
subset of $\compactsepint b$ has a least element
(\Cref{cor:kappa-least}), and this property transfers along the order
equivalence
(\Cref{lem:omega-plus-one-least}(\labelcref{item:least-transfer})).
\end{proof}

\begin{remark}
\label{rem:delta-least-single-code}
The forward implication needs only the single code
$\bL(\lambda i.\,\bZ)$, as in~\Cref{prop:brouwer-iota-equiv-LPO}.
\AgdaRefs{\Agda{Ordinals.BrouwerCodesDiscreteAndCompactInterpretations}{Δ𝟙-least-roots-gives-LPO}.}
\end{remark}

\begin{proposition}
\label{prop:std-trich-least-lpo}
\leavevmode
\begin{enumerate}
\item\label{item:std-least-lpo} If every non-empty complemented subset
  of the standard interpretation $\stdint b$ has a least element, for
  every Brouwer code~$b$, then $\LPO$ holds.
\item\label{item:trich-least-lpo} If every non-empty complemented
  subset of the trichotomous interpretation $\trichint b$ has a least
  element, for every Brouwer code~$b$, then $\LPO$ holds.
\end{enumerate}
\AgdaRefs{\Agda{Ordinals.BrouwerCodesInterpretations}{⟦\_⟧₀-least-roots-gives-LPO},
\texttt{⟦\_⟧₃-least-roots-gives-LPO}.}
\end{proposition}

\begin{proof}
(\labelcref{item:std-least-lpo}) The code
$\bS(\bL(n \mapsto \bS^n\,\bZ))$ is interpreted as
$\succT{\omegaO}$, because $\bL(n\mapsto \bS^n\,\bZ)$ is interpreted
as the supremum of the finite ordinals, which is~$\omegaO$.
So~\Cref{lem:omega-plus-one-least}(\labelcref{item:omega-plus-one-least-lpo})
applies.
\AgdaRefs{\Agda{Ordinals.BrouwerCodesInterpretations}{ω-is-a-standard-interpretation}.}

(\labelcref{item:trich-least-lpo}) The code
$\bS(\bL(\lambda i.\,\bS\,\bZ))$ is interpreted as an ordinal with
underlying type $(\Sigma_{i:\N}\,(\Zero+\One)) + \One$,
by~\Cref{def:non-standard-interpretations}. The map to $\N + \One$
that sends $\mathrm{inl}\,(n,\mathrm{inr}\,\pt)$ to $\mathrm{inl}\,n$
and the added point to the added point is an order equivalence
with~$\succT{\omegaO}$, and
so~\Cref{lem:omega-plus-one-least}(\labelcref{item:least-transfer},\labelcref{item:omega-plus-one-least-lpo})
apply.
\AgdaRefs{\Agda{Ordinals.BrouwerCodesInterpretations}{trichotomous-interpretation-of-ω+𝟙}.}
\end{proof}

\section{An inductive-recursive universe of compact ordinals}
\label{sec:inductive-recursive}

In the Brouwer codes (\Cref{def:brouwer-codes}), the branching of the
limit constructor $\bL$ is always indexed by $\N$. We now let the
branching of a sum be indexed by an ordinal built at an earlier stage,
which requires defining the type $\Ecodes$ of codes and their standard
interpretation $\Disc$ simultaneously by \emph{induction-recursion} in
the sense
of~\cite{Dybjer2000,DybjerSetzer1999,dybjersetzer:IndexedInductionRecursion:2006}.
The codes carry a further interpretation~$\Comp$, with an extremely
dense embedding~$\Disc\nu \hookrightarrow \Comp\nu$ for every code
$\nu:\Ecodes$ (\Cref{thm:E-props}(\labelcref{item:E-iota-embedding})).
The Brouwer codes embed into $\Ecodes$, with the inclusion making
$\alttrichint{-}$ commute with $\Disc$, and $\compactsepint{-}$ with
$\Comp$ (\Cref{thm:brouwer-to-E}).  The interpretation $\Disc$ gives
trichotomous and hence discrete ordinals (\Cref{thm:E-delta}),
and~$\Comp$ gives compact ordinals with infima of complemented subsets
(\Cref{thm:E-props}). The compact interpretation $\Comp$, however, is
no longer totally separated in general
(\Cref{prop:kappa-not-tot-sep}), and hence the ordinals witnessing
this are not denoted by Brouwer codes.

\subsection{Ordinal expressions and their discrete interpretation}

\begin{definition}[$\Ecodes$ and $\Disc$]
\label{def:E}
We define, by induction-recursion, a type $\Ecodes$ of ordinal
expressions by the constructors
\[
  \eOne, \eOmegaPlusOne : \Ecodes, \qquad
  \eAdd, \eMul : \Ecodes \to \Ecodes \to \Ecodes,
\]
\[
  \eSig : \textstyle\prod_{\nu : \Ecodes}\,
  \big((\us{\Disc\nu} \to \Ecodes) \to \Ecodes\big) ,
\]
simultaneously with the standard, \emph{discrete interpretation}
$\Disc : \Ecodes \to \OrdT[]$, by
\begin{align*}
\Disc\,\eOne &= \OneT
&
\Disc\,\eOmegaPlusOne &= \succT{\omegaO}
\\
\Disc\,(\nu_0\eAdd\nu_1) &= \Disc\nu_0 \oplusT \Disc\nu_1
&
\Disc\,(\nu_0\eMul\nu_1) &= \osum{\Disc\nu_0}{}(\lambda\_.\,\Disc\nu_1)
\\
\Disc\,(\eSig\,\nu\,A) &= \osum{\Disc\nu}{}(\Disc\circ A).
\end{align*}
\AgdaRefs{\Agda{Ordinals.InductiveRecursiveCodesInterpretations}{E}; \texttt{Δ}.}
\end{definition}

\begin{theorem}
\label{thm:E-delta}
For every $\nu:\Ecodes$, the underlying type of $\Disc\nu$ is a
retract of~$\N$, hence discrete, and moreover the ordinal $\Disc\nu$ is
trichotomous.
\AgdaRefs{\Agda{Ordinals.InductiveRecursiveCodesInterpretations}{Δ-retract-of-ℕ};
  \texttt{Δ-is-discrete}; \texttt{Δ-is-trichotomous}.}
\end{theorem}

\begin{proof}
By induction on $\nu$, which for an inductive-recursive definition
means induction by the elimination principle of $\Ecodes$, in which the
hypothesis for $\eSig\,\nu\,A$ may use the already-available value
$\Disc\nu$ in order to speak about the index type
$\us{\Disc\nu}$ of $A$.

For $\eOne$ the underlying type is $\One$. For $\eOmegaPlusOne$ it is
$\N+\One$, which is equivalent to $\N$. The underlying type of an
addition is a $\Sigma$ over the index type $\One+\One$
(\Cref{sec:ordinal-arithmetic}), and that of a multiplication
is a $\Sigma$ over a constant family
(\Cref{def:ordinal-multiplication}), so the cases $\eAdd$ and $\eMul$
are, as in the proof
of~\Cref{thm:delta-kappa-props}(\labelcref{item:delta-retract-N}),
instances of the closure of the retracts of $\N$ under $\Sigma$,
the index type $\One+\One$ being a retract of $\N$. So is the case
$\eSig\,\nu\,A$, where $\Disc(\eSig\,\nu\,A)$ is the plain
ordinal-indexed sum $\osum{\Disc\nu}{}(\Disc\circ A)$, with no extension,
whose base $\us{\Disc\nu}$ and fibers
$\us{\Disc(A\,x)}$ are retracts of $\N$ by the induction hypothesis.
Discreteness then
follows because $\N$ is discrete and discreteness is closed under
retracts
(\Cref{prop:discrete-closure}(\labelcref{item:discrete-retract})).

Trichotomy is proved by a separate induction on the same cases,
using that $\One$ is trichotomous, that $\omega$ is and hence so is
its successor, and that trichotomy is preserved by ordinal-indexed
sums of trichotomous families over a trichotomous index, as in the
proof
of~\Cref{thm:delta-kappa-props}(\labelcref{item:delta-trichotomous}),
which covers addition and multiplication under the above readings of
them as sums.
\AgdaRefs{\Agda{Ordinals.ToppedArithmetic}{∑-is-trichotomous},
\texttt{+ᵒ-is-trichotomous}, \texttt{×ᵒ-is-trichotomous}.}
\end{proof}

\noindent
Trichotomy gives a second route to the discreteness of $\Disc\nu$,
by~\Cref{lem:trichotomy-gives-discrete}.
However, compactness fails constructively.
\begin{proposition}
\label{prop:E-delta-lpo}
The statement that the type $\us{\Disc\nu}$ is compact for every code
$\nu : \Ecodes$ is equivalent to $\LPO$.
\AgdaRefs{\Agda{Ordinals.InductiveRecursiveCodesInterpretations}{Δ-compact-iff-LPO},
  combining \texttt{Δ-compact-gives-LPO} and
  \texttt{LPO-gives-Δ-compact}.}
\end{proposition}

\begin{proof}
Under $\LPO$ the type $\N$ is compact
(\Cref{ex:compact-basic}(\labelcref{item:N-compact-lpo})), and
$\us{\Disc\nu}$ is a retract of it by~\Cref{thm:E-delta}, so it is
compact too, compactness being closed under retracts
(\Cref{thm:compact-closure}(\labelcref{item:compact-retract})).
Conversely, the single code $\eOmegaPlusOne$ suffices, since
$\us{\Disc\eOmegaPlusOne}$ is $\N+\One$, which has $\N$ as a retract,
so the same closure property gives the compactness of $\N$, which is
$\LPO$.
\end{proof}

\subsection{The compact interpretation}
\begin{definition}[$\Comp$ and $\iemb_\nu$]
  \label{def:E-kappa}
  We simultaneously define
\[
  \begin{array}{lll}
    \Comp & : & \Ecodes \to \OrdT[], \\
    \iemb_\nu & : & \Disc\nu \to \Comp\nu
  \end{array}
\]
and prove by induction that $\iemb_\nu$ is an embedding as follows.
For a code $\nu : \Ecodes$ and
$A : \us{\Disc\nu}\to\Ecodes$, write
$(\Comp\circ A) \extend \iemb_\nu$ for the extension along
$\iemb_\nu:\us{\Disc\nu}\hookrightarrow\us{\Comp\nu}$ of
$\Comp\circ A : \us{\Disc\nu}\to\OrdT[]$ to a family
$\us{\Comp\nu}\to\OrdT[]$, as illustrated in the following diagram
\AgdaRefs{\Agda{Ordinals.Injectivity}{}, applied to the algebraic
  injectivity of $\OrdT[]$ along embeddings}:
\begin{center}
\begin{tikzcd}[row sep=huge, column sep=huge]
  \us{\Disc\nu} \arrow[r, hook, "\iemb_\nu"] \arrow[d, "A"'] \arrow[dr, dashed] & \us{\Comp\nu} \arrow[d, dashed, "(\Comp\circ A)\extend\iemb_\nu"] \\
  \Ecodes \arrow[r, "\Comp"'] & \OrdT[].
\end{tikzcd}
\end{center}
We define $\Comp : \Ecodes \to \OrdT[]$ and
$\iemb_\nu : \us{\Disc\nu}\to\us{\Comp\nu}$ by
\begin{align*}
\Comp\,\eOne &= \OneT
&
\Comp\,\eOmegaPlusOne &= \NInfO
\\
\Comp\,(\nu_0\eAdd\nu_1) &= \Comp\nu_0 \oplusT \Comp\nu_1
&
\Comp\,(\nu_0\eMul\nu_1) &= \osum{\Comp\nu_0}{}(\lambda\_.\,\Comp\nu_1)
\\
\Comp\,(\eSig\,\nu\,A) &= \osum{\Comp\nu}{}\big((\Comp\circ A)\extend\iemb_\nu\big)
\end{align*}
\begin{align*}
  \iemb_{\eOne} &= \mathrm{id}
&
  \iemb_{\eOmegaPlusOne} &= \iota_1
\\
  \iemb_{\nu_0\eAdd\nu_1}\,(\mathrm{inl}\,\pt,\,y)
  &= (\mathrm{inl}\,\pt,\,\iemb_{\nu_0}\,y)
&
  \iemb_{\nu_0\eMul\nu_1}\,(x,\,y)
&= (\iemb_{\nu_0}\,x,\,\iemb_{\nu_1}\,y)
\\
  \iemb_{\nu_0\eAdd\nu_1}\,(\mathrm{inr}\,\pt,\,y)
&= (\mathrm{inr}\,\pt,\,\iemb_{\nu_1}\,y)
\\
  \iemb_{\eSig\,\nu\,A}\,(x,\,z)
    &= \big(\iemb_\nu\,x,\;\varphi_x^{-1}(\iemb_{A\,x}\,z)\big).
\end{align*}
Here $\iota_1$ is the canonical embedding
$\N+\One\hookrightarrow\NInf$ sending the added point to $\infty$, and
$\varphi_x : \big((\Comp\circ A)\extend\iemb_\nu\big)(\iemb_\nu\,x) \iseq
\Comp(A\,x)$ is the extension property
of~\Cref{lem:extension-restricts} at the point $\iemb_\nu\,x$ of the
image.

That each $\iemb_\nu$ is an embedding is proved by the same induction.
At $\eOne$ the map is the identity and at $\eOmegaPlusOne$ it is the
canonical embedding $\iota_1$, an embedding
by~\Cref{ex:isolated}(\labelcref{item:iota1-embedding-dense}). The
identity is an embedding, a map of sums is an embedding if its index
map and all its fiber maps are
(\Cref{lem:sum-map}(\labelcref{item:sum-map-embedding})), embeddings
are closed under composition, and an equivalence is an embedding. At
$\eAdd$, where the addition is a sum over $\One+\One$
(\Cref{sec:ordinal-arithmetic}), the index map is the identity on
$\One+\One$ and the fiber maps are $\iemb_{\nu_0}$ and
$\iemb_{\nu_1}$, and at $\eMul$, where the multiplication is a sum
over a constant family (\Cref{def:ordinal-multiplication}), the index
map is $\iemb_{\nu_0}$ and the fiber map is $\iemb_{\nu_1}$, the
maps $\iemb_{\nu_0}$ and $\iemb_{\nu_1}$ being embeddings by the
induction hypothesis. At $\eSig\,\nu\,A$ the map is
the map of sums with index map $\iemb_\nu$ and fiber maps
$\varphi_x^{-1}\circ\iemb_{A\,x}$, an embedding as the composition of
the embedding $\iemb_{A\,x}$ with the equivalence~$\varphi_x^{-1}$.
\AgdaRefs{\Agda{Ordinals.InductiveRecursiveCodesInterpretations}{Κ};
  \texttt{ι}; the extension diagram is spelled out loc.\ cit.\ via the
  auxiliary family \texttt{𝓚 ν A} $: \us{\Comp\nu}\to\OrdT[]$, defined
  as $(\Comp\circ A)\extend\iemb_\nu$.  The embedding property is
  \texttt{ι-is-embedding} loc.\ cit., built from
  \Agda{UF.PairFun}{pair-fun-is-embedding},
  \Agda{UF.Embeddings}{id-is-embedding}, \texttt{∘-is-embedding} and
  \texttt{equivs-are-embeddings'}, and
  \Agda{TypeTopology.GenericConvergentSequence}{ι𝟙-is-embedding}.}
\end{definition}

The following generalizes \Cref{thm:kappa-infs,thm:iota}.

\begin{theorem}
\label{thm:E-props}
For every $\nu:\Ecodes$,
\begin{enumerate}
\item\label{item:E-comp-infima} the ordinal $\Comp\nu$ has infima of
  complemented subsets, hence is compact,
\item\label{item:E-iota-embedding}
  the embedding $\iemb_\nu : \us{\Disc\nu}\to\us{\Comp\nu}$ is extremely dense,
  order-preserving and order-reflecting.
\end{enumerate}
\AgdaRefs{\Agda{Ordinals.InductiveRecursiveCodesInterpretations}{K-has-infs-of-complemented-subsets};
\texttt{𝓚-has-infs-of-complemented-subsets};
\texttt{Κ-Compact}, which is proved there by a direct induction on the
code, in the pointed form, rather than through the infima;
\texttt{ι-is-embedding}; \texttt{ι-is-dense};
\texttt{ι-is-order-preserving}; \texttt{ι-is-order-reflecting}.}
\end{theorem}
\begin{proof}
Write
$
  \Comp_{\nu,A} \;=\; (\Comp\circ A)\extend\iemb_\nu \;:\; \us{\Comp\nu}\to\OrdT[]
$
so that
$\Comp(\eSig\,\nu\,A) = \osum{\Comp\nu}{}\,\Comp_{\nu,A}$, and
by~\Cref{def:extension} its value at $y$ is the product of the
$\us{\Comp(A\,x)}$ over the fiber of $\iemb_\nu$ at $y$.
We prove (\labelcref{item:E-comp-infima}) by induction on $\nu$,
using in the case $\eSig\,\nu\,A$ that
\begin{enumerate}
\item [($\dagger$)] \looseness=-1
  the ordinal $\Comp_{\nu,A}\,y$ has infima of
  complemented subsets for every $y : \us{\Comp\nu}$.
\end{enumerate}

For (\labelcref{item:E-comp-infima}), the ordinal $\OneT$ has infima,
its only point being the infimum of every subset. The ordinal
$\NInfO$ has them, by the property of its selection function stated
in~\Cref{ex:compact-pt-basic}(\labelcref{item:NInf-compact}), that the
point it returns is a conditional root and the infimum of the set of
roots. The cases $\eAdd$, $\eMul$ and $\eSig$ are all instances of the
closure of this property under ordinal-indexed sums shown in the proof
of~\Cref{thm:kappa-infs}, the index $\TwoT$ of the addition having
infima by case analysis on its two points, applied in the last case to
the family $\Comp_{\nu,A}$, using ($\dagger$).
\AgdaRefs{\Agda{Ordinals.Closure}{ℕ∞ᵒ-has-infs-of-complemented-subsets};
\texttt{∑-has-infs-of-complemented-subsets};
\texttt{𝟚ᵒ-has-infs-of-complemented-subsets}.}

For ($\dagger$), the ordinal $\Comp_{\nu,A}\,y$ is a product of the
ordinals
$\Comp(A\,x)$ indexed by $\fib{\iemb_\nu}{y}$, which is a proposition
because $\iemb_\nu$ is an embedding, and each factor has infima by the
induction hypothesis for the codes $A\,x$. So~\Cref{thm:micro-inf-tychonoff}
applies.

Compactness then follows from~\Cref{rem:inf-gives-compact}.

(\labelcref{item:E-iota-embedding}) That $\iemb_\nu$ is an embedding is
part of~\Cref{def:E-kappa}. Each of the three remaining properties is
then inherited exactly as in~\Cref{thm:iota}. All three are preserved
by maps of sums
(\Cref{lem:sum-map}(\labelcref{item:sum-map-dense})
and~\Cref{lem:sum-map-order}, order reflection requiring, as there,
that the index map be an embedding), all three are preserved by
composition, and equivalences have all three. The only new base case
is $\eOmegaPlusOne$, where
$\N+\One\hookrightarrow\NInf$ is extremely dense
(\Cref{ex:isolated}(\labelcref{item:iota1-embedding-dense}))
and is order-preserving and order-reflecting, the predecessors of a
point $v$ of $\NInfO$ being the finite points $\underline n$ with
$n \sqlt v$ (\Cref{ex:NInf-ordinal}).
\AgdaRefs{\Agda{TypeTopology.GenericConvergentSequence}{ι𝟙-dense};
\Agda{Ordinals.Closure}{ι𝟙ᵒ-is-order-preserving},
\texttt{ι𝟙ᵒ-is-order-reflecting}; \Agda{UF.PairFun}{pair-fun-dense}
and companions.}
\end{proof}

\begin{corollary}
\label{cor:E-least}
Every non-empty complemented subset of the ordinal $\Comp\nu$ has a
least element for any $\nu : \Ecodes$.
\AgdaRefs{\Agda{Ordinals.InductiveRecursiveCodesInterpretations}{K-has-least-roots-of-complemented-subsets}.}
\end{corollary}

\begin{proof}
By~\Cref{lem:least-element}, applied
to~\Cref{thm:E-props}(\labelcref{item:E-comp-infima}) with respect to
the weak order of the ordinal.
\end{proof}

Classically, the embedding
$\iemb_\nu : \us{\Disc\nu} \hookrightarrow \us{\Comp\nu}$ is a surjection
and hence an equivalence, but this fails constructively, and so does
the discreteness, and hence the trichotomy, of~$\Comp\nu$.

\begin{proposition}
\label{prop:E-iota-equiv-discrete}
\leavevmode
\begin{enumerate}
\item\label{item:E-iota-equiv} The map $\iemb_\nu$ is an equivalence
  for every $\nu : \Ecodes$ if and only if $\LPO$ holds.
  \AgdaRefs{\Agda{Ordinals.InductiveRecursiveCodesInterpretations}{ι-is-equiv-iff-LPO},
    combining \texttt{ι-is-equiv-gives-LPO}, which only needs the
    single instance $\nu = \eOmegaPlusOne$, and
    \texttt{LPO-gives-ι-is-equiv}, which needs the full induction on
    $\nu$.}
\item\label{item:E-kappa-discrete-lpo} $\LPO$ makes $\Comp\nu$
  discrete for every $\nu : \Ecodes$.
  \AgdaRefs{\Agda{Ordinals.InductiveRecursiveCodesInterpretations}{LPO-gives-Κ-discrete}.}
\item\label{item:E-kappa-discrete-wlpo} If $\Comp\nu$ is
  discrete for every $\nu : \Ecodes$, then $\WLPO$ holds.
  \AgdaRefs{\Agda{Ordinals.InductiveRecursiveCodesInterpretations}{Κ-discrete-gives-WLPO}.}
\end{enumerate}
\end{proposition}

\begin{proof}
(\labelcref{item:E-iota-equiv}) ($\Rightarrow$) Instantiate at $\nu = \eOmegaPlusOne$, where $\iemb_\nu$ is
the canonical map $\N+\One\to\NInf$. If it is an equivalence, then
every conatural number is either the image of a unique $n : \N$ or is
$\infty$, and the inverse decides which. Deciding, for an arbitrary
$u$, whether $u$ is finite is $\LPO$, as in the proof
of~\Cref{prop:failure-trichotomy}.
\AgdaRefs{\Agda{Taboos.LPO}{ι𝟙-is-equiv-gives-LPO}.}

($\Leftarrow$) By induction on $\nu$. At $\eOne$ the map is the
identity. At $\eOmegaPlusOne$, the principle $\LPO$ says exactly that
finiteness is decidable, which makes $\N+\One\to\NInf$ split
surjective and hence, being an embedding already
(\Cref{ex:isolated}(\labelcref{item:iota1-embedding-dense})), an
equivalence. At
$\eAdd$, $\eMul$ and
$\eSig$ the map is the map of sums
of~\Cref{def:E-kappa}, which is an equivalence when its index map and
all its fiber maps are
(\Cref{lem:sum-map}(\labelcref{item:sum-map-equiv})). At $\eSig$ the
fiber map is additionally composed with the equivalence
$\varphi_x^{-1}$, which does not affect the conclusion.
\AgdaRefs{\Agda{UF.PairFun}{pair-fun-is-equiv}.}

(\labelcref{item:E-kappa-discrete-lpo}) Under $\LPO$ the map
$\iemb_\nu$ is an equivalence for every $\nu$
by~(\labelcref{item:E-iota-equiv}), so $\us{\Comp\nu}$ is discrete,
being equivalent to the discrete type $\us{\Disc\nu}$
(\Cref{thm:E-delta}).
\AgdaRefs{\Agda{Ordinals.InductiveRecursiveCodesInterpretations}{ι-is-equiv-gives-Κ-discrete}.}

(\labelcref{item:E-kappa-discrete-wlpo}) Instantiate at
$\nu = \eOmegaPlusOne$ again, the discreteness of $\NInf$ being
$\WLPO$, since deciding $u = \infty$ for every $u$ is deciding whether
a binary sequence is constantly~$1$.
\AgdaRefs{\Agda{Taboos.WLPO}{ℕ∞-discrete-gives-WLPO}.}
\end{proof}

\noindent
For the discrete interpretation the least element property fails
constructively.
\begin{proposition}
\label{prop:E-delta-least-lpo}
The statement that every non-empty complemented subset of the discrete
interpretation $\Disc\nu$ has a least element, for every code
$\nu : \Ecodes$, is equivalent to $\LPO$.
\AgdaRefs{\Agda{Ordinals.InductiveRecursiveCodesInterpretations}{Δ-least-roots-iff-LPO},
  combining \texttt{Δ-least-roots-gives-LPO} and
  \texttt{LPO-gives-Δ-least-roots}.}
\end{proposition}

\begin{proof}
The proof is that of~\Cref{prop:delta-least-lpo}. In the forward
implication the single code $\eOmegaPlusOne$ takes the place of
$\bL(\lambda i.\,\bZ)$, its discrete interpretation being
$\succT{\omegaO}$. In the converse, $\LPO$ makes $\iemb_\nu$ an
equivalence
(\Cref{prop:E-iota-equiv-discrete}(\labelcref{item:E-iota-equiv})),
and it is order-preserving and order-reflecting
(\Cref{thm:E-props}(\labelcref{item:E-iota-embedding})), so
$\Disc\nu$ and $\Comp\nu$ are order equivalent. Every non-empty
complemented subset of $\Comp\nu$ has a least element
(\Cref{cor:E-least}), and this property transfers along the order
equivalence
(\Cref{lem:omega-plus-one-least}(\labelcref{item:least-transfer})).
\end{proof}

The embedding $\iemb_\nu$ is order-preserving and order-reflecting,
and it is natural to ask whether it additionally satisfies the
initial-segment condition, so that $\Disc\nu \oemb \Comp\nu$ for all
codes $\nu$. The code $\eOmegaPlusOne$ does satisfy this, but the
requirement that it hold for every code fails constructively.

\begin{proposition}
\label{prop:E-delta-below-kappa}
\leavevmode
\begin{enumerate}
\item\label{item:E-iota1-simulation} The canonical embedding
  $\iemb_{\eOmegaPlusOne} : \N+\One \hookrightarrow \NInf$ is a
  simulation, so that
  $\Disc\eOmegaPlusOne \oemb \Comp\eOmegaPlusOne$.
  \AgdaRefs{\Agda{Ordinals.ConvergentSequence}{ω+𝟙-is-⊴-ℕ∞}.}
\item\label{item:E-delta-below-kappa-lpo} The relation
  $\Disc\nu \oemb \Comp\nu$ holds for every $\nu : \Ecodes$ if and only
  if $\LPO$ holds.
  \AgdaRefs{\Agda{Ordinals.InductiveRecursiveCodesInterpretations}{Δ-⊴-Κ-iff-LPO},
    combining \texttt{Δ-⊴-Κ-gives-LPO}, which only needs the single
    code $\eOmegaPlusOne\eAdd\eOne$, written \texttt{⌜ω+𝟚⌝} there, and
    \texttt{LPO-gives-Δ-⊴-Κ}, which needs the full induction on $\nu$.}
\end{enumerate}
\end{proposition}

\begin{proof}
(\labelcref{item:E-iota1-simulation}) This
is~\Cref{lem:omega-plus-two}(\labelcref{item:iota1-simulation}), the
map $\iemb_{\eOmegaPlusOne}$ being $\iota_1$.

(\labelcref{item:E-delta-below-kappa-lpo}) ($\Leftarrow$) Under $\LPO$
the map $\iemb_\nu$ is an equivalence
(\Cref{prop:E-iota-equiv-discrete}(\labelcref{item:E-iota-equiv})),
and it is order-preserving and order-reflecting
(\Cref{thm:E-props}(\labelcref{item:E-iota-embedding})), so it is an
order equivalence, and an order equivalence is a simulation.

($\Rightarrow$) The single code $\nu = \eOmegaPlusOne \eAdd \eOne$
suffices. Its discrete interpretation is the sum of $\succT{\omegaO}$
and $\OneT$, and its compact interpretation is the sum of $\NInfO$ and
$\OneT$,
so~\Cref{lem:omega-plus-two}(\labelcref{item:omega-plus-two-lpo})
applies.
\end{proof}

\subsection{Inclusion of Brouwer codes into inductive-recursive codes}
\label{sec:brouwer-into-E}

The two notation systems of this paper are related by an inclusion of
the first into the second, under which both interpretations are
preserved.

\begin{definition}
\label{def:brouwer-to-E}
The map $\bemb : \Brouwer \to \Ecodes$ is defined by induction, together
with the family $A_b$ of the limit clause, whose index type
$\us{\Disc\eOmegaPlusOne}$ is $\N+\One$:
\[
  \bemb\,\bZ = \eOne, \qquad
  \bemb\,(\bS\,b) = \bemb\,b \eAdd \eOne, \qquad
  \bemb\,(\bL\,b) = \eSig\,\eOmegaPlusOne\,A_b ,
\]
\[
  A_b\,(\mathrm{inl}\,n) = \bemb\,(b\,n), \qquad
  A_b\,(\mathrm{inr}\,\star) = \eOne .
\]
\AgdaRefs{\Agda{Ordinals.BrouwerCodesIntoInductiveRecursiveCodes}{B-to-E}.}
\end{definition}

\begin{theorem}
\label{thm:brouwer-to-E} \leavevmode
\begin{enumerate}
\item\label{item:E-is-set} The type $\Ecodes$ of codes is a set.
\item\label{item:bemb-embedding} The map $\bemb : \Brouwer \to \Ecodes$ is an embedding.
\item\label{item:bemb-discrete} For every Brouwer code $b$ we have $\alttrichint b \iseq \Disc(\bemb\,b)$.
\item\label{item:bemb-compact} For every Brouwer code $b$ we have $\compactsepint b \iseq \Comp(\bemb\,b)$.
\end{enumerate}
\AgdaRefs{\Agda{Ordinals.InductiveRecursiveCodesInterpretations}{E-is-set};
\Agda{W.Properties}{W-is-set};
\Agda{Ordinals.BrouwerCodesIntoInductiveRecursiveCodes}{B-to-E-is-embedding},
\texttt{Δ-agreement}, \texttt{Κ-agreement}.}
\end{theorem}

\begin{proof}
  (\labelcref{item:E-is-set}) Let $W$ be the
  $\mathsf{W}$-type~\cite[Section~5.3]{HoTTBook} whose shapes are the
  five constructors of $\Ecodes$, of arities $\Zero$, $\Zero$,
  $\One+\One$, $\One+\One$ and $\One+\N$, an element of which is a
  pair of a shape and a family of elements of $W$ indexed by the arity
  of that shape.
  By~\Cref{thm:E-delta} the type $\us{\Disc\nu}$ is a retract of~$\N$,
  with section $\sigma_\nu : \us{\Disc\nu} \to \N$ and retraction
  $\rho_\nu : \N \to \us{\Disc\nu}$. A family indexed by
  $\us{\Disc\nu}$ is determined by its composite with $\rho_\nu$, so
  that $\N$ can serve as the arity of the shape of $\eSig$.
  Define $f : \Ecodes \to W$ by induction on codes, writing $[\;]$ for
  the empty family, $[x,y]$ for the family over $\One+\One$ with
  values $x$ and $y$, and $[x,t]$ for the family over $\One+\N$ with
  value $x$ at the point of $\One$ and $t$ on $\N$:
  \begin{align*}
    f\,\eOne &= (\eOne, [\;]),
    &
    f\,\eOmegaPlusOne &= (\eOmegaPlusOne, [\;]),
    \\
    f\,(\nu_0\eAdd\nu_1) &= (\eAdd, [f\,\nu_0, f\,\nu_1]),
    &
    f\,(\nu_0\eMul\nu_1) &= (\eMul, [f\,\nu_0, f\,\nu_1]),
    \\
    f\,(\eSig\,\nu\,A) &= (\eSig, [f\,\nu, f \circ A \circ \rho_\nu]).
  \end{align*}
  The map $f$ is left cancellable, by induction on codes. Two codes
  with the same image have the same shape, hence are given by the same
  constructor, and, the type of shapes being a set, their branching
  families agree. At $\eSig$ this gives, by the induction hypothesis,
  $\nu = \nu'$ and $A\,(\rho_\nu\,n) = A'\,(\rho_\nu\,n)$ for every
  $n : \N$, and then
  $A\,x = A\,(\rho_\nu\,(\sigma_\nu\,x)) =
  A'\,(\rho_\nu\,(\sigma_\nu\,x)) = A'\,x$ for every
  $x : \us{\Disc\nu}$, so that $A = A'$.
  A $\mathsf{W}$-type whose type of shapes is a set is a set, and a
  type admitting a left-cancellable map into a set is a set, which
  shows that $\Ecodes$ is a set.

(\labelcref{item:bemb-embedding}) In view of
(\labelcref{item:E-is-set}), it is enough to show that $\bemb$ is left
cancellable.  The constructors of $\Ecodes$ have disjoint images and
are left cancellable in each argument, so if $\bemb\,b = \bemb\,b'$
then $b$ and $b'$ are given by the same constructor of~$\Brouwer$, and
we argue by induction. At $\bZ$ there is nothing to prove. At $\bS$
the hypothesis identifies $\bemb\,b\eAdd\eOne$ with
$\bemb\,b'\eAdd\eOne$, hence $\bemb\,b$ with $\bemb\,b'$, so that
$b = b'$ by the induction hypothesis. At $\bL$ it identifies
$\eSig\,\eOmegaPlusOne\,A_b$ with $\eSig\,\eOmegaPlusOne\,A_{b'}$,
hence $A_b$ with $A_{b'}$, and evaluation at $\mathrm{inl}\,n$ gives
$b\,n = b'\,n$ for every $n$, again by the induction hypothesis.

(\labelcref{item:bemb-discrete}) and~(\labelcref{item:bemb-compact})
The two agreements are proved by induction on $b$. At $\bZ$ both sides
are the one-point ordinal. At $\bS$ both sides are additions of $\OneT$, that
is, sums over $\One+\One$ (\Cref{sec:ordinal-arithmetic}) with the same
index, so it is enough to compare their two summands, the left ones by
the induction hypothesis and the right ones being $\OneT$ on both
sides. Order equivalent families over a common index have order
equivalent sums,
by~\Cref{lem:sum-map}(\labelcref{item:sum-map-equiv})
and~\Cref{lem:sum-map-order} with the identity on the index.

At $\bL$ both sides are again sums over a common index, so again it is
enough to compare the fibers.

In~(\labelcref{item:bemb-discrete}) the index is $\succT\omegaO$, whose
underlying type is $\N+\One$. The fibers on the left are the extension
of the family $n \mapsto \alttrichint{b\,n}$ along
$\mathrm{inl} : \N\hookrightarrow\N+\One$ (\Cref{def:extended-and-successor-sum}(\labelcref{item:successor-sum})),
and those on the right are the values of $\Disc\circ A_b$
(\Cref{def:E}). At $\mathrm{inl}\,n$ the extension gives
$\alttrichint{b\,n}$ back (\Cref{lem:extension-restricts}), which the
induction hypothesis compares with $\Disc(\bemb\,(b\,n))$. At
$\mathrm{inr}\,\star$ the fiber of $\mathrm{inl}$ is empty, so the
extension is the one-point ordinal
(\Cref{lem:extension-property}(\labelcref{item:ext-off-image})), and so
is $\Disc\,\eOne$.

In~(\labelcref{item:bemb-compact}) the index is $\NInfO$, but the two
families are extensions along different embeddings into $\NInf$, namely
$\iota : \N\hookrightarrow\NInf$ on the left (\Cref{def:extended-and-successor-sum}(\labelcref{item:extended-sum}))
and $\iemb_{\eOmegaPlusOne} : \N+\One\hookrightarrow\NInf$ on the right
(\Cref{def:E-kappa}). Write $h_n$ for the order equivalence
$\compactsepint{b\,n} \iseq \Comp(\bemb\,(b\,n))$ given by the
induction hypothesis. The comparison at a point $u : \NInf$ is defined
on the points of the members of the two families at $u$. A point of
the member on the left is a function $\varphi$ assigning to each $(n,p) : \fib\iota u$ a point
of $\compactsepint{b\,n}$, and it is
sent to the function assigning to
$(\mathrm{inl}\,n, p) : \fib{\iemb_{\eOmegaPlusOne}}u$ the point
$h_n(\varphi\,(n,p))$, and to
$(\mathrm{inr}\,\star, p)$ the point of the one-point ordinal
$\Comp\,\eOne$. This is an equivalence, its inverse restricting a
function on $\fib{\iemb_{\eOmegaPlusOne}}u$ along
$(n,p) \mapsto (\mathrm{inl}\,n, p)$ and applying $h_n^{-1}$. It is
order-preserving and order-reflecting, because a witness of the order
relation is given at a point of the fiber of the embedding, and at
$(\mathrm{inr}\,\star, p)$ there is none, the order of the one-point
ordinal being empty.
\AgdaRefs{\Agda{Ordinals.Closure}{∑-≃ₒ};
\Agda{Ordinals.Injectivity}{↗-out-of-range}, \texttt{↗-propertyₒ}.}
\end{proof}

As opposed to the $\compactsepint{-}$ interpretation of Brouwer codes,
the $\Comp$ interpretation of $\Ecodes$ codes fails to be totally
separated constructively.

\begin{proposition}
\label{prop:kappa-not-tot-sep}
If $\Comp\nu$ is totally separated for every code $\nu$, then
$\neg\neg\WLPO$ holds.
\AgdaRefs{\Agda{Ordinals.InductiveRecursiveCodesInterpretations}{Κ-totally-separated-gives-¬¬WLPO},
  \texttt{⌜ℕ∞₂⌝} for the code below, and
  \texttt{Κ⌜ℕ∞₂⌝-totally-separated-gives-¬¬WLPO} for what it gives.}
\end{proposition}

\begin{proof}
Take $\nu = \eSig\,\eOmegaPlusOne\,A$, where $A$ sends the finite
points of $\Disc\eOmegaPlusOne$ to $\eOne$ and its added point to
$\eOne \eAdd \eOne$. By~\Cref{def:E-kappa}, the ordinal $\Comp\nu$ is
the sum over $\NInfO$ of the extension of $\Comp \circ A$ along
$\iemb_{\eOmegaPlusOne} : \N+\One \hookrightarrow \NInf$, so a point of
its underlying type is a point $u : \NInf$ together with a function
assigning to each element of $\fib{\iemb_{\eOmegaPlusOne}}u$ a point of
the corresponding member of the family. The elements with first
component $\mathrm{inl}\,n$ contribute the one-point type
$\us{\Comp\,\eOne}$, and the element with first component
$\mathrm{inr}\,\star$, which is available exactly when $u = \infty$,
contributes $\us{\Comp(\eOne \eAdd \eOne)}$, a sum over $\One+\One$ of
one-point ordinals (\Cref{sec:ordinal-arithmetic}), which we identify
with the type $\Two$. Hence
\[
  \us{\Comp\nu} \;\simeq\;
  \textstyle\Sigma_{u:\NInf}\,\Two^{u = \infty},
\]
the type
of~\Cref{ex:tot-sep-failures}(\labelcref{item:two-infinities-not-tot-sep}),
whose total separatedness gives $\neg\neg\WLPO$. Total separatedness is
closed under equivalence
(\Cref{prop:tot-sep-closure}(\labelcref{item:tot-sep-retract})), so the
total separatedness of $\Comp\nu$ gives $\neg\neg\WLPO$ too.
\end{proof}

\begin{remark}
\label{rem:kappa-tot-sep-contrast}
Since the compact interpretation $\compactsepint{-}$ of a Brouwer code
is totally separated (\Cref{thm:four-interp-props}(\labelcref{item:compactsep-interp}))
and $\Comp$ agrees with it along the inclusion
of~\Cref{thm:brouwer-to-E}, the ordinals witnessing the failure are
not in general denoted by Brouwer codes, so that $\Ecodes$ denotes
strictly more compact ordinals.
\end{remark}

\subsection{Isolated points and limit points}

A point in the image of the embedding
$\iemb_\nu : \Disc \nu \to \Comp \nu$ can be isolated in $\Comp\nu$ or
a limit point. A boolean valued function on $\us{\Disc\nu}$, defined
by induction on the code~$\nu$, determines which of the two holds at
each point.
Everything in this subsection holds verbatim for the
Brouwer codes of~\Cref{sec:discrete-compact}, with the same proofs, and
we develop it here because the inductive-recursive codes are the more
general system.
\claim{limit-points-for-brouwer-codes}
\AgdaRefs{\Agda{Ordinals.BrouwerCodesDiscreteAndCompactInterpretations}{ℓ},
\texttt{ℓ-isolated}, \texttt{ℓ-limit},
\texttt{isolatedness-decision}, \texttt{isolatedness-decision'}.}

In classical topology, a limit point is a non-isolated point. Assuming
excluded middle in our type theory, every point of any set is
isolated, and our ordinals are all sets, so that a point of our
ordinals is never non-isolated. In order to overcome this difficulty
in the absence of excluded middle but remaining compatible with
excluded middle, we define an element to be a limit point when its
isolatedness implies $\WLPO$. We choose this taboo because
$\neg \WLPO$ is a weak continuity principle, as discussed
in~\Cref{rem:weak-topological-topos}.

\begin{definition}[Topological limit point]
\label{def:limit-point}
A point $x$ of a type is a \emph{limit point} if its isolatedness
implies $\WLPO$:
\[
  \text{$x$ is a limit point} \;:\equiv\; (\text{$x$ is isolated} \to \WLPO) .
\]
\AgdaRefs{\Agda{TypeTopology.LimitPoints}{is-limit-point}.}
\end{definition}
For example, at the code $\eOmegaPlusOne$, whose
compact interpretation is $\NInfO$, the finite points are isolated
while the added point $\infty$ is a limit point, its isolatedness
being $\WLPO$
(\Cref{ex:isolated}(\labelcref{item:finite-isolated,item:infty-isolated-wlpo})).
\begin{definition}[Limit point indicator]
\label{def:limitfn}
The limit point indicator of a code~$\nu$ is the function
$\limitfn_\nu : \us{\Disc\nu} \to \Two$ defined by induction on $\nu$ by
\begin{align*}
\limitfn_{\eOne}(\pt) &= 0
&
\limitfn_{\nu_0\eAdd\nu_1}(\mathrm{inl}\,\pt,x_0) &= \limitfn_{\nu_0}(x_0)
\\
\limitfn_{\eOmegaPlusOne}(\mathrm{inl}\,n) &= 0
&
\limitfn_{\nu_0\eAdd\nu_1}(\mathrm{inr}\,\pt,x_1) &= \limitfn_{\nu_1}(x_1)
\\
\limitfn_{\eOmegaPlusOne}(\mathrm{inr}\,\pt) &= 1
&
\limitfn_{\nu_0\eMul\nu_1}(x_0,x_1) &= \max(\limitfn_{\nu_0}(x_0),\limitfn_{\nu_1}(x_1))
\\
& &
\limitfn_{\eSig\,\nu\,A}(x,y) &= \max(\limitfn_\nu(x),\limitfn_{A\,x}(y))
\end{align*}
\AgdaRefs{\Agda{Ordinals.InductiveRecursiveCodesInterpretations}{ℓ}.}
\end{definition}

\begin{theorem}
\label{thm:limitfn}
For every $\nu:\Ecodes$ and $x:\us{\Disc\nu}$,
\begin{enumerate}
\item\label{item:limitfn-isolated} if $\limitfn_\nu(x) = 0$ then
  $\iemb_\nu\,x$ is isolated in $\Comp\nu$,
\item\label{item:limitfn-limit} if $\limitfn_\nu(x) = 1$ then $\iemb_\nu\,x$
  is a limit point of $\Comp\nu$,
\item\label{item:limitfn-dichotomy} every point in the
  image of $\iemb_\nu$ is either isolated or a limit point,
\item\label{item:limitfn-decidable} if $\neg\WLPO$ holds,
  isolatedness of $\iemb_\nu\,x$ is decidable.
\end{enumerate}
\AgdaRefs{\Agda{Ordinals.InductiveRecursiveCodesInterpretations}{ℓ-isolated};
\texttt{ℓ-limit}; \texttt{isolatedness-decision};
\texttt{isolatedness-decision'}; \texttt{ℓ-limit} is strengthened
loc.\ cit.\ to \texttt{ℓ-limit⁺}, which concludes
\Agda{TypeTopology.LimitPoints}{is-limit-point⁺}, the implication from
weak isolatedness to $\WLPO$.}
\end{theorem}

\begin{proof}
(\labelcref{item:limitfn-isolated}) By induction on $\nu$, using
throughout that a pair is isolated as soon as both of its components
are
(\Cref{prop:discrete-closure}(\labelcref{item:discrete-sigma-isolated})).

At $\eOne$ the unique point is isolated, the type $\One$ being
discrete. At $\eOmegaPlusOne$ the hypothesis leaves
$x = \mathrm{inl}\,n$, whose image is the finite conatural
number~$\underline n$, which is isolated
(\Cref{ex:isolated}(\labelcref{item:finite-isolated})). At
$\eAdd$ the image is a pair whose index is $\mathrm{inl}\,\pt$ or
$\mathrm{inr}\,\pt$, isolated in $\One+\One$
(\Cref{prop:discrete-closure}(\labelcref{item:discrete-plus})), and
whose second component is isolated by the induction hypothesis, the
value of $\limitfn$ at the pair being its value at that component. At
$\eMul$ and $\eSig$ the value of $\limitfn$ is the maximum of the two
values of~\Cref{def:limitfn}, so the hypothesis makes both of them $0$,
and the induction hypothesis applies to both components.
This finishes the case $\eMul$, and in the case $\eSig$ the second
component of the image is $\varphi_x^{-1}(\iemb_{A\,x}\,y)$
(\Cref{def:E-kappa}), which is isolated because $\iemb_{A\,x}\,y$ is
and equivalences preserve isolatedness
(\Cref{prop:discrete-closure}(\labelcref{item:discrete-reflected})).
\AgdaRefs{\Agda{TypeTopology.SigmaDiscrete}{Σ-isolated};
\Agda{UF.DiscreteAndSeparated}{equivs-preserve-isolatedness}.}

(\labelcref{item:limitfn-limit}) By induction on $\nu$, assuming in
each case that $\iemb_\nu\,x$ is isolated and deriving $\WLPO$.

At $\eOne$ the claim holds vacuously, as the hypothesis
$\limitfn_\nu(x) = 1$ cannot hold, $\limitfn_{\eOne}$ being
constantly~$0$. At $\eOmegaPlusOne$ the hypothesis leaves
$x = \mathrm{inr}\,\pt$, whose image is $\infty$, and the isolatedness of $\infty$ is $\WLPO$
(\Cref{ex:isolated}(\labelcref{item:infty-isolated-wlpo})). The
remaining cases pass the assumed isolatedness down to the component
that $\limitfn$ flags, where the induction hypothesis converts it into
$\WLPO$.

At $\eAdd$ that component is the second one, and isolatedness of a pair
gives isolatedness of its second component when the index type is a
set
(\Cref{prop:discrete-closure}(\labelcref{item:discrete-sigma-isolated})),
which $\One+\One$ is. At $\eMul$ and $\eSig$ the value of $\limitfn$ is
again that maximum, so the hypothesis makes one of the two values $1$,
and the flagged component is the first or the second accordingly. At
$\eMul$ the sum is a product, whose components are both isolated with
no further hypothesis
(same item). At $\eSig$ the first component needs the members of the
extended family to be compact
(\Cref{lem:sigma-isolated}), which they are
by~\Cref{thm:micro-tychonoff}, being products indexed by propositions
of the ordinals $\Comp(A\,x)$, compact pointed
by~\Cref{thm:E-props}(\labelcref{item:E-comp-infima})
and~\Cref{rem:inf-gives-compact}, and the second component is isolated because $\us{\Comp\nu}$ is a set
(\Cref{lem:ext-gives-set}), which gives the isolatedness of
$\varphi_x^{-1}(\iemb_{A\,x}\,y)$ and hence of $\iemb_{A\,x}\,y$,
equivalences reflecting isolatedness
(\Cref{prop:discrete-closure}(\labelcref{item:discrete-reflected})).
\AgdaRefs{\Agda{TypeTopology.CompactTypes}{Σ-isolated-left};
\Agda{Ordinals.InductiveRecursiveCodesInterpretations}{𝓚-Compact};
\Agda{TypeTopology.SigmaDiscrete}{Σ-isolated-right},
\texttt{×-isolated-left}, \texttt{×-isolated-right};
\Agda{UF.DiscreteAndSeparated}{equivs-reflect-isolatedness};
\Agda{UF.DiscreteAndSeparated}{is-isolated-gives-is-isolated'}.}

(\labelcref{item:limitfn-dichotomy}) Decide the value of
$\limitfn_\nu(x)$, which is a point of $\Two$, and
apply~(\labelcref{item:limitfn-isolated})
or~(\labelcref{item:limitfn-limit}).

(\labelcref{item:limitfn-decidable})
By~(\labelcref{item:limitfn-dichotomy}) the point $\iemb_\nu\,x$ is
isolated, in which case isolatedness is decided, or a limit point, in
which case its isolatedness implies $\WLPO$, and contraposing
against $\neg\WLPO$ gives that it is not isolated.
\end{proof}

\begin{remark}
\label{rem:limit-points-do-not-agree}
A point of an ordinal is an \emph{order limit point} when it is
neither the least element nor a successor. This notion does not agree
with the topological notion of limit point (\Cref{def:limit-point}).
The point $(\infty,\underline 1)$ of the ordinal
$\Comp(\eOmegaPlusOne\eMul\eOmegaPlusOne) =
\osum{\NInfO}{}(\lambda\_.\,\NInfO)$ is the successor of
$(\infty,\underline 0)$, and hence is not an order limit point, but it
is a topological limit point, since its isolatedness gives $\WLPO$,
through its first coordinate $\infty$, a limit point of~$\NInfO$.
\AgdaRefs{\Agda{Ordinals.LimitPoints}{is-successor-of},
  \texttt{is-order-limit-point},
  \texttt{example-of-topological-limit-point-which-is-not-order-limit}.}
\end{remark}

\section{Discussion}
\label{sec:discussion}

\subsection{Summary of results}
\label{sec:conclusion}

\begin{table}[htbp]
\caption{The interpretations of the two notation systems.}
\label{tab:interpretations}
\footnotesize
\setlength{\tabcolsep}{3pt}
\begin{tabular}{|L{2.5cm}|L{1.7cm}|L{3.8cm}|L{3.7cm}|}
\toprule
Interpretation & Reading of the limit constructor &
\holds{Holds} for every code & \failsc{Fails} constructively \\
\midrule
\multicolumn{4}{|l|}{Brouwer codes~$b$ (\Cref{def:brouwer-codes}),
  limit constructor~$\bL$} \\
\midrule
standard $\stdint b$\newline in $\Ord[]$\newline (\Cref{def:std-interp})
  & supremum
  & nothing beyond being an ordinal
  & \failsc{trichotomy}, \failsc{compactness} and the
    \failsc{least element property}, each giving $\LPO$
    (\Cref{prop:failure-trichotomy,prop:std-compact-lpo,prop:std-trich-least-lpo}),
    and \failsc{discreteness}, giving $\WLPO$
    (\Cref{prop:failure-trichotomy}) \\
\midrule
$\compactsepint b$\newline in $\OrdT[]$\newline
(\Cref{def:non-standard-interpretations})
  & extended sum
  & \holds{compactness} and \holds{total separatedness}
    (\Cref{thm:four-interp-props}), being a
    \holds{retract of $\Cantor$} (\Cref{thm:four-interp-props}), and
    \holds{infima of complemented subsets} (\Cref{thm:kappa-infs}),
    hence the \holds{least element property} (\Cref{cor:kappa-least})
  & \failsc{discreteness}, giving $\WLPO$
    (\Cref{prop:delta-kappa-fail}),
    and hence \failsc{trichotomy} \\
\midrule
$\compactint b$\newline in $\Ord[]$\newline (\Cref{def:non-standard-interpretations})
  & supremum of the extension
  & \holds{compactness} (\Cref{thm:four-interp-props})
  & \failsc{total separatedness}, giving $\neg\neg\WLPO$
    (\Cref{prop:failure-tot-sep}) \\
\midrule
$\trichint b$\newline in $\OrdThree[]$\newline (\Cref{def:non-standard-interpretations})
  & sum over $\omegaO$
  & \holds{trichotomy} (\Cref{thm:four-interp-props}), hence
    \holds{discreteness} (\Cref{lem:trichotomy-gives-discrete})
  & \failsc{compactness} and the \failsc{least element property},
    each giving $\LPO$
    (\Cref{prop:trich-compact-lpo,prop:std-trich-least-lpo}) \\
\midrule
$\alttrichint b$\newline in $\OrdT[]$\newline (\Cref{def:delta-kappa})
  & successor sum
  & \holds{trichotomy} and being a \holds{retract of $\N$}, hence
    \holds{discreteness} (\Cref{thm:delta-kappa-props})
  & \failsc{compactness} (\Cref{prop:delta-kappa-fail}) and the
    \failsc{least element property} (\Cref{prop:delta-least-lpo}), each
    equivalent to $\LPO$ \\
\midrule
\multicolumn{4}{|l|}{Inductive-recursive codes~$\nu$ (\Cref{def:E}),
  sum constructor~$\eSig$} \\
\midrule
$\Disc\nu$\newline in $\OrdT[]$\newline (\Cref{def:E})
  & sum over $\Disc\nu$
  & \holds{trichotomy} and being a \holds{retract of $\N$}, hence
    \holds{discreteness} (\Cref{thm:E-delta})
  & \failsc{compactness} and the \failsc{least element property}, each
    equivalent to $\LPO$ (\Cref{prop:E-delta-lpo,prop:E-delta-least-lpo}) \\
\midrule
$\Comp\nu$\newline in $\OrdT[]$\newline (\Cref{def:E-kappa})
  & sum over $\Comp\nu$ of the extension
  & \holds{compactness}, with \holds{infima of complemented subsets}
    (\Cref{thm:E-props}), hence the \holds{least element property}
    (\Cref{cor:E-least})
  & \failsc{total separatedness}, giving $\neg\neg\WLPO$
    (\Cref{prop:kappa-not-tot-sep}), and \failsc{discreteness}, giving
    $\WLPO$
    (\Cref{prop:E-iota-equiv-discrete}),
    and hence \failsc{trichotomy} \\
\bottomrule
\end{tabular}
\end{table}

\begin{table}[htbp]
\caption{The embedding of the discrete into the compact
interpretation.}
\label{tab:iota}
\footnotesize
\setlength{\tabcolsep}{4pt}
\begin{tabular}{|L{3.6cm}|L{3.9cm}|L{3.9cm}|}
\toprule
Embedding & \holds{Holds} for every code & \failsc{Fails} constructively \\
\midrule
$\iemb_b : \us{\alttrichint b} \hookrightarrow \us{\compactsepint b}$\newline
for Brouwer codes~$b$\newline (\Cref{def:iota})
  & \holds{extreme density}, \holds{order preservation} and
    \holds{order reflection} (\Cref{thm:iota})
  & \failsc{surjectivity}, equivalent to $\LPO$ (\Cref{prop:brouwer-iota-equiv-LPO}) \\
\midrule
$\iemb_\nu : \us{\Disc\nu} \hookrightarrow \us{\Comp\nu}$\newline
for $\Ecodes$-codes~$\nu$\newline (\Cref{def:E-kappa})
  & \holds{extreme density}, \holds{order preservation} and
    \holds{order reflection} (\Cref{thm:E-props})
  & \failsc{surjectivity}, equivalent to $\LPO$ (\Cref{prop:E-iota-equiv-discrete}) \\
\bottomrule
\end{tabular}
\end{table}

We constructed infinite types that are exhaustively searchable, and
called them \emph{compact} because in many respects they behave like the
topological notion of the same name (\Cref{sec:compact-types-theory}).
We then noticed that the constructions carry natural well-orders and
preserve them, and the ordinals thus obtained measure the logical
complexity of the compact types we produced. We gave two notation
systems for them, each with several interpretations as ordinals.

The interpretations differ mainly in the reading of the limit
constructor, and they are collected in~\Cref{tab:interpretations},
which records for each of them what holds and what fails
constructively. Total
separatedness, one of the properties listed there, is a boolean
Leibniz principle
corresponding to the topological notion of the same name
(\Cref{sec:totally-separated-types}). Two of the interpretations are
related by an embedding of the discrete one into the compact one
(\Cref{def:iota,thm:iota}), whose properties are collected
in~\Cref{tab:iota}. That embedding
needs no assumption, whereas the comparisons among the other
interpretations need excluded middle (\Cref{thm:comparisons}), which
for every Brouwer code~$b$ gives
\[
\begin{array}{ccc}
\stdint{b} & \oemb & \trichint{b} \\[1ex]
\oembdown & & \oembdown \\[1ex]
\compactint{b} & \oemb & \compactsepint{b}\rlap{$.$}
\end{array}
\]

The inductive-recursive universe $\Ecodes$ carries a discrete and a
compact interpretation, $\Disc$ and $\Comp$, related by the embedding
$\iemb$, as the Brouwer
codes do, but its sum constructor is indexed by any previously
constructed discrete ordinal rather than by $\N$
(\Cref{def:E,def:E-kappa}). This is possible because the codes and
their discrete interpretation are defined simultaneously, so that a
code may be indexed by the ordinal denoted by an earlier code. The
Brouwer codes are included in it, and the inclusion preserves both
interpretations up to order equivalence (\Cref{thm:brouwer-to-E}).
Total separatedness of the compact interpretation $\Comp$ for every
code of $\Ecodes$ gives $\neg\neg\WLPO$
(\Cref{prop:kappa-not-tot-sep}), whereas the compact interpretation
$\compactsepint{-}$ of every Brouwer code is totally separated.
The ordinals witnessing the failure are not in general denoted by
Brouwer codes.

In the compact ordinals of both systems, every non-empty complemented
subset has a least element~(\Cref{cor:kappa-least,cor:E-least}).
The isolated points and the limit points in the image of the embedding
are distinguished by an explicit boolean valued function on the
discrete side (\Cref{def:limitfn,thm:limitfn}), and the embedding is
an equivalence for every code precisely when $\LPO$ holds
(\Cref{prop:brouwer-iota-equiv-LPO}
and~\Cref{prop:E-iota-equiv-discrete}(\labelcref{item:E-iota-equiv})).

The above constructions rest on the closure properties of compactness
(\Cref{sec:compact-closure}), the micro-Tychonoff theorem for families
indexed by a proposition (\Cref{thm:micro-tychonoff}), the extension of
a family along an embedding and the sums built from it
(\Cref{sec:extension,sec:extended-sums}), and suprema of families of
ordinals (\Cref{sec:extended-suprema}).

\subsection{Reaching high compact ordinals}

In~\cite[\S11]{EscardoOmniscientJSL2013} we reach compact totally
separated ordinals of every rank below $\eps$ in G\"odel's~$T$, by
iterating squashed sums in the Cantor space. At a Dagstuhl seminar in
October 2011 we conjectured that this is the best we can do, and Normann
proved it~\cite{Normann2016}: if a closed subset of the Baire space has
a search functional definable in~$T$, then it is countable and its
Cantor--Bendixson rank lies below $\eps$.

Coquand, Hancock and Setzer~\cite{CoquandHancockSetzer1997} add the type
$\Brouwer$ of~\Cref{def:brouwer-codes} to~$T$ as a base type, with its
constructors but without its elimination principle, and ask which
ordinals are denoted by the closed terms of type $\Brouwer$. They call a
point of a type $X$ together with an endomap of $X$ and a map
$(\N \to X) \to X$ a \emph{limit structure}, of which $\Brouwer$ with
$\bZ$, $\bS$ and $\bL$ is an example. Our interpretations are
limit structures in this sense, their main difference being in the
third component, which is the supremum for $\stdint{-}$
(\Cref{def:std-interp}), the extended sum for $\compactsepint{-}$ and
the successor sum for $\alttrichint{-}$ (\Cref{def:delta-kappa}).

In system~$T$ they build closed terms denoting $\omega$,
$\omega^{\omega}$, $\omega^{\omega^{\omega}}$ and so on, so that the
ordinal of the system is at least $\eps$, using a device of Gentzen
that turns a limit structure on $X$ into one on $X \to X$. In the
extension of~$T$ by one universe of small types, they also obtain
$\eps$, using a device of Poincar\'e that replaces an impredicative
intersection by a countable one, and they report having all the pieces
for showing that every ordinal below $\varphi_{\eps}0$ is denoted. This
is the easy half of a conjecture of Hancock that one universe gives
exactly those ordinals. In its proof-theoretic form, for any finite
number of universes, the conjecture was proved by Feferman, and the
case of one universe by Aczel (see e.g.~\cite{Palmgren1998}).
\AgdaRefs{\Agda{Ordinals.ChurchEncoding}{B-ε₀};
\Agda{Ordinals.BrouwerCodesDiscreteAndCompactInterpretations}{Κ-of-ε₀-code-is-compact∙}.}

Our system $\Ecodes$ of inductive-recursive codes
(\Cref{sec:inductive-recursive}) extends the Brouwer codes, and
in general it denotes strictly more compact ordinals
(\Cref{rem:kappa-tot-sep-contrast}).

In a different but related direction, Cherubini, Coquand, Geerligs and
Moeneclaey~\cite{CherubiniCoquandGeerligsMoeneclaey2025} extend
homotopy type theory with four axioms, among them a form of Stone
duality, and expect these axioms to hold in the higher topos of light
condensed anima. In their theory, a compact Hausdorff space is a Stone
space precisely when it is totally separated in the sense
of~\Cref{def:tot-sep}~\cite[Lemma~4.17]{CherubiniCoquandGeerligsMoeneclaey2025}.
Their compact Hausdorff spaces are compact in the sense of synthetic
topology for the dominance of open propositions, which are the
countable disjunctions of decidable propositions, whereas our compact
types are so for the dominance~$\Two$ (\Cref{sec:weakly-compact}).
Their axioms prove Markov's principle and the negation of $\WLPO$,
among other principles. On the other hand, in this paper we work
without such topological axioms and investigate the intrinsic topology
of types, with results consistent with excluded middle and choice.

\subsection{Formalization}

Every result above was formalized in Agda~\cite{Agda}, before this paper was
written, in the \textsf{TypeTopology} repository~\cite{TypeTopology},
checked with the \texttt{-{}-safe} and \texttt{-{}-without-K} flags, and
hence with no postulates, and with the necessary HoTT/UF principles as
explicit assumptions of the results that need them. The formalization
is spread over the \textsf{TypeTopology} and \textsf{Ordinals}
directories in the repository, with auxiliary results in many other
directories.
For easy reference, we produced a companion Agda
file~\cite{EscardoCompactCompanion} for this paper, which formalizes
every definition, lemma, proposition, theorem, example and remark, in
the same order and under the same numbering, together with the
mathematical claims in running prose outside proofs. Its entries name
existing code.

\section*{Acknowledgements}

Mike Shulman supplied, in personal communication, the counterexample
of~\Cref{rem:lex-not-extensional}, that extensionality of lexicographic
sums of ordinals implies excluded middle. Tom de Jong constructed
suprema of families of ordinals by set quotients
(\Cref{sec:extended-suprema}), and formalized both his construction and
ours~\cite{deJongThesis2023}. He also proved and formalized that
requiring every discrete ordinal to be trichotomous implies excluded
middle~(\Cref{rem:trichotomy-EM}).

\makeatletter
\let\old@contentsline\contentsline
\renewcommand{\contentsline}[4]{%
  \def\@tempa{#1}%
  \def\@tempb{subsection}%
  \ifx\@tempa\@tempb
    \old@contentsline{#1}%
      {\hyperlink{#4}{\textcolor{teal!70!black}{#2}}}%
      {\hyperlink{#4}{\textcolor{teal!70!black}{#3}}}%
      {}%
  \else
    \old@contentsline{#1}{#2}{#3}{#4}%
  \fi
}
\makeatother
\setcounter{tocdepth}{2}
\clearpage
\phantomsection
\label{toc:full}
\markboth{Full table of contents}{Full table of contents}
\begingroup
\renewcommand{\contentsname}{Full table of contents}
\makeatletter
\@starttoc{toc}\contentsname
\makeatother
\endgroup


\begin{thebibliography}{10}

\bibitem{Beeson1985}
Michael~J. Beeson.
\newblock {\em Foundations of Constructive Mathematics: Metamathematical
  Studies}.
\newblock Ergebnisse der Mathematik und ihrer Grenzgebiete, 3. Folge. Springer,
  Berlin, Heidelberg, 1985.

\bibitem{Bishop1967}
Errett Bishop.
\newblock {\em Foundations of Constructive Analysis}.
\newblock McGraw-Hill, 1967.

\bibitem{bishop_bridges_1985}
Errett Bishop and Douglas Bridges.
\newblock {\em Constructive Analysis}, volume 279 of {\em Grundlehren der
  mathematischen Wissenschaften}.
\newblock Springer-Verlag, Berlin, Heidelberg, New York, Tokyo, 1985.

\bibitem{BridgesVita2006}
Douglas Bridges and Lumini{\c{t}}a~Simona V{\^\i}{\c{t}}{\u{a}}.
\newblock {\em Techniques of Constructive Analysis}.
\newblock Universitext. Springer, 2006.

\bibitem{Capretta_2005}
Venanzio Capretta.
\newblock General recursion via coinductive types.
\newblock {\em Logical Methods in Computer Science}, Volume 1, Issue 2, July
  2005.

\bibitem{CherubiniCoquandGeerligsMoeneclaey2025}
Felix Cherubini, Thierry Coquand, Freek Geerligs, and Hugo Moeneclaey.
\newblock A foundation for synthetic {S}tone duality.
\newblock In Rasmus~Ejlers M{\o}gelberg and Benno van~den Berg, editors, {\em
  30th International Conference on Types for Proofs and Programs (TYPES 2024)},
  volume 336 of {\em LIPIcs}, pages 3:1--3:20. Schloss Dagstuhl --
  Leibniz-Zentrum f\"ur Informatik, 2025.

\bibitem{CoquandHancockSetzer1997}
Thierry Coquand, Peter Hancock, and Anton Setzer.
\newblock Ordinals in type theory.
\newblock Invited talk at Computer Science Logic (CSL), Aarhus, 1997.
\newblock \url{https://www.cse.chalmers.se/~coquand/ordinal.pdf}, also archived
  as
  \url{https://web.archive.org/web/20251203190938/https://www.cse.chalmers.se/~coquand/ordinal.ps}.

\bibitem{CoquandMannaa2017}
Thierry Coquand and Bassel Mannaa.
\newblock The independence of {M}arkov's principle in type theory.
\newblock {\em Logical Methods in Computer Science}, 13(3), 2017.

\bibitem{deJongThesis2023}
Tom de~Jong.
\newblock {\em Domain theory in constructive and predicative univalent
  foundations}.
\newblock PhD thesis, University of Birmingham, 2023.
\newblock \url{https://etheses.bham.ac.uk/id/eprint/13401/}, also available as
  \href{https://arxiv.org/abs/2301.12405}{arXiv:2301.12405}.

\bibitem{DeJong:Escardo:2026}
Tom de~Jong and Mart\'in~H\"otzel Escard\'o.
\newblock Examples and counterexamples of injective types.
\newblock {\em Annals of Pure and Applied Logic}, 178(2), 2027.
\newblock Available online 8 September 2026.

\bibitem{Dybjer2000}
Peter Dybjer.
\newblock A general formulation of simultaneous inductive-recursive definitions
  in type theory.
\newblock {\em The Journal of Symbolic Logic}, 65(2):525--549, 2000.

\bibitem{DybjerSetzer1999}
Peter Dybjer and Anton Setzer.
\newblock A finite axiomatization of inductive-recursive definitions.
\newblock In Jean-Yves Girard, editor, {\em Typed Lambda Calculi and
  Applications ({TLCA} 1999)}, volume 1581 of {\em Lecture Notes in Computer
  Science}, pages 129--146. Springer, 1999.

\bibitem{dybjersetzer:IndexedInductionRecursion:2006}
Peter Dybjer and Anton Setzer.
\newblock Indexed induction-recursion.
\newblock {\em Journal of Logic and Algebraic Programming}, 66:1--49, 2006.

\bibitem{EscardoSynthetic2004}
Mart\'in~H\"otzel Escard\'o.
\newblock Synthetic topology of data types and classical spaces.
\newblock {\em Electronic Notes in Theoretical Computer Science}, 87:21--156,
  2004.

\bibitem{EscardoFastSearch2007}
Mart\'in~H\"otzel Escard\'o.
\newblock Infinite sets that admit fast exhaustive search.
\newblock In {\em 22nd Annual {IEEE} Symposium on Logic in Computer Science
  ({LICS} 2007)}, pages 443--452. {IEEE}, 2007.

\bibitem{EscardoExhaustible2008}
Mart{\'i}n~H{\"o}tzel Escard{\'o}.
\newblock Exhaustible sets in higher-type computation.
\newblock {\em Logical Methods in Computer Science}, 4(3):3:3, 37, 2008.

\bibitem{EscardoOmniscientJSL2013}
Mart{\'i}n~H{\"o}tzel Escard{\'o}.
\newblock Infinite sets that satisfy the principle of omniscience in any
  variety of constructive mathematics.
\newblock {\em The Journal of Symbolic Logic}, 78(3):764--784, 2013.
\newblock Revised version:
  \url{https://martinescardo.github.io/papers/omniscient-journal-revised.pdf}.

\bibitem{EscardoContinuity2015}
Mart\'in~H\"otzel Escard\'o.
\newblock Constructive decidability of classical continuity.
\newblock {\em Mathematical Structures in Computer Science}, 25(7):1578--1589,
  2015.

\bibitem{EscardoTypes2019Abstract}
Mart{\'i}n~H{\"o}tzel Escard{\'o}.
\newblock Compact, totally separated and well-ordered types in univalent
  mathematics.
\newblock In {\em Types for Proofs and Programs ({TYPES} 2019)}, 2019.
\newblock 4-page abstract,
  \url{https://martinescardo.github.io/papers/compact-ordinals-Types-2019-abstract.pdf}.

\bibitem{EscardoInjective2021}
Mart{\'i}n~H{\"o}tzel Escard{\'o}.
\newblock Injective types in univalent mathematics.
\newblock {\em Mathematical Structures in Computer Science}, 2021.

\bibitem{EscardoCompactCompanion}
Mart{\'i}n~H{\"o}tzel Escard{\'o}.
\newblock {TypeTopology.CompactTotallySeparatedTypesArticle}: an {A}gda
  companion to the preprint \emph{{C}ompact totally separated types}.
\newblock
  \url{https://www.cs.bham.ac.uk/~mhe/TypeTopology/TypeTopology.CompactTotallySeparatedTypesArticle.html},
  2026.
\newblock Part of the {TypeTopology} repository.

\bibitem{TypeTopology}
Mart{\'i}n~H{\"o}tzel Escard{\'o} et~al.
\newblock {TypeTopology}: {Agda} development of univalent mathematics, mainly
  {T}ype {T}opology.
\newblock \url{https://github.com/martinescardo/TypeTopology}.
\newblock Agda repository, continuously developed since 2011.

\bibitem{EscardoKnapp2017}
Mart\'in~H\"otzel Escard\'o and Cory~M. Knapp.
\newblock Partial elements and recursion via dominances in univalent type
  theory.
\newblock In Valentin Goranko and Mads Dam, editors, {\em 26th {EACSL} Annual
  Conference on Computer Science Logic ({CSL} 2017)}, volume~82 of {\em
  LIPIcs}, pages 21:1--21:16. Schloss Dagstuhl -- Leibniz-Zentrum f\"ur
  Informatik, 2017.

\bibitem{EscardoOliva2010}
Mart{\'i}n~H{\"o}tzel Escard{\'o} and Paulo Oliva.
\newblock Selection functions, bar recursion and backward induction.
\newblock {\em Mathematical Structures in Computer Science}, 20(2):127--168,
  2010.

\bibitem{EscardoStreicherUniverses2016}
Mart{\'i}n~H{\"o}tzel Escard{\'o} and Thomas Streicher.
\newblock The intrinsic topology of {M}artin-{L}\"of universes.
\newblock {\em Annals of Pure and Applied Logic}, 167(9):794--805, 2016.
\newblock Fourth Workshop on Formal Topology (4WFTop).

\bibitem{Hedberg1998}
Michael Hedberg.
\newblock A coherence theorem for {Martin-L\"of's} type theory.
\newblock {\em Journal of Functional Programming}, 8(4):413--436, 1998.

\bibitem{Hofmann1995Extensional}
Martin Hofmann.
\newblock Extensional concepts in intensional type theory.
\newblock Technical report ECS-LFCS-95-327, Laboratory for Foundations of
  Computer Science, University of Edinburgh, 1995.
\newblock PhD thesis, freely available from
  \url{https://www.lfcs.inf.ed.ac.uk/reports/95/ECS-LFCS-95-327/}.

\bibitem{Hofmann1997Extensional}
Martin Hofmann.
\newblock {\em Extensional Constructs in Intensional Type Theory}.
\newblock Distinguished Dissertations. Springer, 1997.

\bibitem{Hyland1982}
J.~M.~E. Hyland.
\newblock The effective topos.
\newblock In A.~S. Troelstra and D.~van Dalen, editors, {\em The L. E. J.
  Brouwer Centenary Symposium}, volume 110 of {\em Studies in Logic and the
  Foundations of Mathematics}, pages 165--216. North-Holland, 1982.

\bibitem{Johnstone1979}
Peter~T. Johnstone.
\newblock On a topological topos.
\newblock {\em Proceedings of the London Mathematical Society}, 38(2):237--271,
  1979.

\bibitem{johnstone:stonespaces}
P.T. Johnstone.
\newblock {\em Stone Spaces}.
\newblock Cambridge University Press, Cambridge, 1982.

\bibitem{Kleene1959}
Stephen~C. Kleene.
\newblock Countable functionals.
\newblock {\em Constructivity in Mathematics}, pages 81--100, 1959.

\bibitem{KleeneVesley1965}
Stephen~C. Kleene and Richard~E. Vesley.
\newblock {\em The Foundations of Intuitionistic Mathematics, Especially in
  Relation to Recursive Functions}.
\newblock Studies in Logic and the Foundations of Mathematics. North-Holland,
  Amsterdam, 1965.

\bibitem{KrausEscardoCoquandAltenkirch2013}
Nicolai Kraus, Mart{\'i}n~H{\"o}tzel Escard{\'o}, Thierry Coquand, and Thorsten
  Altenkirch.
\newblock Generalizations of {H}edberg's theorem.
\newblock In {\em Typed Lambda Calculi and Applications ({TLCA} 2013)}, volume
  7941 of {\em Lecture Notes in Computer Science}, pages 173--188. Springer,
  2013.

\bibitem{KrausEscardoCoquandAltenkirch2017}
Nicolai Kraus, Mart{\'i}n~H{\"o}tzel Escard{\'o}, Thierry Coquand, and Thorsten
  Altenkirch.
\newblock Generalizations of {H}edberg's theorem.
\newblock {\em Logical Methods in Computer Science}, 13(1):15:1--15:15, 2017.

\bibitem{ForsbergKrausXu2021}
Nicolai Kraus, Fredrik~Nordvall Forsberg, and Chuangjie Xu.
\newblock Connecting constructive notions of ordinals in homotopy type theory.
\newblock In {\em 46th International Symposium on Mathematical Foundations of
  Computer Science ({MFCS} 2021)}, LIPIcs, 2021.

\bibitem{LongleyNormann2015}
John Longley and Dag Normann.
\newblock {\em Higher-Order Computability}.
\newblock Theory and Applications of Computability. Springer, 2015.

\bibitem{Agda}
Ulf Norell, Nils~Anders Danielsson, Jesper Cockx, Andreas Abel, et~al.
\newblock Agda.
\newblock \url{https://wiki.portal.chalmers.se/agda/pmwiki.php}.

\bibitem{Normann2016}
Dag Normann.
\newblock On the {C}antor--{B}endixson rank of a set that is searchable in
  {G}\"odel's {T}.
\newblock {\em Computability}, 5(1):61--74, 2016.
\newblock Accepted version at \url{https://hdl.handle.net/11250/4673359}.

\bibitem{NormannTait2017}
Dag Normann and William Tait.
\newblock On the computability of the fan functional.
\newblock In {\em Feferman on Foundations}, Outstanding Contributions to Logic.
  Springer, 2017.

\bibitem{Palmgren1998}
Erik Palmgren.
\newblock On universes in type theory.
\newblock In Giovanni Sambin and Jan~M. Smith, editors, {\em Twenty-Five Years
  of Constructive Type Theory}, volume~36 of {\em Oxford Logic Guides}, pages
  191--204. Oxford University Press, 1998.

\bibitem{pitts2026setoidsintensionaltypetheory}
Andrew~M. Pitts.
\newblock Setoids in intensional type theory.
\newblock \url{https://arxiv.org/abs/2607.23671}, 2026.

\bibitem{PradicBrown2019}
C{\'e}cilia Pradic and Chad~E. Brown.
\newblock Cantor-{B}ernstein implies excluded middle, 2019.
\newblock \url{https://arxiv.org/abs/1904.09193}.

\bibitem{HoTTBook}
The {Univalent}~{Foundations} {Program}.
\newblock {\em Homotopy Type Theory: Univalent Foundations of Mathematics}.
\newblock \url{https://homotopytypetheory.org/book}, Institute for Advanced
  Study, 2013.

\bibitem{rijke2017joinconstruction}
Egbert Rijke.
\newblock The join construction, 2017.

\bibitem{rijke_2025}
Egbert Rijke.
\newblock {\em Introduction to Homotopy Type Theory}.
\newblock Cambridge Studies in Advanced Mathematics. Cambridge University
  Press, 2025.

\bibitem{shulman2019infty1toposesstrictunivalentuniverses}
Michael Shulman.
\newblock All $(\infty,1)$-toposes have strict univalent universes.
\newblock \url{https://arxiv.org/abs/1904.07004}, 2019.

\bibitem{Smyth1983}
Michael~B. Smyth.
\newblock Power domains and predicate transformers: A topological view.
\newblock In Josep D{\'\i}az, editor, {\em Automata, Languages and Programming,
  10th Colloquium, Barcelona, Spain, July 18--22, 1983, Proceedings}, volume
  154 of {\em Lecture Notes in Computer Science}, pages 662--675. Springer,
  1983.

\bibitem{Smyth1992}
Michael~B. Smyth.
\newblock Topology.
\newblock In S.~Abramsky, Dov~M. Gabbay, and T.~S.~E. Maibaum, editors, {\em
  Handbook of Logic in Computer Science, Volume 1: Background: Mathematical
  Structures}, pages 641--761. Oxford University Press, 1992.

\bibitem{Troelstra1973}
A.~S. Troelstra, editor.
\newblock {\em Metamathematical Investigation of Intuitionistic Arithmetic and
  Analysis}, volume 344 of {\em Lecture Notes in Mathematics}.
\newblock Springer, 1973.

\bibitem{vanOosten2008}
Jaap van Oosten.
\newblock {\em Realizability: An Introduction to its Categorical Side}, volume
  152 of {\em Studies in Logic and the Foundations of Mathematics}.
\newblock Elsevier, Amsterdam, 2008.

\bibitem{Voevodsky2015}
Vladimir Voevodsky.
\newblock An experimental library of formalized mathematics based on the
  univalent foundations.
\newblock {\em Mathematical Structures in Computer Science}, 25(5):1278--1294,
  2015.

\end{thebibliography}
\end{document}